\documentclass[reqno, a4paper, 10pt]{amsart}
\usepackage{amssymb,url, color, mathrsfs}
\usepackage[colorlinks=true, bookmarks=true, pdfstartview=FitH, pagebackref=true, linktocpage=true, linkcolor = magenta, citecolor = blue]{hyperref}
\usepackage[short,nodayofweek]{datetime}
\usepackage[space, compress, sort]{cite}
\usepackage{xcolor}
\usepackage{times}

\theoremstyle{plain}
\numberwithin{equation}{section}
\newtheorem{theorem}{Theorem}[section]
\newtheorem{lemma}[theorem]{Lemma}

\newtheorem{proposition}[theorem]{Proposition}
\theoremstyle{remark}

\DeclareMathOperator{\Rset}{\mathbf{R}}

\definecolor{brown}{rgb}{0.5,0,0}
\definecolor{backgroundcolor}{rgb}{0.98, 0.92, 0.73}

\def\R{{\mathbf R}}
\def\N{\mathbb N}
\def\H{\mathbb H}

\def\sign{\mathop{\rm sign}\nolimits}
 
\definecolor{ORCIDLOGOCOL}{HTML}{A6CE39}

\def\cfac#1{\ifmmode\setbox7\hbox{$\accent"5E#1$}\else\setbox7\hbox{\accent"5E#1}\penalty 10000\relax\fi\raise 1\ht7\hbox{\lower1.05ex\hbox to 1\wd7{\hss\accent"13\hss}}\penalty 10000\hskip-1\wd7\penalty 10000\box7 }

\author[Q.A. Ng\^o]{Qu\cfac{o}c Anh Ng\^o}

\address[Q.A. Ng\^o]{Institute of Research and Development\\
Duy T\^an University\\
D\`a N\v ang, Vi\^et Nam.}

\address[Q.A. Ng\^o]{Department of Mathematics\\
College of Science, Vi\^{e}t Nam National University\\
H\`{a} N\^{o}i, Vi\^{e}t Nam.}
\email{\href{mailto: Q.A. Ng\^o <nqanh@vnu.edu.vn>}{nqanh@vnu.edu.vn}}
\email{\href{mailto: Q.A. Ng\^o <bookworm\_vn@yahoo.com>}{bookworm\_vn@yahoo.com}}

\author[V.H. Nguyen]{Van Hoang Nguyen}

\address[V.H. Nguyen]{Institut de Math\'ematiques de Toulouse\\
Universit\'e Paul Sabatier\\
31062 Toulouse c\'edex 09, France.}

\address[V.H. Nguyen]{Institute of Mathematics\\
Vietnam Academy of Science and Technology\\
Hanoi, Vietnam.}

\email{\href{mailto: V.H. Nguyen <vanhoang0610@yahoo.com>}{vanhoang0610@yahoo.com}}
\email{\href{mailto: V.H. Nguyen <nvhoang@math.ac.vn>}{nvhoang@math.ac.vn}}

\begin{document}

\allowdisplaybreaks

\title[Sharp Adams--Moser--Trudinger type inequalities in the hyperbolic space]
{Sharp Adams--Moser--Trudinger type inequalities in the hyperbolic space}

\begin{abstract} 
The purpose of this paper is to establish some Adams--Moser--Trudinger inequalities, which are the borderline cases of the Sobolev embedding, in the hyperbolic space $\H^n$. First, we prove a sharp Adams' inequality of order two with the exact growth condition in $\H^n$. Then we use it to derive a sharp Adams-type inequality and an Adachi--Tanaka-type inequality. We also prove a sharp Adams-type inequality with Navier boundary condition on any bounded domain of $\H^n$, which generalizes the result of Tarsi to the setting of hyperbolic spaces. Finally, we establish a Lions-type lemma and an improved Adams-type inequality in the spirit of Lions in $\H^n$. Our proofs rely on the symmetrization method extended to hyperbolic spaces.
\end{abstract}

\date{\bf \today \; at \, \currenttime}



\maketitle


\section{Introduction}

Sobolev spaces, geometric and analytic inequalities can be considered as one of the central tools in many areas such as analysis, differential geometry, partial differential equations, calculus of variations, etc. Of importance, among these inequalities, are the classical Sobolev inequalities which assert that the following embedding $W_0^{k,p}(\Omega) \hookrightarrow L^q(\Omega)$ is continuous for $n\geqslant 2$, $kp< n$, and $1 \leqslant q \leqslant np/(n-kp)$ where $\Omega$ is a bounded domain in $\R^n$. However, in the limiting case $kp =n$, we can easily show by many examples that $W_0^{k,n/k}(\Omega)\not\subset L^\infty(\Omega)$. In this special situation, the so-called Moser--Trudinger inequality and its higher order version, known as Adams' inequality, are the perfect replacements; see \cite{T1967, M1970, Adams1988}. 

It is now widely recognized that the Moser--Trudinger and Adams inequalities have played so many important roles and have been widely used in geometric analysis and PDE; for example, we refer the reader to \cite{CY2003, LamLu2012a, LamLu2012b, LamLu2014, Shaw1987, TZ2000} and references therein. 

These remarkable inequalities have also been generalized in many directions. For instance, the singular Moser--Trudinger inequality was discovered in \cite{AS2007}, the best constant for the Moser--Trudinger inequality on domains of finite measure on the Heisenberg group was found in \cite{CohnLu2001,LamLuTang2012}. There has also been substantial progress for the Moser--Trudinger inequality on the Euclidean spheres, on the CR spheres, as well as on any compact Riemannian manifold and on hyperbolic spaces; see \cite{Bec1993,BFM2013, CohnLu2001,CohnLu2004,Fontana1993,Li2005,LT2013}. 

For the question of the existence of optimal functions for the Moser--Trudinger inequality, it was first addressed by Carleson and Chang \cite{CC1986} on the  Euclidean balls. Then, this result was extended to arbitrary smooth domains by Flucher \cite{Flucher1992} and Lin \cite{Lin1996}. 

\subsection{Moser--Trudinger and Adams inequalities on $\R^n$}

\subsubsection{Moser--Trudinger inequalities on $\R^n$}

Speaking of Moser--Trudinger' inequality on bounded domains, it was established independently by Yudovi\v c \cite{Y1961}, Poho\v zaev \cite{P1965}, and Trudinger \cite{T1967}. Later, by sharpening Trudinger's inequality, Moser proved that there exists a sharp constant $\alpha_n > 0$ such that
\begin{equation}\label{eq:Moser--Trudinger}
\sup\limits_{\begin{subarray}{c} 
u \in C_0^\infty(\Omega):\\
 \int_\Omega |\nabla u|^n dx \leqslant 1
\end{subarray}}
 \int_\Omega \exp\big(\alpha |u|^{n/(n-1)} \big) dx < +\infty 
\tag{MT${}^\R_b$}
\end{equation}
for any $\alpha\leqslant \alpha_n$ and for any bounded domain $\Omega$ in $\R^n$. Furthermore, the constant $\alpha_n$ in \eqref{eq:Moser--Trudinger} is sharp in the sense that if $\alpha > \alpha_n$, then the supremum above will become infinity. Moser was able to compute the sharp constant $\alpha_n$ precisely, that is
\[
\alpha_n = n^{n/(n-1)} \Omega_n^{1/(n-1)},
\] 
where $\Omega_n$ denotes the volume of the unit ball $\mathbb B^n$ in $\R^n$. If we denote by $\omega_n$ the volume of the unit sphere $\mathbb S^n$ in $\R^{n+1}$, then
\[
\alpha_n = n \omega_{n-1}^{1/(n-1)}.
\]
When $\Omega$ has \textit{infinite measure}, the sharp version of Moser--Trudinger-type inequality for unbounded domains, or a ``subcritical'' Moser--Trudinger inequality, was established by Adachi and Tanaka in \cite{AT1999}. To be more precise, they proved that 
\begin{equation}\label{eq:Adachi-Tanaka}
\sup\limits_{\begin{subarray}{c} 
 u \in W^{1,n}(\R^n) \backslash \{ 0\} : \\ 
 \int_{\R^n} | \nabla u|^n dx \leqslant 1
\end{subarray}} 
 \frac 1{\|u\|_{L^n(\R^n)}^n} \int_{\R^n} \Phi_{n,1}(\alpha |u|^{n/(n-1)} ) dx 
 < +\infty,
 \tag{MT${}^\R_{us}$} 
\end{equation}
for any $\alpha \in (0,\alpha_n)$, where
\[
\Phi_{n,1}(t) = e^t -\sum\nolimits_{j=0}^{n-2} t^j/j!.
\] 
The constant $\alpha_n$, as appearing in the Moser--Trudinger inequality \eqref{eq:Moser--Trudinger}, is also sharp in the sense that if $\alpha \geqslant \alpha_n$, then the supremum in \eqref{eq:Adachi-Tanaka} is infinite. The question is: \textit{What happens when $\alpha =\alpha_n$?}

When $\alpha =\alpha_n$, the ``critical'' Moser--Trudinger inequality for any unbounded domain in $\R^n$ was proved by Ruf \cite{Ruf2005} for $n=2$ and by Li and Ruf \cite{LiRuf2008} for the case $n\geqslant 2$. This inequality asserts that 
\begin{equation}\label{eq:Li-Ruf}
\sup_{\begin{subarray}{c} u\in W_0^{1,n}(\Omega):\\
\|u\|_{W_0^{1,n}(\Omega)} \leqslant 1
\end{subarray}} \int_{\Omega} \Phi_{n,1}(\alpha_n |u|^{n/(n-1)} ) dx
< +\infty,
\tag{MT${}^\R_{uc}$} 
\end{equation}
for any domain $\Omega \subseteq \R^n$ with the supremum independent of $\Omega$, where
\[
\|u\|_{W^{1,n}_0(\Omega)} =(\|\nabla u\|_{L^n(\R^n)}^n + \|u\|_{L^n(\R^n)}^n )^{1/n}.
\] 
In addition, it was found that the same constant $\alpha_n$ is also sharp in the sense that the supremum in \eqref{eq:Li-Ruf} will be infinite if $\alpha_n$ is replaced by any $\alpha > \alpha_n$.

Following the works of Carleson and Chang \cite{CC1986},  Flucher \cite{Flucher1992}, and Lin \cite{Lin1996}, the existence of optimal functions for the Moser--Trudinger inequality in the entire space was studied in \cite{Ishiwata2011,LiRuf2008,Ruf2005}. More recently, sharp Moser--Trudinger inequalities has been established on the entire Heisenberg group at the critical case in \cite{LamLu2012c}, at the subcritical case in \cite{LamLuTang2014}, or in weighted form in Heisenberg-type groups in \cite{LamTang2013}, where any type of symmetrization arguments is not available. 

We note that there is a fundamental difference between \eqref{eq:Adachi-Tanaka} and \eqref{eq:Li-Ruf}. In fact, inequality \eqref{eq:Adachi-Tanaka} only holds for $\alpha < \alpha_n$ while inequality \eqref{eq:Li-Ruf} holds for all $\alpha \leqslant \alpha_n$. The reason behind this difference is that in \eqref{eq:Adachi-Tanaka} we require functions with the $L^n$-norm of their gradient less than or equal to $1$ while in \eqref{eq:Li-Ruf}, we require functions with $W^{1,n}$-norm less than or equal to $1$. In other word, the failure of the original Moser--Trudinger inequality \eqref{eq:Moser--Trudinger} on the entire space $\R^n$ can be recovered either by weakening the exponent $\alpha_n$ as in \eqref{eq:Adachi-Tanaka} or by strengthening Dirichlet's norm $\|\nabla u\|_{L^n(\R^n)}$ as in \eqref{eq:Li-Ruf}. 

A natural question arises: \textit{Can we still achieve the best constant $\alpha_n$ when we only require the condition $\|\nabla u\|_{L^n(\R^n)} \leqslant 1$?} This question was answered by Ibrahim, Masmoudi and Nakanishi \cite{IMN2015} for the case $n=2$ and by Masmoudi and Sani \cite{MS2015} for arbitrary $n\geqslant 2$. In their works, they proved the following inequality
\begin{equation}\label{eq:exactMT}
\sup\limits_{\begin{subarray}{c} 
 u \in {W^{1,n}}(\R^n) \backslash \{ 0\} : \\ 
 \|\nabla u\|_{L^n(\R^n)} \leqslant 1
\end{subarray}} 
\frac1{\|u\|_{L^n(\R^n)}^n} \int_{\R^n} \frac{\Phi_{n,1}(\alpha_n |u|^{n/(n-1)} )}{(1+|u|)^{n/(n-1)}} dx 
< +\infty. 
\tag{MT${}^\R_{ue}$} 
\end{equation}
Moreover, this inequality is sharp in the sense that it fails if the power $n/(n-1)$ in the denominator of \eqref{eq:exactMT} is replaced by any $p <n/(n-2)$.

\subsubsection{Adams inequalities on $\R^n$}

In the seminal work \cite{Adams1988}, Adams extended the Moser--Trudinger inequality \eqref{eq:Moser--Trudinger} to the higher order Sobolev space $W^{m,n/m}_0(\Omega)$, where $\Omega \subset \R^n$ is of finite measure. Let $m$ be a positive integer less than $n$, we denote the $m$th order gradient of a function $u$ on $\R^n$ by
\begin{equation*}
\nabla^m u=
\begin{cases}
\Delta^{m/2} u &\mbox{if $m$ is even,}\\
\nabla\Delta^{(m-1)/2} u &\mbox{if $m$ is odd.}
\end{cases}
\end{equation*}
Then Adams proved that there exists a sharp constant $\beta(n,m) > 0$ such that the following inequality 
\begin{equation}\label{eq:Adams}
\sup\limits_{\begin{subarray}{c} 
 u\in W^{m,n/m}_0(\Omega) : \\ 
 \int_\Omega |\nabla^m u|^{n/m} dx \leqslant 1
\end{subarray}}
 \int_\Omega \exp(\beta(n,m) |u|^{n/(n-m)}) dx 
< +\infty,
\tag{A${}^\R_b$}
\end{equation}
holds. Moreover, the constant $\beta(n,m)$ in \eqref{eq:Adams} is sharp in the sense that if we replace it by any $\beta > \beta(n,m)$, then the supremum becomes infinite. Adams was able to compute the sharp constant $\beta (n,m)$ to get
\[
\beta(n,m) = 
\begin{cases}
\Omega_n^{-1} \Big( \pi^{n/2} 2^m \dfrac{ \Gamma((m+1)/2) }{ \Gamma((n-m+1)/2) } \Big)^{n/(n-m)}&\mbox{if $m$ is odd,}\\
\Omega_n^{-1} \Big( \pi^{n/2} 2^m \dfrac{ \Gamma(m/2) }{ \Gamma((n-m)/2) } \Big)^{n/(n-m)} &\mbox{if $m$ is even.}
\end{cases}
\]
In terms of $\omega_{n-1}$, we can rewrite $\beta (n,m)$ as follows
\[
\beta(n,m) = 
\begin{cases}
n \omega_{n-1}^{-1} \Big( \pi^{n/2} 2^m \dfrac{ \Gamma ((m+1)/2 ) }{ \Gamma ((n-m+1)/2 ) } \Big)^{n/(n-m)}&\mbox{if $m$ is odd,}\\
n \omega_{n-1}^{-1} \Big( \pi^{n/2} 2^m \dfrac{ \Gamma (m/2 ) }{ \Gamma ((n-m)/2 ) } \Big)^{n/(n-m)} &\mbox{if $m$ is even.}
\end{cases}
\]
Notice that Adams' value of $\beta(n,1)$ agrees with Moser's value of $\alpha_n$.

Adams' inequality \eqref{eq:Adams} on domains of finite measure was recently extended by Tarsi \cite{Tarsi2012} to the larger space
\begin{equation*}
W_N^{m,n/m}(\Omega) =\Big\{u \in W^{m,n/m}\, :\, \Delta^j u = 0 \, \text{ on } \, \partial \Omega \, \text{ for } \, 0 \leqslant j \leqslant \big\lfloor m/2 \big\rfloor\Big\},
\end{equation*}
known as Sobolev's space with homogeneous Navier boundary conditions. Note that the space $W_N^{m,n/m}(\Omega)$ contains $W_0^{m,n/m}(\Omega)$ as a proper, closed subspace.

Inspired by \eqref{eq:Li-Ruf}, the sharp Adams inequality on the entire space $\R^n$ are also known, first proved by Ruf and Sani \cite{RufSani2013} for the case that $m$ is even and then by Lam and Lu \cite{LamLu2012d} for the remaining case. Their results read as follows: Let $m$ be an integer less than $n$ and $\Omega \subseteq \R^n$, for each $u\in W_0^{m,n/m}(\Omega)$ we denote
\[
\|u\|_{m,n} =
\begin{cases}
\|(-\Delta + I)^{m/2} u\|_{L^{n/m}(\Omega)}&\mbox{if $m$ is even,}\\
\left( \begin{gathered}
 \|\nabla(-\Delta + I)^{(m-1)/2} u\|_{L^{n/m}(\Omega)}^{n/m}  \hfill \\
+ \|(-\Delta + I)^{(m-1)/2} u\|_{L^{n/m}(\Omega)}^{n/m} \hfill \\ 
\end{gathered} \right)^{m/n} & \mbox{if $m$ is odd,}
\end{cases}
\]
then there holds
\begin{equation}\label{eq:criticalAdams}
\sup\limits_{\begin{subarray}{c} 
u\in W_0^{m,n/m}(\Omega):\\
 \|u\|_{m,n} \leqslant 1
\end{subarray}} 
\int_\Omega \Phi_{n,m}(\beta(n,m) |u|^{n/(n-m)}) dx < +\infty
\tag{A${}^\R_{uc}$}
\end{equation}
with the supremum independent of $\Omega$, where
\[
\Phi_{n,m}(t) = e^t -\sum\nolimits_{j=0}^{j_{n/m}-2} t^j/j!
\]
with
\[
j_{n/m} = \min\big\{j \in \N, j \geqslant n/m\big\}.
\] 
In addition, the constant $\beta(n,m)$ in \eqref{eq:criticalAdams} is sharp in the sense that if we replace $\beta(n,m)$ in \eqref{eq:criticalAdams} by any $\beta > \beta(n,m)$, then the supremum in \eqref{eq:criticalAdams} will be infinite. We refer the reader to \cite{LamLu2013} for a sharp Adams-type inequality of fractional order $\alpha \in (0,n)$, where a rearrangement-free argument was used.

Recently, Masmoudi and Sani \cite{MS2014} obtained a sharp Adams inequality with exact growth condition in $\R^4$. Then, Lu, Tang, and Zhu \cite{LTZ2015} extended the result of Masmoudi and Sani to all dimension $n \geqslant 2$ to get the following inequality
\begin{equation}\label{eq:AdamsexactRn}
\sup\limits_{\begin{subarray}{c} 
 u\in W^{2,n/2}(\R^n) \backslash \{ 0 \} : \\ 
 \|\nabla^2 u\|_{L^{n/2}(\R^n)} \leqslant 1
\end{subarray}} 
\frac 1{\|u\|_{L^{n/2}(\R^n)}^{n/2}} \int_{\R^n} \frac{\Phi_{n,2}(\beta(n,2) |u|^{n/(n-2)})}{(1 +|u|)^{n/(n-2)}} dx < +\infty. 
\tag{A${}^\R_{ue}$} 
\end{equation}
Moreover, the power $n/(n-2)$ in the denominator of \eqref{eq:AdamsexactRn} is sharp in the sense that the supremum above will become infinite if we replace the power in the denominator by any $p < n/(n-2)$. In applications, the inequality \eqref{eq:AdamsexactRn} implies a subcritical sharp Adams inequality in the spirit of Adachi and Tanaka, which strengthens an inequality of Ogawa and Ozawa \cite{OO1991}. It also implies a sharp Adams-type inequality under the norm
\[
\|u\|_{W^{2,n/2}} = \big(\|u\|_{L^{n/2}(\R^n)}^{n/2} + \|\Delta u\|_{L^{n/2}(\R^n)}^{n/2} \big)^{2/n},
\] 
namely
\begin{equation}\label{eq:Adamssharpstrong}
\sup\limits_{\begin{subarray}{c} 
u \in W^{2,n/2}(\R^n):\\
 \|u\|_{W^{2,n/2}} \leqslant 1
\end{subarray}} 
\int_{\R^n} \Phi_{n,2} \big(\beta(n,2) |u|^{n/(n-2)} \big) dx < +\infty.
\tag{A${}^\R_{uc}$} 
\end{equation}
The constant $\beta(n,2)$ is sharp; see \cite{LTZ2015,MS2014} for more details. A version of higher order derivatives of \eqref{eq:Adamssharpstrong} has recently been proved by Fontana and Morpurgo in \cite{FM2015}. We remark that a version of higher order derivatives of \eqref{eq:exactMT} and \eqref{eq:AdamsexactRn} is still unknown; however, a weaker result can be found in \cite{FM2015}.

\subsection{Moser--Trudinger and Adams inequalities on $\mathbb H^n$}

Although there have been extensive works on the best constants for the Moser--Trudinger and Adams inequalities in the Euclidean space, on Heisenberg's group, and on compact Riemannian manifolds as listed above, much less is known for the sharp constants for the Moser--Trudinger and Adams inequalities on hyperbolic spaces. 

The hyperbolic space $\H^n$ with $n\geqslant 2$ is a complete, simply connected Riemannian manifold having constant sectional curvature equal to $-1$, and for a given dimensional number, any two such spaces are isometric \cite{Wolf1967}. There is a number of models for $\H^n$, however, the most important models are the half-space model, the ball model, and the hyperboloid (or Lorentz model). In this paper, we will use the ball model since this model is especially useful for questions involving rotational symmetry. Given $n \geqslant 2$, we denote by $\mathbb B^n$ the open unit ball in $\R^n$. Clearly, $\mathbb B^n$ can be endowed with the Riemannian metric
\[
g(x) = \sum_{i=1}^n \Big(\frac 2{1-|x|^2}\Big)^2 dx_i^2,
\]
which is called the ball model of the hyperbolic space $\H^n$. The volume element of $\H^n$ is given by 
\[
dV_g(x) = \Big(\frac2{1-|x|^2}\Big)^n dx,
\]
where $dx$ denotes the Lebesgue measure in $\R^n$. For any subset $E\subset \mathbb B^n$, we denote $|E|= \int_E dV_g$. Let $d(0,x)$ denote the hyperbolic distance between the origin and $x$. It is well-known that
$$d(0,x) = \log \big( (1+|x|)/(1-|x|) \big)$$
for arbitrary $x\in \mathbb B^n$. In this new context, we still use $\nabla$ and $\Delta$ to denote the Euclidean gradient and Laplacian as well as $\langle\cdot ,\cdot\rangle$ to denote the standard inner product in $\R^n$. Then, in terms of $\nabla$, $\Delta$, and $\langle\cdot ,\cdot\rangle$, the hyperbolic gradient $\nabla_g$ and the Laplace--Beltrami operator $\Delta_g$ are given by
\[
\nabla_g = \Big(\frac{1-|x|^2}2\Big)^2 \nabla,\quad \Delta_g = \Big(\frac{1-|x|^2}2\Big)^2 \Delta + (n-2) \frac{1-|x|^2}2 \langle x,\nabla\rangle.
\]
Given a bounded domain $\Omega\subset \H^n$, we denote
\[
\|f\|_{p,\Omega} = \Big(\int_\Omega |f|^p dV_g \Big)^{1/p}
\]
for each $1\leqslant p < \infty$. Then we have the following
\[
\|\nabla_g f\|_{n,\Omega} = \Big(\int_\Omega\langle \nabla_g f,\nabla_g f\rangle_g^{n/2} dV_g\Big)^{1/n} = \Big(\int_\Omega |\nabla f|^n dx\Big)^{1/n}.
\]
In the case $\Omega = \H^n$, we simply write $\|f\|_p$ instead of $\|f\|_{p, \H^n}$ for all $1\leqslant p< +\infty$. Throughout the paper, we also use $W_0^{2,n/2}(\Omega)$ to denote the completion of $C_0^\infty(\Omega)$ under the norm
\[
\|u\|_{W^{2,n/2}_0(\Omega)} = \Big(\int_\Omega |u|^{n/2} dV_g + \int_\Omega |\Delta_g u|^{n/2} dV_g\Big)^{2/n}.
\]
In particular, we will denote by $W^{2,n/2}(\H^n)$ the completion of $C_0^\infty(\H^n)$ under the norm
\[
\|u\|_{W^{2,n/2}(\H^n)} = \Big(\int_{\H^n} |u|^{n/2} dV_g + \int_{\H^n} |\Delta_g u|^{n/2} dV_g\Big)^{2/n}.
\] 

\subsubsection{Moser--Trudinger inequalities on $\H^n$}

In \cite{MS2010}, Mancini and Sandeep established a sharp Moser--Trudinger inequality on the $2$-dimensional hyperbolic space $\mathbb H^2$. They proved
\begin{equation*}\label{eq:ManciniSandeep}
\sup\limits_{\begin{subarray}{c} 
u\in C_0^\infty(\mathbb H^2):\\
\|\nabla_g u\|_2 \leqslant 1
\end{subarray}} 
\int_{\mathbb H^2} (e^{4\pi u^2} -1) dV_g 
< +\infty,
\end{equation*}
where the constant $4\pi^2$ is sharp in the sense that the supremum above will be infinite if $4\pi^2$ is replaced by any number larger than $4\pi^2$. 

The Moser--Trudinger inequality on bounded domains $\Omega$ in any hyperbolic space of any higher dimension was proved by Lu and Tang \cite{LT2013}
\begin{equation}\label{eq:LuTangboundeddomain}
\sup\limits_{\begin{subarray}{c} 
u\in C_0^\infty(\Omega) :\\
\|\nabla_g\|_{n,\Omega} \leqslant 1
\end{subarray}}
 \int_{\Omega} \exp(\alpha_n |u|^{n/(n-1)}) dx 
< +\infty \tag{MT${}^{\H}_{b}$}
\end{equation}
with the sharp constant $\alpha_n$. We note that the best constant in the Moser--Trudinger inequality on bounded domains in hyperbolic space \eqref{eq:LuTangboundeddomain} is similar to the one of the Moser--Trudinger inequality on bounded domains in the Euclidean space \eqref{eq:Moser--Trudinger}.

When $\Omega$ has \textit{infinite volume}, a sharp ``subcritical'' Moser--Trudinger-type inequality in the spirit of Adachi--Tanaka was recently proved by Lu and Tang in \cite{LT2013}. They showed that 
\begin{equation}\label{eq:subcriticalMThyperbolic}
\sup\limits_{\begin{subarray}{c} 
u \in W^{1,n}(\H^n)\setminus\{0\}: \\
\|\nabla_g u\|_n \leqslant 1
\end{subarray}}
\frac1{\|u\|_n^n}\int_{\H^n} \Phi_{n,1}(\alpha |u|^{n/(n-1)}) dV_g 
< + \infty, 
\tag{MT${}^{\H}_{us}$}
\end{equation} 
for any $\alpha \in (0,\alpha_n)$ and the constant $\alpha_n$ is sharp in the sense that for $\alpha \geqslant \alpha_n$, the supremum in \eqref{eq:subcriticalMThyperbolic} will be infinite.

It was also established in \cite{LT2013} a sharp ``critical'' Moser--Trudinger inequality on the entire hyperbolic space when we restrict the norms of functions to the full hyperbolic Sobolev norm, namely,
\begin{equation}\label{eq:criticalMThyperbolic}
\sup\limits_{\begin{subarray}{c} 
u\in W^{1,n}(\H^n):\\
 \|\nabla_g u\|_n^n + \tau\|u\|_n^n \leqslant 1
\end{subarray}} 
\int_{\H^n} \Phi_n(\alpha_n |u|^{n/(n-1)} ) dV_g 
< + \infty 
\tag{MT${}^{\H}_{uc}$}
\end{equation}
for any $\tau > 0$. The constant $\alpha_n$ is sharp in the sense that the supremum above will become infinite if $\alpha_n$ is replaced by any $\alpha > \alpha_n$. In view of \eqref{eq:subcriticalMThyperbolic} and \eqref{eq:criticalMThyperbolic}, a natural question, as in the Euclidean space, arises: \textit{Can we still achieve the best constant $\alpha_n$ when we only require the restriction on the norm $\|\nabla_g u\|_n \leqslant 1$?} This question was also answered in \cite{LT2015} by Lu and Tang. They proved a sharp Moser--Trudinger inequality with exact growth condition in hyperbolic space as follows
\begin{equation}\label{eq:exactMThyperbolic}
\sup\limits_{\begin{subarray}{c} 
u\in W^{1,n}(\H^n)\setminus\{0\}:\\
 \|\nabla_g u\|_n\leqslant 1
\end{subarray}} 
\frac1{\|u\|_n^n}\int_{H^n} \frac{\Phi_{n,1}(\alpha_n |u|^{n/(n-1)})}{(1 + |u|)^{n/(n-1)}} dV_g < +\infty. 
\tag{MT${}^{\H}_{ue}$}
\end{equation}
In \eqref{eq:exactMThyperbolic}, the power $n/(n-1)$ in the denominator of \eqref{eq:exactMThyperbolic} is sharp in the sense that the supremum becomes infinite if we replace the power $n/(n-1)$ in the denominator by any $ p < n/(n-1)$. It is evidence that \eqref{eq:exactMThyperbolic} implies \eqref{eq:subcriticalMThyperbolic} and \eqref{eq:criticalMThyperbolic}.

\subsubsection{Adams inequalities on $\H^n$}

A Moser--Trudinger-type inequality of higher order derivatives, or Adams-type inequality, in hyperbolic spaces was recently established in \cite{FM2015,KS2016}. In \cite{KS2016}, Karmakar and Sandeep proved a sharp Adams-type inequality in $\H^n$ with even $n$ under the condition
\[
\int_{H^n} u P_{n/2}^g( u) \, dV_g \leqslant 1,
\] 
where $P_k^g$ denotes the $(2k)$th order GJMS operator defined by
\[\left\{
\begin{split}
P_1^g =& -\Delta_g - n(n-2)/4,\\
P_k^g =& P_1^g(P_1^g +2) (P_1^g+6)\cdots (P_1^g + k(k-1)).
\end{split}
\right.\] 
More precisely, they established the following inequality
\begin{equation}\label{eq:KarmakarSandeep2016}
\sup\limits_{\begin{subarray}{c} 
u \in C_0^\infty(\H^n):\\
 \int_{\H^n} u P_{n/2}^g (u) \, dV_g \leqslant 1
\end{subarray}} 
\int_{\H^n} (e^{\beta(n,n/2) u^2} -1) dV_g < +\infty. 
\tag{A${}^{\H 2}_{}$}
\end{equation}
The constant $\beta(n,n/2)$ in \eqref{eq:KarmakarSandeep2016} is sharp and cannot be improved. For any integer $m$ less than $n$, let us denote
\begin{equation*}
\nabla_g^m = 
\begin{cases}
\Delta_g^{m/2}&\mbox{if $k$ even,}\\
\nabla_g\Delta_g^{(m-1)/2}&\mbox{if $k$ odd.}
\end{cases}
\end{equation*} 
In \cite{FM2015}, Fontana and Morpurgo established the following sharp Adams inequality in the entire hyperbolic space $\H^n$ as follows
\begin{equation}\label{eq:FontanaMorpurgo2015}
\sup\limits_{\begin{subarray}{c} 
u\in C_c^\infty(\H^n):\\
 \|\nabla_g^m u\|_{n/m}\leqslant 1
\end{subarray}} 
\int_{\H^n} \Phi_{n,m}(\beta(n,m) |u|^{n/(n-m)}) dV_g < +\infty. 
\tag{A${}^{\H}_{u}$}
\end{equation}
The constant $\beta(n,m)$ is again sharp in the sense that the supremum in \eqref{eq:FontanaMorpurgo2015} will become infinite if we replace $\beta(n,m)$ by any $\beta > \beta(n,m)$.

Motivated by \eqref{eq:AdamsexactRn}, in the recent paper \cite{Kar2015}, Karmakar established a sharp Adams-type inequality in $\H^4$ with the exact growth condition as follows
\begin{equation}\label{eq:Karmakar2015}
\sup\limits_{\begin{subarray}{c} 
u \in W^{2,2}(\H^4)\setminus\{0\}:\\
 \int_{\H^4} u P_2^g(u) dV_g \leqslant 1
\end{subarray}} 
\frac1{\|u\|_2^2}
\int_{\H^4} \frac{e^{32\pi^2 u^2} -1}{(1+|u|)^2} dV_g < +\infty . 
\tag{A${}^{\H}_{ue}$}
\end{equation}
Moreover, this inequality is sharp in the sense that the supremum in \eqref{eq:Karmakar2015} will become infinite if the power $2$ in the denominator of \eqref{eq:Karmakar2015} is replaced by any $p < 2$.

\subsection{Main results}

As far as we know, no sharp Adams-type inequality with exact growth condition for general $n \geqslant 3$ is known. In the first part of this paper, as an analog of \eqref{eq:AdamsexactRn}, we will provide a sharp Adams-type inequality with exact growth condition in $\H^n$ for all $n\geqslant 3$ under the norm $\|\Delta_g u\|_{n/2}$. The exact statement of this result is as follows.

\begin{theorem}\label{Maintheorem}
There exists a dimensional constant $C(n) > 0$ such that for all $u\in W^{2,n/2}(\H^n)$ with $\|\Delta_g u\|_{n/2} \leqslant 1$ there holds
\begin{equation}\label{eq:Adamsexact}
\int_{\H^n} \frac{\Phi_{n,2}(\beta(n,2) |u|^{n/(n-2)})}{(1+|u|)^{n/(n-2)}} dV_g \leqslant C(n)  \|u\|_{n/2}^{n/2}.
\tag{AMT${}^{\H}_{ue}$}
\end{equation}
Moreover, this inequality is sharp in the sense that the supremum
\[
\sup\limits_{\begin{subarray}{c} 
u\in W^{2,n/2}(\H^n)\setminus\{0\} :\\
 \|\Delta_g u\|_{n/2} \leqslant 1
\end{subarray}} 
\frac1{\|u\|_{n/2}^{n/2}}\int_{\H^n} \frac{\Phi_{n,2}(\beta |u|^{n/(n-2)})}{(1+|u|)^p} dV_g
\]
becomes infinite either for $\beta > \beta(n,2)$ and $p=n/(n-2)$, or $\beta =\beta(n,2)$ and $p <n/(n-2)$.
\end{theorem}

Notice that our inequality \eqref{eq:Adamsexact} when $n=4$ is slightly different from \eqref{eq:Karmakar2015} of Karmakar. To prove Theorem \ref{Maintheorem}, we borrow some ideas in the proof of \eqref{eq:AdamsexactRn} given in \cite{LTZ2015,MS2014} plus some useful inequalities involving the decreasing rearrangement given in Section \ref{sec-Preliminaries}. 

Let us now discuss some interesting consequences of Theorem \ref{Maintheorem}. An immediate consequence of it is the following subcritical sharp Adams-type inequality in the spirit of Adachi and Tanaka in $W^{2,n/2}(\H^n)$.

\begin{theorem}\label{AdachiTanakatype}
For any $\alpha \in (0,\beta(n,2))$, there exists a constant $C(n,\alpha) >0$ such that
\begin{equation}\label{eq:AdachiTanakatype}
\int_{\H^n} \Phi_{n,2}(\alpha |u|^{n/(n-2)}) dV_g \leqslant C(n,\alpha)\|u\|_{n/2}^{n/2}
\tag{AMT${}^{\H}_{us}$}
\end{equation}
for any function $u \in W^{2,n/2}(\H^n)$ with $\|\Delta_g u\|_{n/2} \leqslant 1$. The inequality is sharp in the sense that it is false if $\alpha \geqslant \beta(n,2)$. Furthermore, we have the following estimate
\begin{equation}\label{eq:Cnalpha}
C(n,\alpha) \leqslant \frac{C(n)}{\beta(n,2) -\alpha},
\end{equation}
for some positive constant $C(n)$ depending only on $n$.
\end{theorem}

Clealry, the estimate \eqref{eq:Cnalpha} provides an asymptotic behavior of the constant $C(n,\alpha)$ in the subcritical inequality \eqref{eq:AdachiTanakatype} as $\alpha$ tends to $\beta(n,2)$. Such a result on the Euclidean space can be found in \cite{LTZ2015} for the Moser--Trudinger and Adams inequalities. 

In view of Theorem \ref{AdachiTanakatype}, it is easy to obtain a critical sharp Adams-type inequality in $W^{2,n/2}(\H^n)$ involving the norm
\[
\|u\|_{W^{2,n/2},\tau} =\big(\|\Delta_g u\|_{n/2}^{n/2} + \tau \|u\|_{n/2}^{n/2}\big)^{2/n}
\]
where $\tau >0$. This is the content of the following result.

\begin{theorem}\label{Adamsnormtau}
Let $\tau >0$, there exist a constant $C(n,\tau) > 0$ such that
\begin{equation}\label{eq:Adamsnormtau}
\sup\limits_{u: \|u\|_{W^{2,n/2},\tau} \leqslant 1} \int_{\H^n} \Phi_{n,2}(\beta(n,2) |u|^{n/(n-2)}) dV_g \leqslant C(n,\tau).
\tag{AMT${}^{\H}_{uc}$}
\end{equation}
The constant $\beta(n,2)$ is sharp in the sense that the supremum becomes infinite if we replace $\beta(n,2)$ by any $\beta > \beta(n,2)$. Furthermore, we have the following estimate
\begin{equation}\label{eq:Cntau}
C(n,\tau) \leqslant C(n)/\tau,
\end{equation}
for some positive constant $C(n)$ depending only on $n$.
\end{theorem}

In the next part of our paper, we also prove that Theorem \ref{AdachiTanakatype} can imply an improved version of the sharp Adams inequality \eqref{eq:Adamsnormtau} in the spirit of Lions \cite{Lions1985}. To make this statement clear, we shall prove the following result.

\begin{theorem}\label{improvedversion}
There exists a constant $C(n) > 0$ such that for any $u\in W^{2,n/2}(\H^n)$ with $\|\Delta_g u\|_{n/2} < 1$, the following inequality holds
\begin{equation}\label{eq:Lionsversion}
\int_{\H^n} \Phi_{n,2}\bigg(\frac{2^{2/(n-2)} \beta(n,2)}{\big(1 + \|\Delta_g u\|_{n/2}^{n/2}\big)^{2/(n-2)}} |u|^{n/(n-2)} \bigg) dV_g \leqslant C(n) \frac{\|u\|_{n/2}^{n/2}}{1-\|\Delta_g u\|_{n/2}^{n/2}}.
\tag{AMT${}^{\H}_{ucL}$}
\end{equation}
Consequently, we have for any $\tau > 0$,
\begin{equation}\label{eq:improvement}
\sup\limits_{u: \|u\|_{W^{2,n/2},\tau}\leqslant 1} \int_{\H^n} \Phi_{n,2}\bigg(\frac{2^{2/(n-2)} \beta(n,2)}{\big(1 + \|\Delta_g u\|_{n/2}^{n/2}\big)^{2/(n-2)}} |u|^{n/(n-2)} \bigg) dV_g \leqslant \frac{C(n)}\tau.
\end{equation}
The constant $\beta(n,2)$ is sharp in the sense that inequality \eqref{eq:Lionsversion} does not hold if we replace $\beta(n,2)$ by any larger constant.
\end{theorem}

Notice that if $\|\Delta_g u\|_{n/2} < 1$, then $2^{2/(n-2)} \big(1 + \|\Delta_g u\|_{n/2}^{n/2}\big)^{-2/(n-2)} \beta(n,2) > \beta(n,2)$. Therefore \eqref{eq:improvement} is indeed an improvement of \eqref{eq:Adamsnormtau}. It is noted that the Euclidean versions of \eqref{eq:Lionsversion} and \eqref{eq:improvement} was recently proved by Lam, Lu and Tang in \cite[Theorem 1.5]{LamLuTang2016}. Their proofs are based on the domain decomposition method. Our proof below is different with theirs and is derived from Theorem \ref{AdachiTanakatype}. 

Despite the fact that Theorem \ref{Adamsnormtau} can be derived from Theorem \ref{AdachiTanakatype}, however, it turns out that these two theorems are in fact equivalent; see Section \ref{sec-VariousAdamsType} below. It seems very surprise since Theorem \ref{Adamsnormtau} concerns  the critical version of the sharp Adams-type inequality while Theorem \ref{AdachiTanakatype} concerns the subcritical version. In the Euclidean case, this fact was recently observed by Lam, Lu and Zhang in \cite{LamLuZhang2015}. Furthermore, it is evident that Theorem \ref{improvedversion} implies Theorem \ref{Adamsnormtau}. Hence, up to dimensional constants, the three inequalities \eqref{eq:AdachiTanakatype}, \eqref{eq:Adamsnormtau}, and \eqref{eq:Lionsversion} are equivalent.

We also establish a sharp Adams--Moser--Trudinger-type inequality in the Sobolev space with homogeneous Navier boundary condition $W_{N,g}^{m,n/m}(\Omega)$ for any bounded domain $\Omega \subset \H^n$. Here the space $W^{m,n/m}_{N,g}(\Omega)$ is defined by
\[
W_{N,g}^{m,n/m}(\Omega) =\Big\{u\in W^{m,n/m} (\Omega) :\, \Delta_g^j u = 0 \, \text{ on } \, \partial \Omega, \, j = 0,1, ... ,\big \lfloor m/2\big\rfloor\Big\}.
\]
Note that $W_{N,g}^{m,n/m}(\Omega)$ contains the Sobolev space $W_0^{m,n/m}(\Omega)$ as a closed subspace. Our next theorem is a hyperbolic analog of the result of Tarsi in the Euclidean space; see \cite[Theorem 4]{Tarsi2012}.

\begin{theorem}\label{Tarsihyperbolic}
Let $n > 2$ and $\Omega$ be a bounded domain in $\H^n$. There exists a constant $C(n) > 0$ such that for any integer $m\in [1,n)$ and for all $u\in W_{N,g}^{m,n/m}(\Omega)$ with $\|\nabla_g^m u\|_{n/m} \leqslant 1$, there holds
\begin{equation}\label{eq:Tarsihyperbolic}
\int_\Omega \exp \big( \beta(n,m) |u|^{n/(n-m)} \big) dV_g \leqslant C(n) |\Omega|.
\tag{AMT${}^{\H}_{bcN}$}
\end{equation}
The constant $\beta(n,m)$ is sharp in the sense that the supremum of the left hand side of \eqref{eq:Tarsihyperbolic} in $W_{N,g}^{m,n/m}(\Omega)$ becomes infinity if it is replaced by any larger $\beta$.
\end{theorem}

Another aspect of the Moser--Trudinger and Adams inequalities concerns the concentration-compactness phenomena. In his famous paper \cite{Lions1985}, Lions proved a so-called concentration-compactness principle for the Moser functional, known as Lions' lemma, which asserts that given a bounded domain $\Omega$ in $\R^n$ if a sequence $\{u_j\}_j \subset W^{1,n}_0(\Omega)$ with $\|\nabla u_j\|_{L^n(\Omega)} = 1$ converges weakly to a non-zero function $u\in W^{1,n}_0(\Omega)$, then there holds
\begin{equation}\label{eq:con-comLions}
\sup_{j} \int_\Omega \exp \big(p \beta(n,1) |u_j|^{n/(n-1)} \big) dx < +\infty
\end{equation}
for any $p < (1 -\|\nabla u^\sharp\|_{L^n(\Omega^\sharp)}^n)^{- 1/(n-1)}$. Here, $u^\sharp$ and $\Omega^\sharp$ are the rearrangement of $u$ and $\Omega$, respectively; see Section \ref{sec-Preliminaries} below for the definition. Note that the inequality \eqref{eq:con-comLions} does not give any further information than the Moser--Trudinger inequality if the sequence converges weakly to the zero function, but the implication of \eqref{eq:con-comLions} is that the critical Moser functional is compact outside a weak neighborhood of zero function. 

In \cite{CCH2013}, \v{C}erny, Cianchi and Hencl improved Lions' result by showing that the inequality \eqref{eq:con-comLions} still holds for any 
\[
p < P_{n,1}(u) =: (1 -\|\nabla u\|_{L^n(\Omega)}^n)^{- 1/(n-1)}.
\] 
Moreover, the threshold $P_{n,1}(u)$ is sharp. A more detailed discussion on Lions' lemma and its generalization to functions with unrestricted boundary condition can be found in \cite{CCH2013}. 

Recently,Lions' lemma for the Moser functional has been extended on whole space $\R^n$ by do \'O, de Souza, de Medeiros and Severo \cite{doO2014b} by exploiting the approach of \v{C}erny, Cianchi and Hencl \cite{CCH2013}. The concentration-compactness principle for Adams' functional has been established by do \'O and Macedo \cite{doO2014a} by using the rearrangement argument and the generalization of Talenti's comparison principle. In a very recent paper, Lions-type lemma for Adams' functional on whole space $\R^n$ was proved by Nguyen \cite{VHN2016}. The method used in \cite{VHN2016} is a further modification of the method of \v{C}erny, Cianchi and Hencl \cite{CCH2013} and is completely based on estimates for decreasing rearrangement of functions in terms of their higher order derivatives. 

Following the approach used in \cite{VHN2016}, we establish a Lions-type lemma for Adams' inequality in the whole hyperbolic space $\H^n$. To the best of our knowledge, no Lions-type lemma for Adams' inequality in $\H^n$ in full generality is known except for a few cases. For examples, it was established by Karmakar \cite{Kar2015} in $W^{1,n}(\H^n)$ and $W^{2,n/2}(\H^n)$ by using a cover lemma and a Lions-type lemma for the Moser--Trudinger and Adams inequalities on bounded domains of $\R^n$. However, his proof is completely different with ours given below. The following is our result.

\begin{theorem}\label{Lionslemma}
Let $m$ be a positive integer less than $n$ and let $\{u_j\}_j$ be a sequence in $W^{m,n/m}(\H^n)$ such that $\|\nabla_g^m u_j\|_{n/m} \leqslant 1$ and $u_j$ converges weakly to a non-zero function $u$ in $W^{m,n/m}(\H^n)$. Then
\begin{equation}\label{eq:Lions}
\sup_{j} \int_{\H^n} \Phi_{n,m} \big(p \beta(n,m) |u_j|^{n/(n-m)} \big) dV_g < +\infty
\tag{AMT${}^{\H}_{CC}$}
\end{equation}
for all $p < P_{n,m}(u)$ where
\[
P_{n,m}(u) :=
\begin{cases}
(1 -\|\nabla^m u\|_{n/m}^{n/m})^{-m/(n-m)}&\mbox{if $\|\nabla^m u\|_{n/m} < 1$},\\
+\infty &\mbox{if $\|\nabla_g^m u\|_{n/m} =1$.}
\end{cases}
\]
Moreover, the threshold $P_{n,m}(u)$ is sharp in the sense that \eqref{eq:Lions} is no longer true if $p \geqslant P_{n,m}(u)$.
\end{theorem}

The rest of this paper is organized as follows: We recall some facts about the rearrangement in the hyperbolic space and prove some useful inequalities involving the rearrangement such as a Talenti-type comparison principle and an estimate for the rearrangement of weak solutions to a Dirichlet problem in hyperbolic spaces in Section \ref{sec-Preliminaries}. Having all preliminaries, we prove Theorem \ref{Maintheorem} in Section \ref{sec-AdamsWithExactGrowth} while Theorems \ref{AdachiTanakatype}, \ref{Adamsnormtau}, and \ref{improvedversion} will be proved in Section \ref{sec-VariousAdamsType}. Then we proved Theorem \ref{Tarsihyperbolic} in Section \ref{sec-AdamsTypeWithBoudary}. In Section \ref{sec-LionsLemma}, we prove Theorem \ref{Lionslemma}.

\setcounter{tocdepth}{1}
\renewcommand{\baselinestretch}{0.7}\normalsize
\tableofcontents
\renewcommand{\baselinestretch}{1.0}\normalsize
\parskip=8pt
\setcounter{tocdepth}{2}


\section{Preliminaries}
\label{sec-Preliminaries}

\subsection{Rearrangement in hyperbolic spaces}

It is now known that the symmetrization argument works well in the setting of hyperbolic spaces. It is not only the key tool in the proof of the classical Moser--Trudinger in $\H^n$ \cite{LT2013} but also a key tool in our proof of Theorem \ref{Maintheorem}. 

Let us now recall some facts about the rearrangement in the hyperbolic spaces. Let the function $f: \H^n\to \R$ be such that 
\[
\big |\{x\in \H^n\, :\, |f(x)| > t\} \big | = \int_{\{x\in \H^n\,:\, |f(x)|>t\}} dV_g < +\infty
\]
for every $t >0$. Its \textit{distribution function} is defined by
\[
\mu_f(t) = \big |\{x\in \H^n\, :\, |f(x)| > t\} \big |.
\]
Then its \textit{decreasing rearrangement} $f^*$ is defined by
\[
f^*(t) = \sup\{s > 0\, :\, \mu_f(s) > t\}.
\]
Now, \textit{Schwarz's symmetrization} of $f$, denoted by $f^\sharp$, is the function $f^\sharp: \H^n \to \R$ defined by 
\[
f^\sharp(x) = f^*(|B(0,d(0,x))|),
\]
where the notation $B(0,r)$ denotes the ball in $\H^n$ centered at the origin $0$ with hyperbolic radius $r$ and as already mentioned $|B(0,r)|$ is its hyperbolic volume. In $\R^n$, we use $B_r$ to denote the ball centered at the origin $0$ with radius $r$. Using the distance $d(x,0)$, it is not hard to verify that
\[
B(0, r) = B_{\tanh (r/2)}.
\]
From this fact, we find that
\begin{equation}\label{VolumeSphereInHyperbolic}
\big| \partial B(0,r) \big|  =n \Omega_n   \sinh^{n-1}(r) 
\end{equation} 
and that
\begin{equation}\label{VolumeBallInHyperbolic}
\big| B(0,r) \big|  =n \Omega_n  \int_0^r \sinh^{n-1}(s) ds.
\end{equation} 
Note that for any continuous increasing function $\Phi: [0, +\infty) \to [0, +\infty)$ we have
\[
\int_{\H^n} \Phi(|f|) dV_g = \int_{\H^n} \Phi(f^\sharp) dV_g.
\]
Moreover, the Hardy--Littlewood inequality implies that
\[
\int_{\H^n} |f h| dV_g \leqslant \int_{\H^n} f^\sharp h^\sharp dV_g,
\]
for any functions $f,h:\H^n \to \R$. Since $f^*$ is non-increasing, the \textit{maximal function} $f^{**}$ of the rearrangement $f^*$ defined by
\[
f^{**}(t) = \frac1t \int_0^t f^*(s) ds
\]
for $s \geqslant 0$ is also non-increasing. Furthermore, it is easy to see that $f^{**} \geqslant f^*$. Moreover, we have the following.

\begin{lemma}\label{Hardy}
Let $f\in L^p(\H^n)$ with $p\in (1, +\infty)$. Then we have
\[
\Big(\int_0^{+\infty} f^{**}(s)^p ds \Big)^{1/p} \leqslant p' \Big(\int_0^{+\infty} f^*(s)^p ds \Big)^{1/p},
\]
where $1/p+ 1/p' = 1$. In particular, if ${\rm supp}\, f\subset\Omega \subset \H^n$, then
\[
\Big(\int_0^{|\Omega|} f^{**}(s)^p ds \Big)^{1/p} \leqslant p' \Big(\int_0^{|\Omega|} f^*(s)^p ds \Big)^{1/p}.
\]
\end{lemma}

Lemma \ref{Hardy} above is just an immediate consequence of a well-known result of G.H. Hardy, for interested reader, we refer to \cite[Proposition 3.1]{MS2014}.

\subsection{Some useful inequalities involving rearrangements}

In this subsection, we list some useful facts, which shall be used in the proof of Theorem \ref{Maintheorem}, whose proofs will be given in the next subsection. We first prove a comparison principle for solutions of a Dirichlet problem, which is similar to the one of Talenti in the Euclidean space \cite{Tal1976}. 

Let $\Omega\subset \H^n, n\geqslant 2$, be a bounded, open set and let $f$ be a suitable $L^p$-function with $p>1$. We consider the following Dirichlet problem
\begin{equation}\label{eq:Dirichletproblem}
\begin{cases}
-\Delta_g u = f&\mbox{in $\Omega$},\\
\quad\quad u = 0&\mbox{on $\partial \Omega$.}
\end{cases}
\end{equation}
Let us denote by $\Omega^\sharp$ the ball centered at origin such that $|\Omega^\sharp| =|\Omega|$ and consider the Dirichlet problem
\begin{equation}\label{eq:Dirichletproblemrearrangement}
\begin{cases}
-\Delta_g v = f^\sharp&\mbox{in $\Omega^\sharp$},\\
\quad\quad v = 0&\mbox{on $\partial \Omega^\sharp$.}
\end{cases}
\end{equation}
Then we have the following comparison principle.

\begin{proposition}\label{comparison}
Any weak solutions $u$ and $v$ to \eqref{eq:Dirichletproblem} and \eqref{eq:Dirichletproblemrearrangement} respectively enjoys the following a prior estimate
\[u^\sharp (x) \leqslant v (x)\]
in $\Omega^\sharp$.
\end{proposition}

We next use Proposition \ref{comparison} to obtain a comparison principle for higher derivatives $\Delta_g^k$; see Proposition \ref{evenordercomparison} below. For this reason, given $f$, we consider the following two problems
\begin{equation}\label{eq:DirichletproblemHO}
\begin{cases}
(-\Delta_g)^k u = f&\mbox{ in } \Omega,\\
\quad\quad \Delta_g^i u = 0 & \mbox{ on } \partial \Omega
\end{cases}
\end{equation}
for all $i =0,1, ... , k-1$ and
\begin{equation}\label{eq:DirichletproblemrearrangementHO}
\begin{cases}
(-\Delta_g)^k v = f^\sharp&\mbox{in $\Omega^\sharp$},\\
\quad\quad \Delta_g^i v = 0 & \mbox{ on } \partial \Omega^\sharp
\end{cases}
\end{equation}
for all $i =0,1, ... , k-1$. Here $\Omega$ is again a bounded open domain in $\H^n$. To study \eqref{eq:DirichletproblemHO}, we denote
\[
u_i = (-\Delta_g)^i u
\] 
for $i =0,1, ... , k$. It is obvious to see that $u_0,u_1,\dots,u_{k-1}$ solve the following problems 
\begin{subequations}\label{prob-P_i}
\begin{align}
\begin{cases}
-\Delta_g u_i  = u_{i+1} &\mbox{ in }\Omega,\\
 \quad\quad  u_i = 0 &\mbox{ on }\partial \Omega.
 \end{cases}
\tag*{(\ref{prob-P_i}$_i$)}
\end{align}
\end{subequations}
Similarly, to study \eqref{eq:DirichletproblemrearrangementHO} we denote 
\[
v_i = (-\Delta_g)^i v
\] 
for $i =0,1, ... , k$. Clearly $v_0, v_1, ..., v_{k-1}$ solve
\begin{subequations}\label{prob-Ps_i}
\begin{align}
\begin{cases}
-\Delta_g v_i = v_{i+1}&\mbox{ in }\Omega^\sharp,\\
\quad\quad v_i = 0 &\mbox{ on }\partial \Omega^\sharp.
\end{cases}
\tag*{(\ref{prob-Ps_i}$_i^\sharp$)}
\end{align}
\end{subequations}
Then we have the following comparison result.

\begin{proposition}\label{evenordercomparison}
Suppose that $u$ and $v$ are weak solutions to \eqref{eq:DirichletproblemHO} and \eqref{eq:DirichletproblemrearrangementHO}, respectively. Then for any $i = 0,1, ... ,k-1$ there holds
\[
u_i^\sharp(x) \leqslant v_i(x) 
\]
everywhere in $\Omega^\sharp$.
\end{proposition}

We also establish the following estimate for the rearrangement function of solutions to \eqref{eq:Dirichletproblem}, which is a hyperbolic analogue of \cite[Proposition 3.4]{MS2014} and is a crucial tool in the proof of Theorem \ref{Maintheorem}.

\begin{proposition}\label{estimateforsolution}
Any weak solution $u$ to \eqref{eq:Dirichletproblem} enjoys the following a prior estimate
\[
u^*(t_1) -u^*(t_2) \leqslant \frac{1}{(n \Omega_n^{1/n})^2} \int_{t_1}^{t_2} \frac{f^{**}(s)}{s^{1-2/n}} ds
\]
for any $0 < t_1 < t_2 \leqslant |\Omega|$.
\end{proposition}

To prove Theorem \ref{Tarsihyperbolic}, we also need an estimate for the arrangement function of solutions to problem \eqref{eq:DirichletproblemHO}, which is a higher order version of Proposition \ref{estimateforsolution}. To be precise, we will prove the following result.

\begin{proposition}\label{estimateforsolution1}
Let $n>2k \geqslant 2$ and $u$ be a weak solution to problem \eqref{eq:DirichletproblemHO}. Then there holds
\[
u^*(t) \leqslant \frac n{n-2k}\frac{c_{n,k}}{(n \Omega_n^{1/n})^{2k}} \int_{t}^{|\Omega|} \frac{f^{*}(s)}{s^{1-2k/n}} ds + \frac{c_{n,k+1}}{(n \Omega_n^{1/n})^{2k}} t^{2k/n-1} \int_0^t f^*(s) ds,
\]
where
\[
c_{n,k} = 
\begin{cases}
\dfrac{n^{2(k-1)}}{2^{k-1}(k-1)! \prod\nolimits_{j=1}^{k-1}(n-2j)} & \text{ if } k\geqslant 2,\\
1, & \text{ if } k=1.
\end{cases}
\]
\end{proposition}

It is worth noting that $c_{n,k} = 2 (n-2k) \Gamma(n/2 -k) /(4^k \Gamma(n/2)\Gamma (k))$ when $k \geqslant 2$. As mentioned before, the rest of this section is devoted to proofs of Propositions \ref{comparison}, \ref{evenordercomparison}, \ref{estimateforsolution}, and \ref{estimateforsolution1}. 

\subsection{Proofs of Propositions \ref{comparison}, \ref{evenordercomparison}, \ref{estimateforsolution}, and \ref{estimateforsolution1}}

First, we prove Proposition \ref{comparison}. 

\begin{proof}[\bf Proof of Proposition \ref{comparison}] 
Our proof follows closely the argument in \cite{Tal1976}. For fixed $t, h > 0$, we apply H\"older's inequality to get
\[
\frac1h \int_{\{t < |u| \leqslant t+h\}} |\nabla_g u| dV_g \leqslant \bigg (\frac1h \int_{\{t < |u| \leqslant t+h\}} |\nabla_g u|^2 dV_g\bigg )^{1/2} \Big (\frac{\mu_u(t) -\mu_u(t+h)} h \Big )^{1/2}.
\]
Letting $h\searrow 0$, we obtain
\begin{equation}\label{eq:abcd}
-\frac d{dt} \int_{\{|u| >t\}} |\nabla_g u|_g dV_g \leqslant \bigg(-\frac d{dt} \int_{\{|u| >t\}} |\nabla_g u|_g^2 dV_g\bigg)^{1/2} (-\mu_u'(t))^{1/2}.
\end{equation}
Using the co-area formula, we deduce that
\begin{align*}
\int_{\{|u|>t\}} |\nabla_g u|_g dV_g &= \int_{\{|u| >t\}} |\nabla u| \Big (\frac2{1-|x|^2} \Big )^{n-1} dx\\
&= \int_{\{|s|>t\}} \int_{\{u =s\}} \Big (\frac2{1-|x|^2} \Big )^{n-1} d\mathcal H^{n-1}(x) ds,
\end{align*}
where $d\mathcal H^{n-1}(x)$ denotes the $(n-1)$-dimensional Hausdorff measure. Consequently, for almost everywhere $t >0$, we obtain
\[
-\frac d{dt} \int_{\{|u| >t\}} |\nabla_g u|_g dV_g = \int_{\{|u|=t\}} \Big (\frac2{1-|x|^2} \Big )^{n-1} d\mathcal H^{n-1}(x).
\]
For each $t>0$, let $\rho(t)$ denote the radius of the ball centered at origin having hyperbolic volume $\mu_u (t)$, namely, $|B(0,\rho(t))| =\mu_u(t)$. Applying the isoperimetric inequality in hyperbolic space \cite{BDS2015} and in view of \eqref{VolumeSphereInHyperbolic}, we obtain
\begin{align*}
\int_{\{|u|=t\}} \Big (\frac2{1-|x|^2}\Big )^{n-1} d\mathcal H^{n-1}(x) &\geqslant \int_{\partial B(0,\rho(t))} \Big (\frac2{1-|x|^2}\Big )^{n-1} d\mathcal H^{n-1}(x)\\
& = n \Omega_n \sinh^{n-1} \rho(t).
\end{align*}
On the other hand, from \eqref{VolumeBallInHyperbolic} we have
\begin{align*}
\mu_u(t) = \big| B(0,\rho(t)) \big| &=n\Omega_n \int_0^{\rho(t)} \sinh^{n-1}(s) ds.
\end{align*}
Hence there exists a continuous, strictly increasing function $F$ such that
\[
\rho(t) = F(\mu_u(t)).
\] 
Consequently, we obtain from \eqref{eq:abcd} the following estimate
\begin{equation}\label{eq:ABCD12}
1 \leqslant \frac{-\mu_u'(t)}{[n\Omega_n \sinh^{n-1} F(\mu_u(t))]^2} \Big(-\frac{d}{dt} \int_{\{|u| >t\}} |\nabla_g u|_g^2 dV_g\Big).
\end{equation}
For fixed $t, h >0$, let us define the test function
\begin{equation*}
\phi(x) = 
\begin{cases}
0 &\mbox{if $|u|\leqslant t$},\\
(|u|-t) \sign(u)&\mbox{if $t < |u| \leqslant t+h$},\\
h\sign(u)&\mbox{if $|u| > t+h$.}
\end{cases}
\end{equation*}
Clearly $\phi \in W_0^{1,2}(\Omega)$ and 
$\int_\Omega \langle \nabla_g u, \nabla_g \phi\rangle_g dV_g = \int_\Omega f\phi dV_g$ 
since $u$ is weak solution to \eqref{eq:Dirichletproblem}. An easy computation shows that
\begin{align*}
\int_{\{t < |u| \leqslant t+h\}}|\nabla_g u|_g^2 dV_g =& \int_\Omega \langle \nabla_g u,\nabla_g\phi\rangle_g dV_g\\
=& \int_{\{t < |u| \leqslant t+h\}}f(|u|-t)\sign(u) dV_g  + \int_{\{|u| >t+h\}} fh \sign(u) dV_g\\
=& \int_{\{t < |u|\}}f(|u|-t)\sign(u) dV_g \\
& -\int_{\{t+h < |u|\}}f(|u|-t-h)\sign(u) dV_g.
\end{align*}
Dividing both sides by $h$ then letting $h\searrow 0$ in the resulting equation, and using the Hardy--Littlewood inequality, we obtain
\begin{equation}\label{eq:abcdef}
\begin{split}
-\frac d{dt}\int_{\{|u|>t\}} |\nabla_g u|_g^2 dV_g& = -\frac d{dt} \int_{\{|u|>t\}} f(|u|-t)\sign(u) dV_g \\
& = \int_{\{|u|>t\}} f \sign(u) dV_g \\
&\leqslant \int_{\{|u|>t\}} |f| dV_g \\
&\leqslant \int_0^{\mu_u(t)} f^{*}(s) ds = \mu_u(t) f^{**}(\mu_u(t)).
\end{split}
\end{equation}
Plugging \eqref{eq:abcdef} into \eqref{eq:ABCD12} and integrating the resulting over $(s',s)$ to get
\begin{equation}\label{eq:keykey}
\begin{split}
s-s' &\leqslant \int_{s'}^s \frac{-\mu_u'(t)}{[n\Omega_n \sinh^{n-1} F(\mu_u(t))]^2} \mu_u(t) f^{**}(\mu_u(t)) dt \\
&=\int_{\mu_u(s)}^{\mu_u(s')}\frac1{[n\Omega_n \sinh^{n-1}F(t)]^2} tf^{**}(t) dt.
\end{split}
\end{equation}
Letting $s'\searrow 0$ in \eqref{eq:keykey} we obtain
\[
s \leqslant \int_{\mu_u(s)}^{|\Omega|}\frac1{[n\Omega_n \sinh^{n-1}F(t)]^2} tf^{**}(t) dt.
\]
For any $t \in (0, |\Omega|)$, if $u^*(t) > 0$, then for any $0 < s < u^*(t)$ we must have $\mu_u(s) > t$ by the definition of the rearrangement function. Therefore
\[
s \leqslant \int_{t}^{|\Omega|}\frac1{[n\Omega_n \sinh^{n-1}F(r)]^2} rf^{**}(r) dr.
\]
Letting $s\nearrow u^*(t)$ we get
\[
u^*(t) \leqslant \int_{t}^{|\Omega|}\frac1{[n\Omega_n \sinh^{n-1}F(r)]^2} rf^{**}(r) dr.
\]
It is obvious that if $u^*(t) =0$, then the inequality above is true. Hence for any $t\in (0,|\Omega|)$ we have
\[
u^*(t) \leqslant \int_{t}^{|\Omega|}\frac1{[n\Omega_n \sinh^{n-1}F(r)]^2} rf^{**}(r) dr.
\]
It is easy to verify that 
\[
v(x) = \int_{|B(0,d(0,x))|}^{|\Omega|}\frac1{[n\Omega_n \sinh^{n-1}F(r)]^2} rf^{**}(r) dr
\]
is unique solution to \eqref{eq:Dirichletproblemrearrangement}. The inequality $u^\sharp \leqslant v$ obviously holds true, hence the proof of Proposition \ref{comparison} is finished.
\end{proof}

Now we prove Proposition \ref{evenordercomparison} by applying consecutively Proposition \ref{comparison} and the maximum principle. 

\begin{proof}[\bf Proof of Proposition \ref{evenordercomparison}] 
Observe that $u_k^\sharp = v_k$. Then by making use of Proposition \ref{comparison} we obtain from (\ref{prob-P_i}$_{k-1}$) and (\ref{prob-Ps_i}$_{k-1}^\sharp$) the following 
$$u_{k-1}^\sharp \leqslant v_{k-1}$$ 
in $\Omega^\sharp$. Now we argue by an induction argument. Suppose that for some $1\leqslant i < k$ we already have $u_{k-i}^\sharp \leqslant v_{k-i}$ in $\Omega^\sharp$. Then we have to show that $u_{k-i-1}^\sharp \leqslant v_{k-i-1}$ in $\Omega^\sharp$. Indeed, consider the problem
\begin{equation*}
\begin{cases}
-\Delta_g \omega = u_{k-i}^\sharp &\mbox{ in } \Omega^\sharp,\\
\quad\quad \omega = 0&\mbox{ on }\partial \Omega^\sharp.
\end{cases}
\end{equation*}
Again by applying Proposition \ref{comparison} we obtain from the preceding problem for $\omega$ and (\ref{prob-P_i}$_{k-i-1}$) the following 
$$u_{k-i-1}^\sharp \leqslant \omega.$$
Recall that $u_{k-i}^\sharp \leqslant v_{k-i}$. From this we can apply the maximum principle to (\ref{prob-Ps_i}$_{k-i-1}$) to get $\omega \leqslant v_{k-i-1}$. Therefore, $u_{k-i-1}^\sharp \leqslant v_{k-i-1}$ and the proof follows. 
\end{proof}

Then we show that Proposition \ref{estimateforsolution} follows from the proof of Proposition \ref{comparison}. 

\begin{proof}[\bf Proof of Proposition \ref{estimateforsolution}] 
Using the simple inequality $\cosh s \geqslant 1$, the definition of $\rho (t)$, and \eqref{VolumeBallInHyperbolic}, it is evident that
\begin{align*}
\mu_u(t)  &\leqslant n\Omega_n\int_0^{\rho(t)} \sinh^{n-1}(s) \cosh(s) ds = \Omega_n \sinh^n F(\mu_u(t)).
\end{align*}
Hence, we obtain
\begin{equation}\label{eq:lowerestF}
\sinh^{n-1} F(r) \geqslant \Big(\frac r{\Omega_n}\Big)^{1-1/n}.
\end{equation}
Combining \eqref{eq:lowerestF} and \eqref{eq:keykey} gives 
\[
s-s' \leqslant \frac1{[n\Omega_n^{1/n}]^2}\int_{\mu_u(s)}^{\mu_u(s')} \frac{f^{**}(t)}{t^{1-2/n}} dt.
\]
Now, let $0 < t_1 < t_2 \leqslant |\Omega|$. If $u^*(t_1) = u^*(t_2)$, then the conclusion is trivial. If $u^*(t_1) > u^*(t_2)$, then for any $s$, $s'$ such that $u^*(t_2) < s' < s < u^*(t_1)$, by the definition of rearrangement function, we obviously have $\mu_u(s) > t_1$ and $\mu_u(s') \leqslant t_2$. Then we have
\[
s-s' \leqslant \frac1{(n\Omega_n^{1/n})^2}\int_{\mu_u(s)}^{\mu_u(s')} \frac{f^{**}(t)}{t^{1-2/n}} dt\leqslant \frac1{(n\Omega_n^{1/n})^2}\int_{t_1}^{t_2} \frac{f^{**}(t)}{t^{1-2/n}} dt.
\]
Letting $s\nearrow u^*(t_1)$ and $s'\searrow u^*(t_2)$ implies our desired inequality.
\end{proof}

Finally, we can easily prove Proposition \ref{estimateforsolution1} by applying consecutively Proposition \ref{estimateforsolution}.

\begin{proof}[\bf Proof of Proposition \ref{estimateforsolution1}] 
If $k=1$, then by Proposition \ref{estimateforsolution} we have
\[
u^*(t) \leqslant \frac1{n^2 \Omega_n^{2/n}} \int_t^{|\Omega|} \frac{f^{**}(s)}{s^{1-2/n}} ds = \frac1{n^2 \Omega_n^{2/n}} \int_t^{|\Omega|} \Big(\int_0^s f^*(r) dr\Big) s^{2/n-2} ds.
\]
Integration by parts then gives our desired estimate. If $k\geqslant 2$, then by denoting $u_k =f$ we have from Proposition \ref{estimateforsolution} that
\[
u_i^*(t) \leqslant \frac1{n^2 \Omega_n^{2/n}} \int_t^{|\Omega|} \frac{u_{i+1}^{**}(s)}{s^{1-2/n}} ds
\]
for all $i=0,1, ... ,k-1$. From this, the definition of the maximal function, and Fubini's theorem we conclude that
\[
u_i^{**}(t) = \frac 1 t \int_0^t u_i^*(s) ds  \leqslant \frac1{n^2 \Omega_n^{2/n}} \int_0^{|\Omega|} g(t,s) u_{i+1}^{**}(s) ds,
\]
where
\[
g(t,s) = 
\begin{cases}
t^{-1} s^{2/n}&\mbox{if $s < t$},\\
s^{2/n-1}&\mbox{if $s \geqslant t$}.
\end{cases}
\]
This helps us to conclude that
\begin{equation}\label{eq:ReasonToInduction}
u_{i-1}^*(t) \leqslant \frac1{(n \Omega_n^{1/n})^4} \int_t^{|\Omega|} r^{2/n-1} \Big( \int_0^{|\Omega|} g(r,s) u_{i+1}^{**}(s) ds \Big)  dr
\end{equation}
for all $i=1,2, ... ,k-1$. Now we consecutively define a sequence of functions $\{G_j\}_{j\geqslant 1}$ as follows:
\[
\begin{cases}
G_1(t,s) = g(t,s), & \\
G_i(t,s) = \displaystyle\int_0^{|\Omega|} G_{i-1}(t,s') g(s',s) ds' & \text{ for all } i\geqslant 2.
\end{cases}
\]
Letting $i=k-1$ in \eqref{eq:ReasonToInduction} we arrive at
\[
u_{k-2}^*(t) \leqslant \frac1{(n\Omega_n^{1/n})^{4}} \int_t^{|\Omega|} r^{2/n -1}\Big( \int_0^{|\Omega|} G_1 (r,s) f^{**}(s) ds \Big) dr.
\]
By repeating the calculation leading to \eqref{eq:ReasonToInduction}, we can prove by induction that
\[
u^*(t) \leqslant \frac1{(n\Omega_n^{1/n})^{2k}} \int_t^{|\Omega|} r^{2/n -1}\Big( \int_0^{|\Omega|} G_{k-1}(r,s) f^{**}(s) ds \Big) dr.
\]
Choose $R > 0$ such that $R^n \Omega_n = |\Omega|$. For $x\in B_R$, let us define
\[
g(x) = f^*(\Omega_n |x|^n)
\] 
and
\[
v(x) = \frac1{(n\Omega_n^{1/n})^{2k}} \int_{ \Omega_n|x|^n}^{|\Omega|} r^{2/n -1} \Big( \int_0^{|\Omega|} G_{k-1}(r,s) f^{**}(s) ds \Big) dr.
\]
Then the rearrangement function of $g$, being considered in $\R^n$, satisfies $g^* = f^*$ and
\[
v^*(t) = \frac1{(n\Omega_n^{1/n})^{2k}} \int_t^{|\Omega|} r^{2/n -1} \Big( \int_0^{|\Omega|} G_{k-1}(r,s) f^{**}(s) ds \Big) dr.
\]
A straightforward computation shows that
\[
\begin{cases}
(-\Delta)^k v = g & \text{ in } B_R, \\
\Delta^i v\big|_{ \partial B_R}  =0 & \text{  for } 0 \leqslant i  \leqslant k-1.
\end{cases}
\]
Now we extend $g$ to $\widetilde g$ in such a way that
\[
\widetilde g (x) = 
\begin{cases}
g(x) & \text{ in } B_R, \\
0 & \text{  in } \R^n \setminus B_R.
\end{cases}
\]
Recall that $n>2k$ and that Green's function of $(-\Delta)^k$ in $\Rset^n$ is 
$$ \frac{c_{n,k}}{(n-2k)n^{2k-1}\Omega_n}         |x-y|^{2k-n},$$ 
where the constant $c_{n,k}$ is as in the statement of the proposition. Therefore, if we define 
\[
w(x) = \frac{c_{n,k}}{(n-2k)n^{2k-1}\Omega_n}\int_{\R^n} |x-y|^{2k-n} \widetilde g(y) dy,
\]
then it is easy to verify that 
\[
(-\Delta)^k w = \widetilde g
\]
in $\R^n$. Furthermore, as $w$ is being expressed in terms of Riesz's potential, it is not hard to compute $(-\Delta)^i w$ to get
\[
(-\Delta)^i w \geqslant 0
\]
in $\R^n$ for all $1 \leqslant i \leqslant k-1$. Therefore, limiting ourselves to $B_R$ we obtain
\[
\begin{cases}
(-\Delta)^k w = g & \text{ in } B_R, \\
(-\Delta)^i w \big|_{ \partial B_R} \geqslant 0 & \text{ for } 0 \leqslant i  \leqslant k-1.
\end{cases}
\]
By a finite induction with a help from the maximum principle, we obtain $v(x) \leqslant w(x)$ for $x\in B_R$. Equivalently, there holds $v^*(t) \leqslant w^*(t) $ for any $t\in (0, |\Omega|)$. On the other hand, it follows from a result due to O'Neil \cite{O'Neil1963} that
\[
\begin{split}
w^{**} (t) \leqslant &\frac{c_{n,k}}{(n-2k)n^{2k-1}\Omega_n}
\left( 
\begin{split}
&\frac1t\int_0^t\Big(\frac{\Omega_n}s\Big)^{1-2k/n} ds \int_0^t g^*(s) ds \\
&+ \int_t^{+\infty} g^*(s) \Big(\frac{\Omega_n}s\Big)^{1-2k/n} ds
\end{split}
\right),
\end{split}
\]
which implies that
\begin{equation}\label{eq:wstarestimate}
\begin{split}
w^*(t) \leqslant &\frac{n}{n-2k} \frac{c_{n,k}}{(n\Omega_n^{1/n})^{2k}} \int_t^{|\Omega|} f^*(s) s^{2k/n -1} ds + \frac{c_{n,k+1}}{(n\Omega_n^{1/n})^{2k}}t^{2k/n-1} \int_0^t f^*(s) ds,
\end{split}
\end{equation}
since $g^*(s) =f^*(s)$ and $g^*(s) =0$ for $s > |\Omega|$. From this, Proposition \ref{estimateforsolution1} follows from \eqref{eq:wstarestimate} and the estimates $u^* \leqslant v^* \leqslant w^*$.
\end{proof}


\section{Adams inequality with exact growth: Proof of Theorem \ref{Maintheorem}}
\label{sec-AdamsWithExactGrowth}

In this section, we prove Theorem \ref{Maintheorem} by following the same lines as  in \cite{LTZ2015,MS2014}. In the following subsection, we introduce some crucial tools which shall be used in our proof. However, we first recall some widely used symbols. Here and in what follows, by $\lesssim$ and $\gtrsim$ we mean inequalities up to uniform and dimensional constants. If both $\lesssim$ and $\gtrsim$ occur, then we use the symbol $\sim$.

\subsection{Some crucial lemmas}

First, we recall following lemma whose proof can be found in \cite[Lemma 4.2]{LTZ2015}. 

\begin{lemma}\label{crucial2}
Given any sequence $a = (a_k)_{k\geqslant 0}$ and any $p > 1$ let us denote
\begin{align*}
\|a\|_1 =& \sum_{k=0}^{+\infty} |a_k|, \quad \|a\|_p = \Big(\sum_{k=0}^{+\infty} |a_k|^p \Big)^{1/p}, \quad \|a\|_{(e)} = \Big(\sum_{k=0}^{+\infty} |a_k|^p e^k\Big)^{1/p},
\end{align*} 
and 
\[
\mu(h) = \inf\{\|a\|_{(e)}\, :\, \|a\|_1 =h, \|a\|_p \leqslant 1\}.
\]
Then we have
\[
\mu(h) \sim \exp \big( h^{p/(p-1)}/p \big) h^{-1/(p-1)}
\]
for $h >1$
\end{lemma}

Our first crucial lemma is the following.

\begin{lemma}\label{crucial1}
Let $p >1$ and let $u,f \in L^p((0, +\infty))$ be non-negative and decreasing functions such that
\begin{equation}\label{eq:giathiet}
u(t_1) -u(t_2) \leqslant c \int_{t_1}^{t_2} \frac{f(s)}{s^{1 -1/p}} ds
\end{equation}
for any $0<t_1 < t_2$ and $c$ is a positive constant. If $u(A) >1$ and
\[
\int_{A}^{+\infty} f(s)^{p} ds \leqslant \Big(\frac p{p-1}\Big)^{p},
\]
then
\[
\frac{\exp\Big(\Big(\frac{p-1}{cp}\Big)^{p/(p-1)} u(A)^{p/(p-1)} \Big)}{(u(A))^{p/(p-1)}} A \lesssim \int_A^{+\infty} u(s)^{p} ds.
\]
\end{lemma}

\begin{proof}
Denote $h_k = c_1 u(e^k A)$, where $c_1 = (p-1)/cp$. Define $a_k = h_k -h_{k+1} \geqslant 0$; hence
\[
\sum_{k=0}^{+\infty} |a_k| = h_0 = c_1 u(A).
\]
On one hand, it follows from \eqref{eq:giathiet} and H\"older's inequality that
\begin{align*}
a_k &= c_1(u(e^k A) - u(e^{k+1}A))\leqslant \frac{p-1}p \bigg(\int_{e^k A}^{e^{k+1}A} f(s)^{p} ds \bigg)^{1/p}.
\end{align*}
Consequently, we have
\[
\sum_{k=0}^{+\infty} |a_k|^{p} \leqslant \Big(\frac{p-1}p \Big)^{p} \int_{A}^{+\infty} f(s)^{p} ds \leqslant 1.
\]
On the other hand
\begin{align*}
\frac 1A \int_A^{+\infty} u(s)^{p} ds =&\sum_{k=0}^{+\infty} \frac 1A \int_{e^kA}^{e^{k+1}A} u(s)^{p} ds \\
\geqslant & \sum_{k=0}^{+\infty} u(e^{k+1})^{p}e^k(e-1) \\
\geqslant & \frac{e-1}e\sum_{k=1}^{+\infty} a_k^{p} e^k.
\end{align*}
Therefore,
\begin{equation}\label{eq:a111}
\|a\|_{(e)}^{p} = a_0^{p} +\sum_{k=1}^{+\infty} a_k^{p} e^k \lesssim h_0^{p} + \frac 1A \int_A^{+\infty} u(s)^{p} ds.
\end{equation}
Next we estimate $h_0$. To do this, we choose $ b = (c_1/2)^{p/(p-1)}$; hence for any $1\leqslant r\leqslant e^b$, we have
\begin{align*}
h_0 -c_1 u(r A)&\leqslant \frac{p-1}{p}\int_{A}^{e^bA} \frac{f(s)}{s^{1-1/p}} ds\\
&\leqslant \frac{p-1}p \bigg(\int_{A}^{e^bA} f(s)^{p} ds\bigg)^{1/p} b^{1-1/p}\\
&\leqslant \frac{c_1}2 \leqslant \frac{h_0}2,
\end{align*}
here we have used the inequality $u(A) > 1$. From this, we easily get $h_0 \leqslant 2 c_1 u(rA)$ for any $1\leqslant r \leqslant e^b$. Therefore,
\begin{equation}\label{eq:s111}
\frac 1A \int_A^{+\infty} u(s)^{p} ds \geqslant \frac 1A \int_A^{e^b A} u(s)^{p} ds \geqslant \Big(\frac{h_0}{2c_1}\Big)^{p} (e^b -1)\gtrsim h_0^{p}.
\end{equation}
Combining \eqref{eq:s111} and \eqref{eq:a111} gives
\[
\|a\|_{(e)}^{p} \lesssim \frac 1A \int_A^{+\infty} u(s)^{p} ds.
\]
By Lemma \ref{crucial2}, we obtain
\[
\|a\|_{(e)}^{p} \gtrsim h_0^{-p/(p-1)} \exp(h_0^{p/(p-1)}) \gtrsim (u(A))^{-p/(p-1)} \exp \big( (c_1 u(A))^{p/(p-1)} \big).
\]
This completes the proof of Lemma \ref{crucial1}.
\end{proof}

Next our second crucial lemma is the following well-known lemma due to Adams, which plays a crucial role in \cite{Adams1988}.

\begin{lemma}\label{Adamslemma}
Let $p > 1$ and $p' =p/(p-1)$. Let also $a(s,t)$ be a nonnegative measurable function on $\R \times [0, +\infty)$ such that $a(s,t) \leqslant 1$ for $0< s < t$ and
\[
\sup_{t >0} \bigg(\int_{-\infty}^0 + \int_t^{+\infty} \bigg)a(s,t)^{p'} ds ^{1/p'} = b < +\infty.
\]
Then there exists a constant $c_0$ depending only on $p$ and $b$ such that
\[
\int_0^{+\infty} e^{-F(t)} dt < c_0
\]
for any non-negative function $\phi$ satisfying $\int_\R \phi(t)^p dt \leqslant 1$ with
\[
F(t) = t - \Big(\int_\R a(s,t) \phi(s) ds \Big)^{p'}.
\]
\end{lemma}

\subsection{Proof of Theorem \ref{Maintheorem}}

We are now in a position to prove Theorem \ref{Maintheorem}. For clarity, we divide our proof into several parts located in a few subsubsections below.

\subsubsection{Proof of \eqref{eq:Adamsexact}}

Using a density argument, we only need to prove Theorem \ref{Maintheorem} for functions in $C_0^\infty(\H^n)$. By the property of rearrangement, we have
\[
\int_{\H^n} \frac{\Phi_{n,2}(\beta(n,2) |u|^{n/(n-2)})}{(1+|u|)^{n/(n-2)}} dV_g = \int_{\H^n} \frac{\Phi_{n,2}(\beta(n,2) |u^\sharp|^{n/(n-2)})}{(1+|u^\sharp|)^{n/(n-2)}} dV_g
\]
and
\[
\|u\|_{n/2}^{n/2} = \|u^\sharp\|_{n/2}^{n/2}.
\]
Therefore, it suffices to prove that
\begin{equation}\label{eq:usharp}
\int_{\H^n} \frac{\Phi_{n,2}(\beta(n,2) |u^\sharp|^{n/(n-2)})}{(1+|u^\sharp|)^{n/(n-2)}} dV_g   \lesssim   \|u^\sharp\|_{n/2}^{n/2}.
\end{equation}
To this purpose, we will split the integral appearing in \eqref{eq:usharp} into two parts as done in \cite{LTZ2015,MS2014} as follows
\[\begin{split}
\int_{\H^n} & \frac{\Phi_{n,2}(\beta(n,2) |u^\sharp|^{n/(n-2)})}{(1+|u^\sharp|)^{n/(n-2)}} dV_g \\
&= \bigg( \int_{B(0,R_0)} + \int_{\H^n\setminus B(0,R_0)} \bigg) \frac{\Phi_{n,2}(\beta(n,2) |u^\sharp|^{n/(n-2)})}{(1+|u^\sharp|)^{n/(n-2)}} dV_g,
\end{split}\]
where 
\[
R_0 = \inf\{r \geqslant 0\, :\, u^*(|B(0,r)|) \leqslant 1\} \in [0, +\infty).
\] 
Our aim is to estimate the two integrals term by term. To estimate the integral $\int_{\H^n\setminus B(0,R_0)}$, we observe that 
\[
\left\{
\begin{gathered}
u^*(|B(0,r)|) > 1 \quad \text{ when } r < R_0, \hfill \\
u^*(|B(0,R_0)|) = 1, \hfill \\
u^*(|B(0,r)|)\leqslant 1 \quad \text{ when } r > R_0. \hfill 
\end{gathered}
\right.
\]
Since $\Phi_{n,2}(\beta(n,2)x^{n/(n-2)}) \leqslant C x^{n/2}$ for $0\leqslant x\leqslant 1$, we conclude that
\begin{equation}\label{eq:outsideball}
\begin{split}
\int_{\H^n\setminus B(0,R_0)} & \frac{\Phi_{n,2}(\beta(n,2) |u^\sharp|^{n/(n-2)})}{(1+|u^\sharp|)^{n/(n-2)}} dV_g  \lesssim   \int_{\H^n\setminus B(0,R_0)} (u^{\sharp})^{n/2} dV_g \lesssim  \|u\|_{n/2}^{n/2}.
\end{split}
\end{equation}
We next consider the integral $\int_{B(0,R_0)}$. For simplicity, we denote 
$$f = -\Delta_g u$$ 
in $\H^n$ and 
$$\alpha = \int_0^{+\infty} (f^{**}(s))^{n/2} ds.$$ 
Clearly $f \in L^{n/2} (\H^n)$. Then by Lemma \ref{Hardy}, we obtain
\[
\alpha \leqslant \Big(\frac n{n-2}\Big)^{n/2} \int_0^{+\infty} [f^*(s)]^{n/2} ds = \Big(\frac n{n-2}\Big)^{n/2} \|\Delta_g u\|_{n/2}^{n/2} \leqslant \Big(\frac n{n-2}\Big)^{n/2}.
\]
Fix $\epsilon_0 \in (0, 1)$ and choose $R_1$ in such a way that
\[
\int_{0}^{|B(0,R_1)|} [f^{**}(s)]^{n/2} ds \leqslant \alpha \epsilon_0,\quad \int_{|B(0,R_1)|}^{+\infty} [f^{**}(s)]^{n/2} ds \leqslant \alpha (1-\epsilon_0).
\]
By applying Proposition \ref{estimateforsolution} and H\"older's inequality, we have
\[
u^*(t_1) -u^*(t_2) \leqslant \frac1{(n \Omega_n^{1/n})^2} \Big(\int_{t_1}^{t_2} [f^{**}(s)]^{n/2} ds\Big)^{2/n} \Big(\log \frac{t_2}{t_1}\Big)^{1 -2/n},
\]
for any $0 < t_1 < t_2$. Therefore,
\begin{equation}\label{eq:insideballR1}
u^*(|B(0,r_1)|) -u^*(|B(0,r_2)|) \leqslant \frac {(\alpha \epsilon_0)^{2/n}} {(n \Omega_n^{1/n})^2} \Big(\log \frac{|B(0,r_2)|}{|B(0,r_1)|}\Big)^{1 -2/n}
\end{equation}
for any $0< r_1 < r_2 < R_1$ and
\begin{equation}\label{eq:outsideballR1}
u^*(|B(0,r_1)|) -u^*(|B(0,r_2)|) \leqslant \frac {(\alpha (1-\epsilon_0))^{2/n}} {(n \Omega_n^{1/n})^2} \Big(\log \frac{|B(0,r_2)|}{|B(0,r_1)|}\Big)^{1 -2/n}
\end{equation}
for any $r_2 > r_1 > R_1$. In order to estimate the integral $\int_{B(0,R_0)}$, we need to consider the two cases: $R_1 \geqslant R_0$ and $R_1 < R_0$.

\smallskip\noindent{\bf Case 1: Suppose $R_1\geqslant R_0$.} By \eqref{eq:insideballR1}, we obtain
\[
u^*(|B(0,r)|) \leqslant 1 + \frac {(\alpha \epsilon_0)^{2/n}}{(n \Omega_n^{1/n})^2} \Big(\log \frac{|B(0,R_0)|}{|B(0,r)|}\Big)^{1 -2/n}
\]
for any $0< r\leqslant R_0$. For $\epsilon > 0$ to be determined later, by using the elementary inequality
$
(1+ s^{(n-2)/n})^{n/(n-2)} \leqslant (1+\epsilon) s +C_\epsilon
$ 
for $s \geqslant 0$ with $C_\epsilon  = [1-(1+\epsilon )^{-(n-2)/2}]^{-2/(n-2)}$, we get
\[
[u^*(|B(0,r)|)]^{n/(n-2)} \leqslant (1+\epsilon ) \frac{(\alpha\epsilon _0)^{2/(n-2)}}{(n\Omega_n^{1/n})^{2n/(n-2)}}\log \frac{|B(0,R_0)|}{|B(0,r)|} + C_\epsilon .
\]
We now choose $\epsilon  = 1 -\epsilon _0^{2/(n-2)}$. Clearly, $(1+\epsilon ) \epsilon _0^{2/(n-2)} <1$. Since $\alpha \leqslant (n/(n-2))^{n/2}$ and $$\beta(n,2) (n\Omega_n^{1/n})^{-2n/(n-2)} = ((n-2)/n )^{n/(n-2)},$$ we know that
\begin{align}\label{eq:insideballcase1}
\int_{B(0,R_0)} & \frac{\Phi_{n,2}(\beta(n,2) |u^\sharp|^{n/(n-2)})}{(1+|u^\sharp|)^{n/(n-2)}} dV_g \notag\\
& \leqslant \int_{B(0,R_0)} \exp(\beta(n,2) |u^\sharp|^{n/(n-2)}) dV_g\notag\\
&=n\Omega_n \int_0^{R_0} \exp(\beta(n,2) |u^*(|B(0,r)|)|^{n/(n-2)}) \sinh^{n-1}(r) dr\notag\\
&\leqslant e^{\beta(n,2) C_\epsilon } n\Omega_n \int_0^{R_0} \exp \Big( (1+\epsilon )\epsilon _0^{2/(n-2)} \log \frac{|B(0,R_0)|}{|B(0,r)|}\Big) \sinh^{n-1}(r) dr\notag\\
&= e^{\beta(n,2) C_\epsilon } |B(0,R_0)|^{(1+\epsilon )\epsilon _0^{2/(n-2)}} \int_0^{|B(0,R_0)|} s^{-(1+\epsilon )\epsilon _0^{2/(n-2)}} ds\notag\\
&\lesssim |B(0,R_0)|\notag\\
&\lesssim n\Omega_n \int_0^{R_0} u^*(|B(0,r)|)^{n/2} \sinh^{n-1}(r) dr\notag\\
&\lesssim \|u\|_{n/2}^{n/2}.
\end{align}
From this we get the desired inequality when $R_1 \geqslant R_0$, thanks to \eqref{eq:outsideball} and \eqref{eq:insideballcase1}.

\smallskip\noindent{\bf Case 2: Suppose $R_1 < R_0$.} We split the integral $\int_{B(0,R_0)}$ into two parts as follows
\begin{align*}
\int_{B(0,R_0)} &\frac{\Phi_{n,2}(\beta(n,2) |u^\sharp|^{n/(n-2)})}{(1+|u^\sharp|)^{n/(n-2)}} dV_g \\
& = \bigg( \int_{B(0,R_0)\setminus B(0,R_1)} + \int_{B(0,R_1)} \bigg) \frac{\Phi_{n,2}(\beta(n,2) |u^\sharp|^{n/(n-2)})}{(1+|u^\sharp|)^{n/(n-2)}} dV_g.
\end{align*}
An estimate for the integral on $B(0,R_0)\setminus B(0,R_1)$ is easy to get. In fact, by the inequality \eqref{eq:outsideballR1}, we have
\[
u^*(|B(0,r)|) \leqslant 1 + \frac {(\alpha(1- \epsilon_0))^{2/n}} {(n \Omega_n^{1/n})^2} \Big(\log \frac{|B(0,R_0)|}{|B(0,r)|}\Big)^{1 -2/n}
\]
for any $R_1 < r < R_0$. Let $\epsilon _1 = 1 -(1-\epsilon _0)^{2/(n-2)}$. Clearly, $(1+\epsilon _1)(1-\epsilon _0)^{2/(n-2)} < 1$. Similar to Case $1$ above, we have
\[
u^*(|B(0,r)|)^{n/(n-2)} \leqslant (1+\epsilon _1)\frac { (\alpha(1- \epsilon_0))^{2/(n-2)}} {(n \Omega_n^{1/n})^{2n/(n-2)}} \log \frac{|B(0,R_0)|}{|B(0,r)|} + C_{\epsilon _1}.
\]
Hence
\begin{align}\label{eq:insideballcase2}
\int_{B(0,R_0)\setminus B(0,R_1)} & \frac{\Phi_{n,2}(\beta(n,2) |u^\sharp|^{n/(n-2)})}{(1+|u^\sharp|)^{n/(n-2)}} dV_g\notag\\
&\leqslant \int_{B(0,R_0)\setminus B(0,R_1)} \exp(\beta(n,2) [u^\sharp]^{n/(n-2)}) dV_g\notag\\
&\leqslant n\Omega_n \int_{R_1}^{R_0} \exp(\beta(n,2) u^*(|B(0,r)|)^{n/(n-2)})\sinh^{n-1}(r) dr\notag\\
&\lesssim n\Omega_n \int_{R_1}^{R_0} \Big( \frac{|B(0,R_0)|}{|B(0,r)|} \Big) ^{(1+\epsilon _1)(1-\epsilon _0)^{2/(n-2)}} \sinh^{n-1}(r) dr\notag\\
&\lesssim |B(0,R_0)|^{(1+\epsilon _1)(1-\epsilon _0)^{2/(n-2)}} \int_{|B(0,R_1)|}^{|B(0,R_0)|} s^{-(1+\epsilon _1)(1-\epsilon _0)^{2/(n-2)}} ds\notag\\
&\lesssim |B(0,R_0)|\notag\\
&\lesssim n\Omega_n \int_0^{R_0} u^*(|B(0,r)|)^{n/2} \sinh^{n-1}(r) dr\notag\\
&\lesssim \|u\|_{n/2}^{n/2}.
\end{align}
Next we estimate the integral on $B(0,R_1)$. Note that when $0< r < R_1$ we can write
\[
u^*(|B(0,r)|) =\big [u^*(|B(0,r)|) -u^*(|B(0,R_1)|) \big] + u^*(|B(0,R_1)|)
\]
and apply Proposition \ref{estimateforsolution} to get
\begin{align*}
u^*(|B(0,r)|)^{n/(n-2)} \leqslant &(1+\epsilon _2) \big[u^*(|B(0,r)|) -u^*(|B(0,R_1)|) \big]^{n/(n-2)} \\
& + C_{\epsilon _2} \big[u^*(|B(0,R_1)|) \big]^{n/(n-2)}\\
\leqslant & (1+\epsilon _2) \bigg(\frac1{(n\Omega_n^{1/n})^2} \int_{|B(0,r)}^{|B(0,R_1)|} \frac{f^{**}(s)}{s^{1-2/n}} ds\bigg)^{n/(n-2)} \\
& + C_{\epsilon _2} [u^*(|B(0,R_1)|)]^{n/(n-2)}
\end{align*}
for some positive constant $\epsilon_2$ to be specified later. Recall that $\beta(n,2) = (n \Omega_n^{1/n})^{2n/(n-2)} ((n-2)/n)^{n/(n-2)}$. Therefore,
\begin{align*}
\int_{B(0,R_1)} & \frac{\Phi_{n,2}(\beta(n,2) |u^\sharp|^{n/(n-2)})}{(1+|u^\sharp|)^{n/(n-2)}} dV_g\notag\\
\lesssim &\frac{n\Omega_n}{[u^*(|B(0,R_1)|)]^{n/(n-2)}}   \int_0^{R_1} \exp\Big(\beta(n,2)u^*(|B(0,r)|)^{n/(n-2)}\Big) \sinh^{n-1}(r) dr\notag\\
\lesssim &\frac{n\Omega_n \exp\left(C_{\epsilon _2} [u^*(|B(0,R_1)|)]^{n/(n-2)}\right)}{[u^*(|B(0,R_1)|)]^{n/(n-2)}}\notag\\
& \times \int_0^{R_1}\exp\left(\left[
\begin{gathered}
(1+\epsilon _2)^{(n-2)/n} \frac{n-2}n
 \int_{|B(0,r)|}^{|B(0,R_1)|} \frac{f^{**}(s)}{s^{1-2/n}} ds
\end{gathered}
\right]^{n/(n-2)}\right) \sinh^{n-1}(r) dr\notag\\
= &\frac{\exp (C_{\epsilon _2} [u^*(|B(0,R_1)|)]^{n/(n-2)} )}{[u^*(|B(0,R_1)|)]^{n/(n-2)}} \\
& \times \int_0^{|B(0,R_1)|}\exp\left(\left[
\begin{gathered}
(1+\epsilon _2)^{(n-2)/n} \frac{n-2}n \int_{r}^{|B(0,R_1)|} \frac{f^{**}(s)}{s^{1-2/n}} ds
\end{gathered}
\right]^{n/(n-2)}\right) dr.
\end{align*}
Using the change of variables $r = e^{-t} |B(0,R_1)|$, we have
\begin{align}\label{eq:insideballR11}
\int_{B(0,R_1)} & \frac{\Phi_{n,2}(\beta(n,2) |u^\sharp|^{n/(n-2)})}{(1+|u^\sharp|)^{n/(n-2)}} dV_g \nonumber\\
\lesssim& |B(0,R_1)| \frac{\exp\left(C_{\epsilon _2} [u^*(|B(0,R_1)|)]^{n/(n-2)}\right)}{[u^*(|B(0,R_1)|)]^{n/(n-2)}} \nonumber\\
& \times \int_0^{+\infty}\exp\left(\left[
\begin{gathered}
(1+\epsilon _2)^{(n-2)/n} \frac{n-2}n 
\int_{e^{-t}|B(0,R_1)|}^{|B(0,R_1)|} \frac{f^{**}(s)}{s^{1-2/n}} ds
\end{gathered}
\right]^{n/(n-2)}\right) e^{-t} dt \nonumber\\
 \lesssim& |B(0,R_1)| \frac{\exp (C_{\epsilon _2}\beta(n,2) [u^*(|B(0,R_1)|)]^{n/(n-2)} )}{[u^*(|B(0,R_1)|)]^{n/(n-2)}} \nonumber\\
& \times \int_0^{+\infty}\exp\left(\left[
\begin{gathered}
|B(0,R_1)|^{2/n}(1+\epsilon _2)^{(n-2)/n} \frac{n-2}n  \\
\times \int_{0}^{t} f^{**}(|B(0,R_1)|e^{-s}) e^{-\frac{2s}n} ds
\end{gathered}
\right]^{n/(n-2)}\right) e^{-t} dt.
\end{align}
Now define
\[
\varphi(t) = |B(0,R_1)|^{2/n}(1+\epsilon _2)^{(n-2)/n} \frac{n-2}n f^{**}(|B(0,R_1)|e^{-t}) e^{-2t/n} \chi_{\{t >0\}}.
\]
Then by the choice of $R_1$, we get
\begin{align*}
\int_\R \varphi(t)^{n/2}dt &= |B(0,R_1)|(1+\epsilon _2)^{(n-2)/n} \Big(\frac{n-2}n\Big)^{n/2} \int_0^{+\infty} f^{**}(|B(0,R_1)|e^{-t})^{n/2} e^{-t} dt\\
&= (1+\epsilon _2)^{(n-2)/n} \Big(\frac{n-2}n\Big)^{n/2} \int_0^{|B(0,R_1)|} [f^{**}(s)]^{n/2} ds\\
&\leqslant \epsilon _0 (1+\epsilon _2)^{(n-2)/n}.
\end{align*}
We now choose $\epsilon _2 = \epsilon _0^{-2/(n-2)} -1$. Clearly, $\int_\R \varphi(t)^{n/2} dt \leqslant 1$. Setting $a(s,t) = \chi_{(0,t)}(s)$. By \eqref{eq:insideballR11} and Lemma \ref{Adamslemma}, we have 
\begin{align*}
\int_{B(0,R_1)} & \frac{\Phi_{n,2}(\beta(n,2) |u^\sharp|^{n/(n-2)})}{(1+|u^\sharp|)^{n/(n-2)}} dV_g \\
& \lesssim |B(0,R_1)| \frac{\exp (C_{\epsilon _2}\beta(n,2) [u^*(|B(0,R_1)|)]^{n/(n-2)} )}{[u^*(|B(0,R_1)|)]^{n/(n-2)}}.
\end{align*}
Note that $C_{\epsilon _2} = (1-\epsilon _0)^{-2/(n-2)}$, therefore
\begin{align*}
\int_{B(0,R_1)} & \frac{\Phi_{n,2}(\beta(n,2) |u^\sharp|^{n/(n-2)})}{(1+|u^\sharp|)^{n/(n-2)}} dV_g \\
 & \lesssim |B(0,R_1)| \frac{\exp\big((1-\epsilon _0)^{-2/(n-2)}\beta(n,2) [u^*(|B(0,R_1)|)]^{n/(n-2)}\big)}{[u^*(|B(0,R_1)|)]^{n/(n-2)}}
\end{align*}
Recall that 
\[
\int_{|B(0,R_1)|}^{+\infty} [f^{**}(s)]^{n/2} ds \leqslant (n/(n-2))^{n/2} (1-\epsilon _0).
\] 
Applying Lemma \ref{crucial1} to the functions $u^*(1-\epsilon )^{-2/n}$, $f^{**}(1-\epsilon )^{-2/n}$, $p=n/2$, $c =(n\Omega_n^{1/n})^{-2}$, and $A = |B(0,R_1)|$, we then have
\begin{align*}
|B(0,& R_1)| \exp\big((1-\epsilon _0)^{-2/(n-2)} \beta(n,2) [u^*(|B(0,R_1)|)]^{n/(n-2)}\big) \\
 \times & [u^*(|B(0,R_1)|)]^{-n/(n-2)} \lesssim (1-\epsilon _0)^{-n/(n-2)}\int_{|B(0,R_1)|}^{+\infty} [u^*(s)]^{n/2} ds \lesssim \|u\|_{n/2}^{n/2}.
\end{align*}
Therefore, putting these estimates together, we have just shown that
\begin{equation}\label{eq:insideballR12}
\int_{B(0,R_1)} \frac{\Phi_{n,2}(\beta(n,2) |u^\sharp|^{n/(n-2)})}{(1+|u^\sharp|)^{n/(n-2)}} dV_g\lesssim \|u\|_{n/2}^{n/2}.
\end{equation}
Combining \eqref{eq:insideballR12} and \eqref{eq:insideballcase2} finishes our proof of Case $2$, and hence completes our proof of inequality \eqref{eq:Adamsexact}.

\subsubsection{The sharpness of (\ref{eq:Adamsexact})} 

It remains to prove the sharpness of Theorem \ref{Maintheorem}. To see this, let us consider the sequence of functions $\{v_m\}_m$ given as follows
\begin{equation*}
v_m(x) = 
\begin{cases}
\Big(\dfrac {\log m}{\beta(n,2)}\Big)^{1-2/n} +\dfrac n2\beta(n,2)^{2/n-1} \dfrac{1 -m^{2/n}|x|^2}{ (\log m )^{2/n}} &\mbox{if $0\leqslant |x| \leqslant m^{-1/n}$},\\
-n \beta(n,2)^{2/n -1} (\log m )^{-2/n} \log |x| &\mbox{if $m^{-1/n} \leqslant |x| \leqslant 1$},\\
\xi_{m}(x) &\mbox{if $|x| >1$,}
\end{cases}
\end{equation*}
where $\xi_m \in C_0^\infty(\R^n)$ is a radial function such that $\xi_m(x) =0$ if $|x|\geqslant 2$, and 
\[
\xi_m\big|_{\{|x|=1\}} = 0, \quad \frac{\partial\xi_m}{\partial r}\Big|_{\{|x|=1\}} = -n \beta(n,2)^{2/n -1}  (\log m )^{-2/n},
\]
and $\xi_m$, $|\nabla \xi_m|$ and $\Delta \xi_m$ are all $O((\log m)^{-2/n})$. The choice of this sequence is inspired by similar sequences used in \cite{MS2014} and in \cite{LTZ2015} for the case of $\R^n$.

Following the idea in \cite{Kar2015}, let us define $\widetilde{v}_m(x) = v_m(3x)$, which then implies that $\widetilde{v}_m \in W^{2,n/2}(\H^n)$ for all $m$. Moreover, we can readily check that
\[
\|\widetilde{v}_m\|_{n/2}^{n/2} = O\Big(\frac 1{\log m}\Big)
\]
and 
\[
1\leqslant \|\Delta_g\widetilde{v}_m\|_{n/2}^{n/2}\leqslant 1 + O\Big(\frac 1{\log m}\Big).
\]
Setting $u_m = \widetilde{v}_m \|\Delta_g\widetilde{v}_m\|_{n/2}^{-1}$, we obtain $\|u_m\|_{n/2}^{n/2} \leqslant O\big(1/\log m\big)$ and $\|\Delta_g u_m\|_{n/2}^{n/2} = 1$. Moreover, for any $\beta > 0$ and $p >0$, we have
\begin{align*}
\int_{\H^n} & \frac{\Phi_{n,2}(\beta u_m^{n/(n-2)})}{(1+|u|)^p} dV_g \\
&\geqslant \int_{\{|x|\leqslant 3^{-1}m^{-1/n}\}} \frac{\Phi_{n,2}(\beta u_m^{n/(n-2)})}{(1+|u|)^p} dV_g\\
&\gtrsim (\log m)^{-(n-2)p/n} \int_{\{|x|\leqslant 3^{-1}m^{-1/n}\}} \exp\Big(\frac{\beta}{\beta(n,2)} \frac{\log m}{ (1 + O (1 / \log m) )^{n/(n-2)}} \Big) dV_g\\
&\gtrsim (\log m)^{-(n-2)p/n} \int_0^{3^{-1} m^{-1/n}} \exp\Big(\frac{\beta}{\beta(n,2)} \log m \Big) r^{n-1} dr\\
&\geqslant (\log m)^{-(n-2)p/n} \exp\Big(\Big(\frac{\beta}{\beta(n,2)}-1\Big) \log m \Big).
\end{align*}
Therefore, we get
\[
\frac1{\|u_m\|_{n/2}^{n/2}}\int_{\H^n} \frac{\Phi_{n,2}(\beta u_m^{n/(n-2)})}{(1+|u|)^p} dV_g \gtrsim (\log m)^{1-(n-2)p/n} \exp\Big(\Big(\frac{\beta}{\beta(n,2)}-1\Big) \log m \Big) .
\]
This shows that if $ \beta > \beta(n,2)$ or $\beta = \beta(n,2)$ and $ p<n/(n-2)$, then 
\[
\lim\limits_{m\to +\infty} \frac1{\|u_m\|_{n/2}^{n/2}}\int_{\H^n} \frac{\Phi_{n,2}(\beta u_m^{n/(n-2)})}{(1+|u|)^p} dV_g = +\infty.
\]
This proves the sharpness of Theorem \ref{Maintheorem}.


\section{Adams-type inequalities: Proof of Theorems \ref{AdachiTanakatype}, \ref{Adamsnormtau} and \ref{improvedversion}}
\label{sec-VariousAdamsType}

\subsection{Proof of Theorem \ref{AdachiTanakatype}} 

We shall prove that this theorem is simply a consequence of Theorem \ref{Maintheorem}. Indeed, for any $u\in W^{2,n/2}(\H^n)$ such that $\|\Delta_g u\|_{n/2} \leqslant 1$, we denote
\[
\Omega =\{x\in \H^n\, :\, |u(x)| >1\}.
\]
In $\Omega^c$, we have $|u| \leqslant 1$. Then by the definition of $\Phi_{n,2}$ we have
\begin{align*}
\Phi_{n,2}(\alpha |u|^{n/(n-2)}) = & \sum_{j=j_{n/2}-1}^{+\infty} \frac{\alpha^j}{j!} |u|^{jn/(n-2)} \\
\leqslant & |u|^{n/2}\sum_{j=j_{n/2}-1}^{+\infty} \frac{\alpha^j}{j!} \\
\leqslant & e^{\alpha} |u|^{n/2} \leqslant e^{\beta(n,2)}|u|^{n/2}.
\end{align*}
(Note that $(j_{n/2}-1)n/(n-2) \geqslant n/2$.) Therefore,
\begin{equation}\label{eq:est1}
\int_{\Omega^c} \Phi_{n,2}(\alpha |u|^{n/(n-2)}) dV_g \leqslant e^{\beta(n,2)}\|u\|_{n/2}^{n/2} \leqslant \frac{\beta(n,2) \exp \big( \beta(n,2) \big)}{\beta(n,2)-\alpha} \|u\|_{n/2}^{n/2}.
\end{equation}
In $\Omega$ we have $|u| \geqslant 1$, then 
\[
\frac{\Phi_{n,2}(\beta(n,2) |u|^{n/(n-2)})}{(1+|u|)^{n/(n-2)}} \gtrsim \exp(\beta(n,2) |u|^{n/(n-2)}) |u|^{-n/(n-2)}.
\]
Using the elementary inequality $e^{-t} \leqslant e^{-1} t^{-1}$ for any $t >0$ and Theorem \ref{Maintheorem}, we have
\begin{align}\label{eq:est2}
\int_{\Omega} \Phi_{n,2}(\alpha |u|^{n/(n-2)}) dV_g &\leqslant \int_{\Omega} \exp(\alpha |u|^{n/(n-2)}) dV_g\notag\\
&=\int_{\Omega} \exp(\beta(n,2) |u|^{n/(n-2)}) \exp(-(\beta(n,2)-\alpha) |u|^{n/(n-2)}) dV_g \notag\\
&\leqslant \frac{1}{e(\beta(n,2) -\alpha)} \int_{\Omega} \exp(\beta(n,2) |u|^{n/(n-2)}) |u|^{-n/(n-2)} dV_g\notag\\
&\lesssim \frac{1}{e(\beta(n,2) -\alpha)} \int_{\Omega} \frac{\Phi_{n,2}(\beta(n,2) |u|^{n/(n-2)})}{(1+|u|)^{n/(n-2)}} dV_g\notag\\
&\lesssim \frac{1}{\beta(n,2) -\alpha} \|u\|_{n/2}^{n/2}.
\end{align}
Combining estimates \eqref{eq:est1} and \eqref{eq:est2}, we obtain \eqref{eq:AdachiTanakatype}. The estimate \eqref{eq:Cnalpha} then follows accordingly.

The sharpness of constant $\beta(n,2)$ follows from Theorem \ref{Maintheorem}. Hence we finish the proof of Theorem \ref{AdachiTanakatype}.

\subsection{Proof of Theorem \ref{Adamsnormtau}} 

It is enough to prove inequality \eqref{eq:Adamsnormtau} for $u\in W^{2,n/2}(\H^n)$ such that $ \|u\|_{n/2}>0$ and $\|u\|_{W^{2,n/2}, \tau} = 1$. This restriction tells us that
\[
\|\Delta_g u\|_{n/2}^{n/2} = 1 -\tau \|u\|_{n/2}^{n/2} \in [0, 1).
\]
Denote
\[
v =u \|\Delta_g u\|_{n/2}^{-1}, \quad \alpha = \beta(n,2) \|\Delta_g u\|_{n/2}^{n/(n-2)}.
\] 
Clearly, $\|\Delta_g v\|_{n/2} = 1$ and $\alpha \in (0, \beta(n,2))$. We now apply Theorem \ref{AdachiTanakatype} to get
\begin{align*}
\int_{\H^n} \Phi_{n,2}(\beta(n,2) |u|^{n/(n-2)}) dV_g & = \int_{\H^n} \Phi_{n,2}(\alpha |v|^{n/(n-2)}) dV_g\\
&\leqslant \frac{C(n)}{\beta(n,2) (1 -\|\Delta_g u\|_{n/2}^{n/(n-2)} )} \frac{\|u\|_{n/2}^{n/2}}{\|\Delta_g u\|_{n/2}^{n/2}}.
\end{align*}
It is easy to show that for any $t\in (0,1)$ and any $a\in (0,2]$ there holds
\[
(1-t)^a \leqslant 1 -\min\{a, 1\}\, t.
\]
Using this elementary inequality, we obtain
\[
1 -\|\Delta_g u\|_{n/2}^{n/(n-2)} = 1 -(1 -\tau \|u\|_{n/2}^{n/2})^{2/(n-2)} \geqslant \min \Big\{ \frac 2{n-2},1 \Big\} \tau \|u\|_{n/2}^{n/2}.
\]
Hence if $\|\Delta_g u\|_{n/2} \geqslant 1/2$, then we have
\[ 
\int_{\H^n} \Phi_{n,2}(\beta(n,2) |u|^{n/(n-2)}) dV_g \leqslant \frac{2C(n)}{\beta (n,2) \min \big\{ \frac 2{n-2},1 \big\} \tau}.
\]
If $0 < \|\Delta_g u\|_{n/2} \leqslant 1/2$, then we let $v = 2 u$. Clearly, $\|\Delta_g v\|_{n/2} \leqslant 1$; hence by Theorem \ref{AdachiTanakatype}, we have
\begin{align*}
\int_{\H^n} \Phi_{n,2}(\beta(n,2) |u|^{n/(n-2)}) dV_g& = \int_{\H^n} \Phi_{n,2}\Big(\frac{\beta(n,2)}{2^{n/(n-2)}} |v|^{n/(n-2)}\Big) dV_g\\
&\leqslant C(n) \|v\|_{n/2}^{n/2} =C(n) \frac{1 -\|\Delta_g u\|_{n/2}^{n/2}}\tau \leqslant \frac{C(n)}\tau.
\end{align*}
Therefore, we have shown that
\[
\int_{\H^n} \Phi_{n,2}(\beta(n,2) |u|^{n/(n-2)}) dV_g \leqslant C(n) /\tau,
\] 
which is our desired inequality \eqref{eq:Adamsnormtau}. The estimate \eqref{eq:Cntau} follows accordingly. To conclude Theorem \ref{Adamsnormtau}, we note that the sharpness of \eqref{eq:Adamsnormtau} follows from the sharpness of \eqref{eq:AdachiTanakatype} since Theorems \ref{AdachiTanakatype} and \ref{Adamsnormtau} are equivalent; see Subsection \ref{subsect-ThmsAreEquivalent} below.


\subsection{Proof of Theorem \ref{improvedversion}} 

Fix $u\in W^{2,n/2}(\H^n)$ with $\|\Delta_g u\|_{n/2} < 1$. If $u\equiv 0$, then there is nothing to prove; hence we only consider the case $u\not\equiv 0$. For simplicity, we divide our proof into two cases.

\smallskip\noindent\textbf{Case 1}. Suppose $\|\Delta_g u\|_{n/2} \leqslant 1/2$. By denoting $v =2u$, we clearly have
\begin{equation}\label{eq:gan0}
\begin{split}
\int_{\H^n} \Phi_{n,2}\bigg( & \frac{2^{2/(n-2)} \beta(n,2)}{\big(1 + \|\Delta_g u\|_{n/2}^{n/2}\big)^{2/(n-2)}} |u|^{n/(n-2)} \bigg) dV_g \\
& =\int_{\H^n} \Phi_{n,2}\bigg(\frac{\beta(n,2)}{2\big(1 + \|\Delta_g u\|_{n/2}^{n/2}\big)^{2/(n-2)}} |v|^{n/(n-2)} \bigg) dV_g \\
&\leqslant \int_{\H^n} \Phi_{n,2} \Big(\frac{\beta(n,2)}{2} |v|^{n/(n-2)} \Big) dV_g \\
&\leqslant \frac{2C_(n)}{\beta(n,2)} \|v\|_{n/2}^{n/2} \\
&\leqslant C(n) \frac{\|u\|_{n/2}^{n/2}}{1-\|\Delta_g u\|_{n/2}^{n/2}},
\end{split}
\end{equation}
here we have used Theorem \ref{AdachiTanakatype}.

\smallskip\noindent\textbf{Case 2}. Suppose $\|\Delta_g u\|_{n/2} \geqslant 1/2$. In this scenario, let us first denote
\[
v = u \|\Delta_g u\|_{n/2}^{-1},\quad \alpha = \bigg(\frac{2 \|\Delta_g u\|_{n/2}^{n/2}}{1+ \|\Delta_g u\|_{n/2}^{n/2}}\bigg)^{2/(n-2)} \beta(n,2).
\]
Then it is clear to see that $\|v\|_{n/2} = 1$ and $\alpha < \beta(n,2)$. By applying Theorem \ref{AdachiTanakatype} we obtain
\begin{align*}
\int_{\H^n} \Phi_{n,2}\bigg( &\frac{2^{2/(n-2)} \beta(n,2)}{\big(1 + \|\Delta_g u\|_{n/2}^{n/2}\big)^{2/(n-2)}} |u|^{\frac{n}{n-2}} \bigg) dV_g \\
& = \int_{\H^n} \Phi_{n,2}(\alpha |v|^{n/(n-2)}) dV_g\\
&\leqslant \frac{C(n)}{\beta(n,2)} \bigg(1 - \Big(\frac{2 \|\Delta_g u\|_{n/2}^{n/2}}{1+ \|\Delta_g u\|_{n/2}^{n/2}}\Big)^{\frac 2{n-2}}\bigg)^{-1} \frac{\|u\|_{n/2}^{n/2}}{\|\Delta_g u\|_{n/2}^{n/2}}.
\end{align*}
Since $1 > \|\Delta_g u\|_{n/2} \geqslant 1/2$, $2/(n-2) \in (0,2]$, and
\[
\frac{ 2^{1-n/2}}{1 + 2^{- n/2}} \leqslant \frac{2 \|\Delta_g u\|_{n/2}^{n/2}}{1+ \|\Delta_g u\|_{n/2}^{n/2}} \leqslant 1,
\]
there exists some $C'(n) >0$ such that 
\[
\bigg[ 1 -\Big(\frac{2 \|\Delta_g u\|_{n/2}^{n/2}}{1+ \|\Delta_g u\|_{n/2}^{n/2}}\Big)^{\frac 2{n-2}} \bigg]^{-1} \leqslant C'(n) \bigg[ 1 -\frac{2 \|\Delta_g u\|_{n/2}^{n/2}}{1+ \|\Delta_g u\|_{n/2}^{n/2}} \bigg]^{-1} \leqslant \frac{2 C'(n)}{1 -\|\Delta_g u\|_{n/2}^{n/2}}.
\]
Therefore,
\begin{equation}\label{eq:xa0}
\begin{split}
\int_{\H^n} \Phi_{n,2}\bigg( & \frac{2^{2/(n-2)} \beta(n,2)}{\big(1 + \|\Delta_g u\|_{n/2}^{n/2}\big)^{2/(n-2)}} |u|^{\frac{n}{n-2}} \bigg) dV_g \\
 & \leqslant \frac{2^{1+n/2}C(n)C'(n)}{\beta(n,2)} \frac{\|u\|_{n/2}^{n/2}}{1-\|\Delta_g u\|_{n/2}^{n/2}} \leqslant C(n) \frac{\|u\|_{n/2}^{n/2}}{1-\|\Delta_g u\|_{n/2}^{n/2}}.
\end{split}
\end{equation}
Inequality \eqref{eq:Lionsversion} now follows from the estimates \eqref{eq:gan0} and \eqref{eq:xa0} above. Finally, we conclude the sharpness of \eqref{eq:Lionsversion}. To see this, as we have already observed once that
\[
2^{2/(n-2)} \big(1 + \|\Delta_g u\|_{n/2}^{n/2}\big)^{-2/(n-2)} > 1
\]
provided $\|\Delta_g u\|_{n/2} < 1$. Therefore, the sharpness of \eqref{eq:Lionsversion} follows from the sharpness of \eqref{eq:Adamsnormtau}. The proof of Theorem \ref{improvedversion} hence is finished.

\subsection{Theorems \ref{AdachiTanakatype} and \ref{Adamsnormtau} are equivalent}
\label{subsect-ThmsAreEquivalent}

Let us finish this section by showing that Theorems \ref{AdachiTanakatype} and \ref{Adamsnormtau} are, in fact, equivalent. To realize this interesting fact, we only have to show that Theorem \ref{AdachiTanakatype} can be derived from Theorem \ref{Adamsnormtau}. 

For any $\alpha \in (0, \beta(n,2))$ and any $u\in W^{2,n/2}(\H^n)$ such that $\|\Delta_g u\|_{n/2} \leqslant 1$, we denote 
\[
v =\Big(\frac\alpha {\beta(n,2)}\Big)^{(n-2)/n} u, \quad \tau = \frac{1 -\|\Delta_g v\|_{n/2}^{n/2}}{\|v\|_{n/2}^{n/2}}.
\] 
Clearly, 
\[
 \tau = \frac{\beta(n,2)^{n/2-1} -\alpha^{n/2-1} \|\Delta_g u\|_{n/2}^{n/2}}{\alpha^{n/2-1} \|u\|_{n/2}^{n/2}}\geqslant \frac{\beta(n,2)^{n/2-1} -\alpha^{n/2-1}}{\alpha^{n/2-1} \|u\|_{n/2}^{n/2}}.
\]
Applying Theorem \ref{Adamsnormtau} gives
\begin{align*}
\int_{\H^n} \Phi_{n,2}(\alpha |u|^{n/(n-2)}) dV_g& =\int_{\H^n} \Phi_{n,2}(\beta(n,2) |v|^{n/(n-2)}) dV_g\leqslant \frac{C(n)}\tau\\
&\leqslant C(n) \frac{\alpha^{n/2-1} \|u\|_{n/2}^{n/2}}{\beta(n,2)^{n/2-1} -\alpha^{n/2-1}}.
\end{align*}
It is easy to prove that there is some $C'(n)$ depending only on $n$ such that
\[
\frac{\alpha^{n/2-1}}{\beta(n,2)^{n/2-1} -\alpha^{n/2-1}} \leqslant \frac{C'(n)}{\beta(n,2) -\alpha}
\]
for all $\alpha \in (0,\beta(n,2))$. Hence, for any $\alpha \in (0,\beta(n,2))$ we have
\begin{equation*}
\int_{\H^n} \Phi_{n,2}(\alpha |u|^{n/(n-2)}) dV_g \leqslant \frac{C(n)}{\beta(n,2)-\alpha} \|u\|_{n/2}^{n/2},
\end{equation*}
which is nothing but \eqref{eq:AdachiTanakatype}.


\section{Adams inequality with homogeneous Navier boundary: Proof of Theorem \ref{Tarsihyperbolic}}
\label{sec-AdamsTypeWithBoudary}

In this section, we prove Theorem \ref{Tarsihyperbolic} whose proof relies on Proposition \ref{estimateforsolution1} and Lemma \ref{Adamslemma}. 

\subsection{Proof of (\ref{eq:Tarsihyperbolic})}

For simplicity, we divide the proof into two cases.

\smallskip\noindent\textbf{Case 1. Suppose that $m$ is even.} In this case, we can write $m =2k$ for some $k\geqslant 1$. This case is a simple consequence of Proposition \ref{estimateforsolution1} and Lemma \ref{Adamslemma}. Indeed, denoting 
\[
f = (-\Delta_g)^k u
\] 
and extending $f$ to be zero outside $\Omega$, it follows from Proposition \ref{estimateforsolution1} that
\[
u^*(t) \leqslant \frac n{n-2k}\frac{c_{n,k}}{(n \Omega_n^{1/n})^{2k}} \int_{t}^{|\Omega|} \frac{f^{*}(s)}{s^{1-2k/n}} ds + \frac{c_{n,k+1}}{(n \Omega_n^{1/n})^{2k}} t^{2k/n-1} \int_0^t f^*(s) ds.
\]
Recall that 
\[
\beta(n,2k) = \Big(\frac n{n-2k} \frac{c_{n,k}}{(n \Omega_n^{1/n})^{2k}}\Big)^{-n/(n-2k)}.
\]
Hence by Hardy--Littlewood's inequality, we have
\begin{align*}
\int_{\Omega} \exp & \big(\beta(n,2k) |u|^{n/(n-2k)} \big)dV_g \\
&\leqslant \int_0^{|\Omega|} \exp\big(\beta(n,2k) (u^*(t))^{n/(n-2k)} \big) dt\\
&\leqslant \int_0^{|\Omega|} \exp\bigg[\bigg(\int_t^{|\Omega|} s^{2k/n-1} f^*(s) ds + \frac n{2k} t^{2k/n-1} \int_0^{t} f^*(s) ds\bigg)^{n/(n-2k)} \bigg] dt.
\end{align*}
By changing the variables $t: = |\Omega| e^{-t}$, we obtain
\begin{equation}\label{eq:ustartichphan}
\begin{split}
\int_{\Omega} &\exp \big(\beta(n,2k) |u|^{n/(n-2k)} \big)dV_g \\
&\leqslant |\Omega|\int_0^{+\infty}\exp\left[-t +\left(
\begin{gathered}
\int_{|\Omega| e^{-t}}^{|\Omega|} s^{2k/n -1} f^*(s) ds + \\
 \frac n{2k}(|\Omega|e^{-t})^{2k/n -1} \int_0^{|\Omega|e^{-t}} f^*(s) ds
\end{gathered}
\right)^{n/(n-2k)}\right] dt.
\end{split}
\end{equation}
Denote $\phi(s) = f^*(|\Omega| e^{-s}) (|\Omega| e^{-s})^{2k/n} $ and 
\[
a(s,t) = 
\begin{cases}
0 &\mbox{if $s < 0$},\\
1 &\mbox{if $0\leqslant s < t$},\\
n e^{(s-t) (2k/n -1)}/(2k) &\mbox{if $s \geqslant t$.}
\end{cases}
\]
Then by changing of the variables $s:= |\Omega|e^{-s}$ in \eqref{eq:ustartichphan}, it is straightforward to see that
\[\begin{split}
\int_{\Omega} \exp &  (\beta(n,2k) |u|^{n/(n-2k)}  )dV_g \\
&\leqslant |\Omega| \int_0^{+\infty}\exp\Big[-t + \Big(\int_0^{+\infty} a(s,t) \phi(s) ds\Big)^{n/(n-2k)} \Big] dt.
\end{split} \]
We can easily verify that
\[
\int_{\R} \phi(s)^{n/(2k)} ds = \int_0^{|\Omega|} (f^{*}(s))^{n/(2k)} ds = \int_{\Omega} |f|^{n/(2k)} dV_g = 1
\]
and that
\[
\sup_{t > 0} \bigg[ \Big(\int_{-\infty}^0 + \int_t^{+\infty}\Big) a(s,t)^{n/(n-2k)} ds\bigg]^{(n-2k)/n} = \frac{n}{2k}.
\]
By Lemma \ref{Adamslemma}, therefore there is a constant $C(n,k)$ depending only on $n$ and $k$ such that
\[
\int_{\Omega} \exp\big(\beta(n,2k) |u|^{n/(n-2k)} \big)dV_g\leqslant C(n,k)|\Omega|.
\]
This completes the first case.

\smallskip\noindent\textbf{Case 2. Suppose that $m$ is odd.} In this scenario, we can write $m=2k+1$ for some $k\geqslant 0$. If $k =0$, then the space $W_{N,g}^{1,n}(\Omega)$ is exactly the space $W_0^{1,n}(\Omega)$. Therefore, the conclusion follows from \cite[Corollary $1.1$]{LT2013}. Hence we need to concentrate on the case $k \geqslant 1$. Denote
\[
f =(-\Delta_g)^k u
\]
and extend $f$ to be zero outside $\Omega$, then by Proposition \ref{estimateforsolution1}, we obtain
\begin{equation}\label{eq:stepp1}
u^*(t) \leqslant \frac n{n-2k}\frac{c_{n,k}}{(n \Omega_n^{1/n})^{2k}} \int_{t}^{|\Omega|} \frac{f^{*}(s)}{s^{1-2k/n}} ds + \frac{c_{n,k+1}}{(n \Omega_n^{1/n})^{2k}} t^{2k/n-1} \int_0^t f^*(s) ds.
\end{equation}
Recall that $f^\sharp(x) = f^* ( |B (0,d(0,x)) | )$ with $d(0,x) = \log \big( (1+|x|)/(1-|x|)\big)$. Hence
\[
\nabla_g f^\sharp(x) = n\Omega_n\frac{1-|x|^2}{2} (f^*)' ( |B (0,d(0,x) ) | )\Big(\frac{2|x|}{1-|x|^2}\Big)^{n-1}.
\]
Thus we have
\[\begin{split}
\int_{\H^n}& |\nabla_g f^\sharp|^{n/m} dV_g \\
&=(n\Omega_n)^{n/m +1} \int_0^1\Big|(f^*)'\Big(\Big|B\Big(0,\log \frac{1+r}{1-r}\Big)\Big|\Big)\Big|^{n/m}\Big(\frac{2r}{1-r^2}\Big)^{(n-1)n/m +n} \frac{ dr}r.
\end{split}\]
Upon using the change of variables
\[
s= \big| B\big(0,\log \big[ (1+r)/(1-r) \big] \big) \big|
\] 
we obtain
\[
F(s) = \log \big[ (1+r)/(1-r) \big],
\]
where $F$ is a continuous, strictly increasing function as in the proof of Proposition \ref{comparison}. Resolving this equation gives
\[
r =\big( e^{F(s)}-1 \big) / \big( e^{F(s)}+1 \big).
\] 
Hence
\[
\int_{\H^n} |\nabla_g f^\sharp|^{n/m} dV_g =(n\Omega_n)^{n/m} \int_0^{|\Omega|} |(f^*)'(s)|^{n/m}  (\sinh F(s) )^{n(n-1)/m} ds.
\]
(Note that $ds = n\Omega_n (2r/(1-r^2) )^{n} dr/r$.) Let us define the function 
\[
\varphi(t) = (n\Omega_n)^{-n/(n-m)} \int_t^{|\Omega|} (\sinh F(s) )^{-n(n-1)/(n-m)} ds.
\]
Then $\varphi$ is strictly decreasing and has the following asymptotic behavior: $\varphi(|\Omega|) =0$ and $\lim_{t\to 0} \varphi(t) =+\infty$. Let $g$ be an increasing function such that $f^*(s) = g(\varphi(s))$, then it is easy to check that
\[
(n\Omega_n)^{n/m}\int_0^{|\Omega|} |(f^*)'(s)|^{n/m} (\sinh F(s) )^{n(n-1)/m} ds = \int_0^{+\infty} (g'(s))^{n/m} ds.
\]
Observe that $\|\nabla_g f^\sharp\|_{n/m} \leqslant \|\nabla_g f\|_{n/m} \leqslant 1$; hence
\[
\int_0^{+\infty} (g'(s))^{n/m} ds \leqslant 1.
\]
Denote by $k = (g')^*$ the rearrangement function of $g'$ in $(0, + \infty)$, by Hardy--Littlewood's inequality, we obtain 
\[
f^*(s) = \int_0^{\varphi(s)} g'(t) dt \leqslant \int_{0}^{\varphi(s)} k(t) dt
\]
for any $s\in (0,|\Omega|)$. By using integration by parts, we get
\begin{equation}\label{eq:stepp2}
\begin{split}
\int_{t}^{|\Omega|} \frac{f^{*}(s)}{s^{1-2k/n}} ds& \leqslant \frac n{2k}\int_t^{|\Omega|} \int_0^{\varphi(s)}k(r) dr ds^{2k/n} \\
&=-\frac n{2k} t^{2k/n}\int_0^{\varphi(t)} k(s) ds -\frac n{2k}\int_t^{|\Omega|} k(\varphi(s)) \varphi'(s) s^{2k/n} ds
\end{split}
\end{equation}
and
\begin{equation}\label{eq:stepp3}
\begin{split}
\int_0^t f^*(s) ds & \leqslant \int_0^t \int_0^{\varphi(s)} k(r) dr ds \\
&= t\int_0^{\varphi(t)} k(s) ds -\int_0^t k(\varphi(s)) \varphi'(s) sds.
\end{split}
\end{equation}
(Here we have used $\varphi(|\Omega|) =0$ and $\lim_{s\to 0} s\int_0^{\varphi(s)} k(r) dr =0$.) Upon plugging \eqref{eq:stepp3} and \eqref{eq:stepp2} into \eqref{eq:stepp1}, we arrive at
\[\begin{split}
u^*(t) \leqslant &-\frac{c_{n,k+1}}{(n \Omega_n^{1/n})^{2k}} \int_{t}^{|\Omega|} k(\varphi(s)) \varphi'(s) s^{2k/n} ds 
- \frac{c_{n,k+1}}{(n \Omega_n^{1/n})^{2k}} t^{2k/n-1} \int_0^t k(\varphi(s))\varphi'(s) s ds.
\end{split}\]
It follows from the definition of $\varphi$ and \eqref{eq:lowerestF} that
\[\begin{split}
-\varphi'(s) = & (n\Omega_n)^{-n/(n-m)}  (\sinh F(s) )^{- n(n-1)/(n-m)} \\
\leqslant & (n\Omega_n^{1/n})^{- n/(n-m)} s^{-(n-1)/(n-m)}.
\end{split}\]
Denote 
\[
l(s) = k(\varphi(s)) (-\varphi'(s))^{m/n},
\] 
then we have 
\begin{equation}\label{eq:stepp4}
\begin{split}
u^*(t) \leqslant &\frac{c_{n,k+1}}{(n \Omega_n^{1/n})^{2k+1}} \int_{t}^{|\Omega|} \frac{l(s)}{ s^{1-(2k+1)/n}} ds  + \frac{c_{n,k+1}}{(n \Omega_n^{1/n})^{2k+1}} t^{2k/n-1} \int_0^t l(s) s^{1/n} ds \\
\leqslant &\frac{c_{n,k+1}}{(n \Omega_n^{1/n})^{2k+1}} \int_{t}^{|\Omega|} \frac{l(s)}{ s^{1-(2k+1)/n}} ds + \frac{c_{n,k+1}}{(n \Omega_n^{1/n})^{2k+1}} t^{2k/n} \int_0^t l(s) ds.
\end{split}
\end{equation}
(Keep in mind that $m=2k+1$.) Now we can repeat the argument used in the case when $m$ is even to finish the proof of Theorem \ref{Tarsihyperbolic} when $m$ is odd by using Lemma \ref{Adamslemma}, estimate \eqref{eq:stepp4}, and the fact
\[
\int_0^{|\Omega|} l(s)^{n/m}ds \leqslant \int_0^{+\infty} k(s)^{n/m} ds \leqslant 1.
\]
To conclude Theorem \ref{Tarsihyperbolic}, it suffices to establish the sharpness of \eqref{eq:Tarsihyperbolic} and this is the content of the next subsection.

\subsection{The sharpness of (\ref{eq:Tarsihyperbolic})} 

The way to see the sharpness of \eqref{eq:Tarsihyperbolic} is as follows: Note that since $W_{N,g}^{m,n/m}(\Omega) \subset W_0^{m,n/m}(\Omega)$, the supremum of the left hand side of \eqref{eq:Tarsihyperbolic} in $W_{N,g}^{m,n/m}(\Omega)$ is greater than that in $W_0^{m,n/m}(\Omega)$. Since \eqref{eq:Adams} is sharp, it follows that \eqref{eq:Tarsihyperbolic} is also sharp.


\section{A Lions-type lemma for Adams inequality: Proof of Theorem \ref{Lionslemma}}
\label{sec-LionsLemma}

In this long section, we prove Theorem \ref{Lionslemma}. To achieve that goal, we borrow some ideas from \cite{CCH2013} for the case $m=1$ and a fine analysis in \cite{VHN2016} for the Euclidean case. Our approach basically consists of two steps: First we reduce the sequence $\{u_j\}_j \subset \H^n$ in Theorem \ref{Lionslemma} to the case of $u_j \in C_0^\infty(\H^n)$; see Proposition \ref{reducesmooth}. Then we establish Theorem \ref{Lionslemma} for any sequence $u_j \in C_0^\infty(\H^n)$ by way of contradiction; see Subsection \ref{subsec-ProofForCompactlySupport}.

\subsection{An useful estimate for rearrangement functions}

In this subsection, we prove an useful estimate for rearrangement functions; see Proposition \ref{keyforLions}. We note that this result shares some similarity with Proposition \ref{estimateforsolution}.

Let $u\in C_0^\infty(\H^n)$, our aim is to estimate $u^*(t_1) -u^*(t_2)$ from above for any $0 < t_1 < t_2 < +\infty$. For simplicity, we denote
\[
u_i = (-\Delta)^i u
\] 
for each $i =0,1, ... ,k$ with a convention that $u_0 \equiv u$. Then we have
\begin{equation}\label{eq:u*-u*}
u_i^*(t_1) -u_i^*(t_2) \leqslant \int_{t_1}^{t_2} \frac{t u_{i+1}^{**}(t)}{(n\Omega_n (\sinh F(t))^{n-1})^2} dt
\end{equation}
for all $i=0,1, ... , k-1$. By sending $t_2 \nearrow +\infty$ and using $\lim_{t\to +\infty} u_i^*(t) =0$, we deduce that
\[
u_i^*(t) \leqslant \int_{t}^{+\infty} \frac{s u_{i+1}^{**}(s)}{(n\Omega_n (\sinh F(s))^{n-1})^2} ds.
\]
Now integrating by parts gives
\begin{align*}
u_i^{**}(t) &= \frac 1t \int_0^t u_i^*(s) ds\\
&\leqslant \frac 1t \int_0^s \Big(\int_{t}^{+\infty} \frac{a u_{i+1}^{**}(a)}{(n\Omega_n (\sinh F(a))^{n-1})^2} da\Big) ds\\
&=\int_{t}^{+\infty} \frac{s u_{i+1}^{**}(s)}{(n\Omega_n (\sinh F(s))^{n-1})^2} ds + \int_0^t\frac{t^{-1}s^2 u_{i+1}^{**}(s)}{(n\Omega_n (\sinh F(s))^{n-1})^2} ds.
\end{align*}
Define
\[
G(t,s) =
\begin{cases}
s (n\Omega_n (\sinh F(s))^{n-1})^{-2} &\mbox{if $s \geqslant t$},\\
t^{-1}s^2 (n\Omega_n (\sinh F(s))^{n-1})^{-2}&\mbox{if $s < t$.}
\end{cases}
\]
It is not hard to see that
\begin{equation}\label{eq:u**<intG}
u_i^{**}(t) \leqslant \int_0^{+\infty} G(t,s) u_{i+1}^{**}(s) ds.
\end{equation}
Combining \eqref{eq:u*-u*} and \eqref{eq:u**<intG} gives
\begin{equation}\label{eq:u*-u*<intintG}
u_i^*(t_1) -u_i^*(t_2) \leqslant \int_{t_1}^{t_2} \frac{t}{(n\Omega_n (\sinh F(t))^{n-1})^2} \int_0^{+\infty} G(t,s) u_{i+1}^{**}(s) ds dt
\end{equation}
for all $i=0,1, ... , k-1$. We now define a sequence $(G_i)_{i\geqslant 1}$ as follows: Set 
\begin{align*}
\left.
\begin{aligned}
G_1 &= G\\ 
\text{and } \qquad 
G_{i+1}(t,s) & = \int_0^{+\infty} G_i(t,s_1) G(s_1,s) ds_1 .
\end{aligned}
\right\}
\end{align*}
Obviously, $G_{i+1}(t,s) = \int_0^{+\infty} G(t,s_1) G_i(s_1,s) ds_1$. By setting $i=0$ in \eqref{eq:u*-u*<intintG} and using \eqref{eq:u**<intG} repeatedly, we arrive at
\begin{equation}\label{eq:bieudienchou-tamthoi}
\begin{split}
u^*(t_1) -u^*(t_2) &\leqslant \int_{t_1}^{t_2} \frac{t}{(n\Omega_n (\sinh F(t))^{n-1})^2} \int_0^{+\infty} G_{k-1}(t,s) u_k^{**}(s) ds dt \\
&= \int_0^{+\infty} u_k^{**}(s) \int_{t_1}^{t_2} \frac{t}{(n\Omega_n (\sinh F(t))^{n-1})^2} G_{k-1}(t,s) dt ds.
\end{split}
\end{equation}
Let us define consecutively the functions $L_i, H_i,K_i$ for $i =1,2\ldots,k-1$ by
\begin{align*}
\left.
\begin{aligned}
L_1(t) &= \frac{t}{(n\Omega_n (\sinh F(t))^{n-1})^2},\\
H_i(t) &= \int_0^t L_i(s) ds, \\
K_i(t) &= \int_t^{+\infty} s^{-2} H_i(s) ds,\\
\text{ and } \qquad
L_{i+1}(t) &= \frac{t}{(n\Omega_n (\sinh F(t))^{n-1})^2} K_i(t).
\end{aligned}
\right\}
\end{align*}
Using these notations, we can rewrite \eqref{eq:bieudienchou-tamthoi} as follows
\begin{equation}\label{eq:bieudienchou}
\begin{split}
u^*(t_1) -u^*(t_2) &\leqslant \int_0^{+\infty} u_k^{**}(s) \int_{t_1}^{t_2} G_{k-1}(t,s) L_1(t) dt ds.
\end{split}
\end{equation}
For $i< k-1$, using integration by parts, we get
\begin{align*}
\int_{t_1}^{t_2} G_{k-i}(t,s) L_i(t) dt = &G_{k-i}(t_2,s)H_i(t_2) -G_{k-i}(t_1,s) H_i(t_1)\\
& +\int_{t_1}^{t_2} H(t) t^{-2} \int_0^t \frac{s_1^2}{(n\Omega_n (\sinh F(s_1))^{n-1})^2} G_{k-i-1}(s_1,s) ds_1 dt\\
=&G_{k-i}(t_2,s)H_i(t_2) -G_{k-i}(t_1,s) H_i(t_1)\\
& -K_i(t_2) \int_0^{t_2} \frac{s_1^2}{(n\Omega_n (\sinh F(s_1))^{n-1})^2} G_{k-i-1}(s_1,s) ds_1 \\
& + K_i(t_1) \int_0^{t_1} \frac{s_1^2}{(n\Omega_n (\sinh F(s_1))^{n-1})^2} G_{k-i-1}(s_1,s) ds_1\\
& + \int_{t_1}^{t_2} G_{k-i-1}(t,s) L_{i+1}(t) dt.
\end{align*}
When $i=k-1$, we use integration by parts again to obtain
\begin{align*}
\int_{t_1}^{t_2} G(t,s) L_{k-1}(t) dt = &G(t_2,s) H_{k-1}(t_2) -G(t_1,s) H_{k-1}(t_1) \\
& + \int_{t_1}^{t_2} \chi_{\{s < t\}} \frac{t^{-2}s^2}{(n\Omega_n (\sinh F(s))^{n-1})^2} H_{k-1}(t) dt.
\end{align*}
We are now able to estimate $\int_{t_1}^{t_2} G_{k-1}(t,s) L_1(t) dt$ as follows
\begin{align}\label{eq:trunggianusao}
\int_{t_1}^{t_2} G_{k-1}(t,s) L_1(t) dt\leqslant &\sum_{i=1}^{k-1} (G_{k-i}(t_2,s)H_i(t_2) -G_{k-i}(t_1,s) H_i(t_1) ) \nonumber \\
& +\sum_{i=1}^{k-2}K_i(t_1) \int_0^{t_1} \frac{s_1^2}{(n\Omega_n (\sinh F(s_1))^{n-1})^2} G_{k-i-1}(s_1,s) ds_1  \nonumber  \\
& -\sum_{i=1}^{k-2}K_i(t_2) \int_0^{t_2} \frac{s_1^2}{(n\Omega_n (\sinh F(s_1))^{n-1})^2} G_{k-i-1}(s_1,s) ds_1  \nonumber  \\
& +\int_{t_1}^{t_2} \chi_{\{s < t\}} \frac{t^{-2}s^2}{(n\Omega_n (\sinh F(s))^{n-1})^2} H_{k-1}(t) dt.
\end{align}
When plugging the preceding inequality into \eqref{eq:bieudienchou}, there are terms needed separately attention. First, we handle the term involving the last term on the right hand side of \eqref{eq:trunggianusao}. Clearly,
\begin{align}\label{eq:goodexpress}
\int_0^{+\infty} u_k^{**}(s) &\int_{t_1}^{t_2} \chi_{\{s < t\}} \frac{t^{-2}s^2}{(n\Omega_n (\sinh F(s))^{n-1})^2} H_{k-1}(t) dt ds  \nonumber\\
 = &\int_0^{t_2} u_k^{**}(s) \int_{t_1}^{t_2} \chi_{\{s < t\}} \frac{t^{-2}s^2}{(n\Omega_n (\sinh F(s))^{n-1})^2} H_{k-1}(t) dt ds \nonumber \\
=&\big(K_{k-1}(t_1) -K_{k-1}(t_2)\big) \int_0^{t_1}u_k^{**}(s)\frac{s^2}{(n\Omega_n (\sinh F(s))^{n-1})^2} ds  \nonumber \\
& + \int_{t_1}^{t_2} u_k^{**}(s)\frac{s^2}{(n\Omega_n (\sinh F(s))^{n-1})^2} (K_{k-1}(s) -K_{k-1}(t_2) ) ds \nonumber\\
= &K_{k-1}(t_1) \int_0^{t_1}u_k^{**}(s)\frac{s^2}{(n\Omega_n (\sinh F(s))^{n-1})^2} ds \nonumber \\
& - K_{k-1}(t_2)\int_0^{t_2}u_k^{**}(s)\frac{s^2}{(n\Omega_n (\sinh F(s))^{n-1})^2} ds \nonumber \\
& + \int_{t_1}^{t_2}u_k^{**}(s)\frac{s^2}{(n\Omega_n (\sinh F(s))^{n-1})^2}K_{k-1}(s) ds . 
\end{align}
To handle the term involving the first term on the right hand side of \eqref{eq:trunggianusao}, we denote
\begin{align*}
F(t_1,t_2,s) = &\sum_{i=1}^{k-1} [G_{k-i}(t_2,s)H_i(t_2) -G_{k-i}(t_1,s) H_i(t_1) ] \\
& +\sum_{i=1}^{k-2}K_i(t_1) \int_0^{t_1} \frac{s_1^2}{(n\Omega_n (\sinh F(s_1))^{n-1})^2} G_{k-i-1}(s_1,s) ds_1 \\
& -\sum_{i=1}^{k-2}K_i(t_2) \int_0^{t_2} \frac{s_1^2}{(n\Omega_n (\sinh F(s_1))^{n-1})^2} G_{k-i-1}(s_1,s) ds_1.
\end{align*}
Hence, combining \eqref{eq:bieudienchou}, \eqref{eq:trunggianusao}, and \eqref{eq:goodexpress} gives
\begin{align}\label{eq:finalexpress}
u^*(t_1) -u^*(t_2) \leqslant &\int_{t_1}^{t_2} u_k^{**}(s) F(t_1,t_2,s) ds  \nonumber  \\
&+ K_{k-1}(t_1)\int_0^{t_1}u_k^{**}(s)\frac{s^2}{(n\Omega_n (\sinh F(s))^{n-1})^2} ds \nonumber  \\
& -K_{k-1}(t_2)\int_0^{t_2}u_k^{**}(s)\frac{s^2}{(n\Omega_n (\sinh F(s))^{n-1})^2} ds  \nonumber \\
& + \int_{t_1}^{t_2}u_k^{**}(s)\frac{s^2}{(n\Omega_n (\sinh F(s))^{n-1})^2}K_{k-1}(s) ds.
\end{align}
Our job has not finished yet. In the following step, we aim to estimate $L_i(t)$, $H_i(t)$, $K_{i}(t)$, and $\int_0^{+\infty} F(t_1,t_2,s)^{n/(n-2k)} ds$. 

Concerning the terms $L_i(t)$, $H_i(t)$, and $K_{i}(t)$, we have the following result.

\begin{proposition}\label{boundforLHKfunction}
Let $c_{n,i}$ be the constant given in Proposition \ref{estimateforsolution1}. Then for $1 \leqslant i \leqslant  k-1$ we have the following claims:
\begin{enumerate}
 \item There holds $L_i(t) \leqslant (n\Omega_n^{1/n})^{-2i} c_{n,i} t^{2i/n -1}$ for all $t>0$ and 
\[
L_i(t) \sim \frac{1}{(n-1)^i (i-1)!}\frac{(\log t)^{i-1}} t
\]
as $t\to +\infty$.
 \item There holds $H_i(t) \leqslant (n\Omega_n^{1/n})^{-2i} c_{n,i} t^{2i/n} n /(2i)$ for all $t>0$ and 
\[
H_i(t) \sim \frac{1}{(n-1)^i i!}(\log t)^{i} 
\] 
as $t\to +\infty$.
 \item There holds $K_i(t) \leqslant (n\Omega_n^{1/n})^{-2i} c_{n,i+1} t^{2i/n -1}$ for all $t>0$ and
\[
K_i(t) \sim \frac{1}{(n-1)^i i!}\frac{(\log t)^{i}} t 
\]
as $t\to +\infty$.
\end{enumerate}
\end{proposition}

\begin{proof}
This is elementary, simply by induction argument; hence we omit its details.
\end{proof}

\begin{proposition}\label{boundforGfunction}
There exists a constant $C$ depending only on $n,k$ such that 
\[
G_i(t,s) \leqslant 
\begin{cases}
C s^{-1+ 2i/n} &\mbox{if $s \geqslant t$},\\
C t^{-1+2(i-1)/n} &\mbox{if $s < t$},\\
\end{cases}
\]
for $i =1,2, ... ,k-1$.
\end{proposition}

\begin{proof}
To prove, we first observe that $\sinh F(s) \geqslant (s/\Omega_n)^{1/n}$. Therefore,
\[
G_1(t,s) \leqslant 
\begin{cases}
(n\Omega_n^{1/n})^{-2} s^{-1+2/n}&\mbox{if $s \geqslant t$},\\
(n\Omega_n^{1/n})^{-2} t^{-1} s^{2/n} &\mbox{if $s < t$.}
\end{cases}
\]
This shows that the conclusion holds for $i=1$. Using induction argument, we obtain the conclusion; for a detailed explanation, we refer the reader to \cite{VHN2016}.
\end{proof}

An immediately consequence of Proposition \ref{boundforGfunction} is the following estimate
\[
\int_0^{t} \frac{s_1^2}{(n\Omega_n (\sinh F(s_1))^{n-1})^2} G_{k-i-1}(s_1,s) ds_1 \leqslant 
\begin{cases}
C t^{2(k-i)/n} &\mbox{if $s < t$},\\
C t^{1+2/n} s^{-1 + 2(k-i-1)/n} &\mbox{if $s > t$},
\end{cases}
\]
for $i=1, ... ,k-2$, which then implies
\[
\Big(\int_0^{+\infty} \Big(\int_0^t \frac{s_1^2}{(n\Omega_n (\sinh F(s_1))^{n-1})^2} G_{k-i-1}(s_1,s) ds_1\Big)^{n/(n-2k)}ds\Big)^{(n-2k)/n} \leqslant C t^{1-2i/n}
\]
for $i=1, ... ,k-2$. This inequality and Proposition \ref{boundforLHKfunction} give 
\begin{equation}\label{eq:boundfornormofG}
\Big(\int_0^{+\infty} F(t_1,t_2,s)^{n/(n-2k)} ds\Big)^{(n-2k)/n} \leqslant C
\end{equation}
for any $0< t_1< t_2$, where the constant $C$ depends only on $n$, $k$. Moreover, we have 
\begin{align*}
\int_0^{t}u_k^{**}(s) & \frac{s^2}{(n\Omega_n (\sinh F(s))^{n-1})^2} ds\\
&\leqslant \frac1{(n\Omega_n^{1/n})^2} \int_0^t u_k^{**}(s) s^{2/n} ds\\
&\leqslant \frac1{(n\Omega_n^{1/n})^2} \Big(\int_0^{+\infty} (u_k^{**}(s))^{n/(2k)} ds\Big)^{2k/n}\Big(\int_0^t s^{2/(n-2k)} ds\Big)^{(n-2k)/n}\\
&\leqslant C \Big(\int_0^{+\infty} (u_k^*(s))^{n/(2k)} ds\Big)^{2k/n} t^{1- 2(k-1)/n}
\end{align*}
for any $t > 0$, here we have used Lemma \ref{Hardy}. Combining this inequality and Proposition \ref{boundforLHKfunction}, we obtain
\begin{equation}\label{eq:boundatkminus1}
\int_0^{t}(\Delta^k u)^{**}(s)\frac{s^2}{(n\Omega_n (\sinh F(s))^{n-1})^2} ds K_{k-1}(t) \leqslant C \|u_k\|_{n/(2k)}.
\end{equation}
Combining \eqref{eq:finalexpress}, \eqref{eq:boundfornormofG}, and \eqref{eq:boundatkminus1}, we arrive at
\begin{equation}\label{eq:keytwostar}
u^*(t_1) -u^*(t_2) \leqslant \int_{t_1}^{t_2}(\Delta^k u)^{**}(s)\frac{s^2}{(n\Omega_n (\sinh F(s))^{n-1})^2}K_{k-1}(s) ds + C \|\Delta^k u\|_{n/(2k)}
\end{equation}
for any $0 < t_1 < t_2 < +\infty$, with the notation $K_0(s) = s^{-1}$. Denote
\[
M(t) = \int_t^{+\infty} \frac{s}{(n\Omega_n (\sinh F(s))^{n-1})^2}K_{k-1}(s) ds.
\]
Then we have from Proposition \ref{boundforLHKfunction} that
\begin{equation}\label{eq:boundforMfunction1}
M(t) \leqslant \frac{c_{n,k}}{(n\Omega_n^{1/n})^{2k}} \frac n{n-2k} t^{2k/n -1}
\end{equation}
for all $t > 0$ and
\begin{equation*}\label{eq:boundforMfunction2}
M(t)\sim\frac1{(n-1)^k (k-1)!} \frac{(\log t)^{k-1}}{t}
\end{equation*}
as $t\to +\infty$. Using integration by parts, we get
\begin{equation}\label{eq:ukstar}
\begin{split}
\int_{t_1}^{t_2} & (\Delta^k u)^{**}(s)  \frac{s^2}{(n\Omega_n (\sinh F(s))^{n-1})^2}K_{k-1}(s) ds \\
= &-M(t_2) \int_0^{t_2} (\Delta^k u)^*(s) ds + M(t_1) \int_0^{t_1} (\Delta^k u)^*(s) ds + \int_{t_1}^{t_2} (\Delta^k u)^*(s) M(s) ds.
\end{split}
\end{equation}
In view of \eqref{eq:boundforMfunction1}, we easily see that
\begin{equation}\label{eq:bound1star}
M(t) \int_0^{t} (\Delta^k u)^*(s) ds \leqslant C\|\Delta^k u\|_{n/(2k)}
\end{equation}
for all $t > 0$. Here the constant $C$ depends only on $n$ and $k$. By combining \eqref{eq:keytwostar}, \eqref{eq:ukstar}, and \eqref{eq:bound1star}, we have shown the following key result.

\begin{proposition}\label{keyforLions}
For any $u\in C_0^\infty(\H^n)$ and for any $1\leqslant k < n/2$, there exists a constant $C(n,k)$ such that 
\begin{equation}\label{eq:keyforLions}
u^*(t_1) -u^*(t_2) \leqslant \int_{t_1}^{t_2}(\Delta^k u)^{*}(s)M(s) ds + C(n,k) \|\Delta^k u\|_{n/(2k)}
\end{equation}
for any $0 < t_1 < t_2 < +\infty$.
\end{proposition}

\subsection{Reduce to compactly supported smooth functions}

We start this section by showing that if $u\in W^{m,n/m}(\H^n)$, then $|u|^{n/(n-m)}$ will be exponentially integrable. Along the proof of this fact, we shall frequently apply the following elementary inequality 
\begin{equation}\label{eqELEMENTARY}
|a|^{n/(n-m)} \leqslant (1+\delta) |a-b|^{n/(n-m)} + C_{\delta} |b|^{n/(n-m)}
\end{equation}
for any $\delta >0$ with the constant $C_\delta = (1 -(1+\delta)^{-(n-m)/m})^{-m/(n-m)}$. We shall prove the following.

\begin{lemma}\label{integrability}
For any function $u\in W^{m,n/m}(\H^n)$ and any $p >0$, we have
\[
\int_{\H^n} \Phi_{n,m}(p |u|^{n/(n-m)}) dV_g < +\infty.
\]
\end{lemma}

\begin{proof}
For $\epsilon  >0$, by a density argument, we can choose $v\in C_0^\infty(\H^n)$ in such a way that $\|\nabla^m(u-v)\|_{n/m} < \epsilon $. Let us divide $\H^n$ into two parts as follows
\begin{align*}
\Omega_1 &=\{x\, :\, |u(x)-v(x)| \leqslant 1\}, \quad
\Omega_2  =\{x\, :\, |u(x)-v(x)| > 1\}.
\end{align*}
\underline{On $\Omega_1$}, we have $|u| \leqslant 1+\max_{\H^n}|v|=:C_v$ then 
\[
\Phi_{n,m}(p |u|^{n/(n-m)}) \leqslant C(n,m,p,v)|u|^{n/m}
\]
for some constant $C(n,m,p,v)>0$ depending only on $n$, $m$, $p$, and $v$. Hence
\begin{align*}
\int_{\Omega_1} \Phi_{n,m}(p |u|^{n/(n-m)}) dV_g &\leqslant \int_{\{|u|\leqslant C_v\}}\Phi_{n,m}(p |u|^{n/(n-m)}) dV_g\\
&\leqslant C(n,m,p,v)\int_{\{|u|\leqslant C_v\}} |u|^{n/m} dV_g\\
&\leqslant C(n,m,p,v) C \|\nabla^m u\|_{n/m} < +\infty.
\end{align*}
(Here we have used Poincar\'e--Sobolev's inequality in $\H^n$; see \cite[Theorem 18]{FM2015}.) \underline{On $\Omega_2$}, we can estimate the integral as follows
\begin{align*}
\int_{\Omega_2}\Phi_{n,m}(p |u|^{n/(n-m)}) dV_g &\leqslant \int_{\Omega_2} \exp(p|u|^{n/(n-m)}) dV_g\\
&\leqslant \int_{\Omega_2} \exp(2p |u-v|^{n/(n-m)} +pC_1|v|^{n/(n-m)}) dV_g\\
&\leqslant C(n,m,p,v) \int_{\H^n} \Phi_{n,m}(2p |u-v|^{n/(n-m)}) dV_g.
\end{align*}
(In the preceding estimate, we have used the fact that $|u-v| \geqslant 1$ on $\Omega_2$, that $v$ is bounded, and the elementary inequality \eqref{eqELEMENTARY} with $\delta = 1$.) Choosing $\epsilon $ small enough such that $2p\epsilon ^{n/(n-m)} \leqslant \beta(n,m)$, then we have, by Adams' inequality \eqref{eq:FontanaMorpurgo2015}, that
\begin{align*}
\int_{\Omega_2} \Phi_{n,m}(2p |u-v|^{n/(n-m)}) dV_g\leqslant \int_{\H^n} \Phi_{n,m}\big(\beta(n,m) \big( |u-v|/\epsilon  \big)^{n/(n-m)}\big)dV_g < +\infty,
\end{align*}
since $\|\nabla^m(u-v)\|_{n/(n-m)} < \epsilon $. Therefore, we obtain 
\[
\int_{\Omega_2}\Phi_{n,m}(p |u|^{n/(n-m)}) dV_g < +\infty.
\]
Thus, we have just shown that
\[\begin{split}
\int_{\H^n}\Phi_{n,m}(p |u|^{n/(n-m)}) dV_g = & \int_{\Omega_1}\Phi_{n,m}(p |u|^{n/(n-m)}) dV_g\\
& +\int_{\Omega_2}\Phi_{n,m}(p |u|^{n/(n-m)}) dV_g < +\infty
\end{split}\]
as claimed.
\end{proof}

In the following result, we show that it is enough to prove Theorem \ref{Lionslemma} for compactly supported smooth functions.

\begin{proposition}\label{reducesmooth}
Let $\{u_j\}_j$ be the sequence given in the statement of Theorem \ref{Lionslemma}. Let $v_j\in C_0^\infty(\H^n)$ be such that $\|\nabla^m(v_j-u_j)\|_{n/m} < j^{-1}$ for any $j\in \N$. Then for any $p_1 \in (p, P_{n,m}(u))$ there exists some positive constant $C$ such that 
\[
\begin{split}
\sup_{j\in \N} \int_{\H^n} \Phi_{n,m}&(p\beta(n,m) |u_j|^{n/(n-m)}) dV_g \\
& \leqslant C \sup_{j\in \N} \int_{\H^n} \Phi_{n,m}(p_1\beta(n,m) |v_j|^{n/(n-m)}) dV_g + C.
\end{split}
\]
\end{proposition}

\begin{proof}
It is easy to see that for any $A > 0$, there is a constant $C(n,m,A)$ depending only on $n,m,A$ such that
\[
\Phi_{n,m}(t) \leqslant C(n,m,A) t^{(n-m)/m}
\] 
for any $t\leqslant A$. This implies the existence of some constant $C$ independent of $j$ such that
\begin{equation}\label{eq:mienkobichan}
\int_{\{|u_j| \leqslant 2\}} \Phi_{n,m}(p \beta(n,m) |u_j|^{n/(n-m)}) dV_g \leqslant C \int_{\H^n} |u_j|^{n/m} dV_g \leqslant C.
\end{equation}
(Note that we have used Poincar\'e's inequality once.) We divide the set $\{|u_j| > 2\}$ into two parts as follows
\begin{align*}
\Omega_{j,1} &=\{|u_j| >2\} \cap \{|u_j -v_j| \leqslant 1\}, \quad
\Omega_{j,2} = \{|u_j|  >2\} \cap \{|u_j -v_j| > 1\}.
\end{align*}
\underline{On $\Omega_{j,1}$} we have $|v_j| \geqslant |u_j| -|u_j-v_j| > 1$; hence 
\begin{equation}\label{eq:Omegaj1}
\begin{split}
\int_{\Omega_{j,1}} \Phi_{n,m} & (p \beta(n,m) |u_j|^{n/(n-m)}) dV_g \\
&\leqslant \int_{\Omega_{j,1}} \exp(p \beta(n,m) |u_j|^{n/(n-m)}) dV_g \\
&\leqslant \int_{\Omega_{j,1}} \exp(p \beta(n,m) (1+\delta) |v_j|^{n/(n-m)}+ p\beta(n,m) C_\delta) dV_g \\
&\leqslant C \int_{\H^n} \Phi_{n,m}(p_1 \beta(n,m) |v_j|^{n/(n-m)}) dV_g \\
&\leqslant C \sup_{j \in \N} \int_{\H^n} \Phi_{n,m}(p_1 \beta(n,m) |v_j|^{n/(n-m)}) dV_g,
\end{split}
\end{equation}
thanks to the fact that $|v_j| \geqslant 1 $ on $\Omega_{j,1}$ and by the choice $\delta = p_1/p -1$ in \eqref{eqELEMENTARY}. 

\noindent\underline{On $\Omega_{j,2}$}, we further split it into two smaller parts as follows
\begin{align*}
\Omega_{j,2}^1 &=\Omega_{j,2} \cap \{|v_j| < 1\},\quad
\Omega_{j,2}^2 = \Omega_{j,2} \cap \{|v_j| \geqslant 1\}.
\end{align*}
On $\Omega_{j,2}^1$ we can apply \eqref{eqELEMENTARY} with $\delta = 1$ to get
\begin{align*}
\int_{\Omega_{j,2}^1} \Phi_{n,m}  & (p \beta(n,m) |u_j|^{n/(n-m)}) dV_g \\
&\leqslant \int_{\Omega_{j,2}^1} \exp(p \beta(n,m) |u_j|^{n/(n-m)}) dV_g \\
&\leqslant \int_{\Omega_{j,2}^1} \exp(2p \beta(n,m) |u_j-v_j|^{n/(n-m)}+ p\beta(n,m) C_1) dV_g \\
&\leqslant C \int_{\H^n} \Phi_{n,m}(2p \beta(n,m)|u_j-v_j|^{n/(n-m)}) dV_g.
\end{align*}
Choose $J_0$ such that $J_0 \geqslant (2p)^{(n-m)/n}$. Then for any $j \geqslant J_0$ we can apply Adams' inequality \eqref{eq:FontanaMorpurgo2015} to get
\begin{align*}
\int_{\H^n} \Phi_{n,m} & (2p \beta(n,m)|u_j-v_j|^{n/(n-m)}) dV_g\\
 &\leqslant \int_{\H^n} \Phi_{n,m}\Big( \beta(n,m)\Big|\frac{u_j-v_j}{\|\nabla^m(u_j-v_j)\|_{n/m}}\Big|^{n/(n-m)}\Big) dV_g \leqslant C.
\end{align*}
Putting the above estimates together, we deduce that
\begin{equation}\label{eq:Omegaj21}
\sup_{j\in \N}\int_{\Omega_{j,2}^1} \Phi_{n,m}(p \beta(n,m) |u_j|^{n/(n-m)}) dV_g \leqslant C.
\end{equation}
On $\Omega_{j,2}^2$, from \eqref{eqELEMENTARY} we have 
\[
|u_j|^{n/(n-m)} \leqslant (1+\epsilon ) |v_j|^{n/(n-m)} + C_\epsilon  |u_j-v_j|^{n/(n-m)}
\]
with $\epsilon  = (p_1-p)/(2p)$. Denote $r = 2p_1/(p+p_1)$ and $r' = r/(r-1)$. Using H\"older's inequality, we obtain
\begin{align*}
\int_{\Omega_{j,2}^2} \Phi_{n,m} & (p \beta(n,m) |u_j|^{n/(n-m)}) dV_g \\
\leqslant &\int_{\Omega_{j,2}^2} \exp(p \beta(n,m) |u_j|^{n/(n-m)}) dV_g \\
\leqslant &\int_{\Omega_{j,2}^2} \exp \big((1+\epsilon )p \beta(n,m) |v_j|^{n/(n-m)}+ p\beta(n,m) C_\epsilon  |u_j-v_j|^{n/(n-m)} \big) dV_g \\
\leqslant &\Big(\int_{\Omega_{j,2}^2} \exp(p_1 \beta(n,m)|v_j|^{n/(n-m)}) dV_g\Big)^{(p_1+p)/(2p_1)} \\
& \times \Big(\int_{\Omega_{j,2}^2} \exp(r'pC_\epsilon  \beta(n,m) |u_j-v_j|^{n/(n-m)})dV_g\Big)^{1/r'}\\
\leqslant & C \Big(\int_{\H^n} \Phi_{n,m}(p_1 \beta(n,m)|v_j|^{n/(n-m)}) dV_g\Big)^{(p_1+p)/(2p_1)}\\
& \times \Big(\int_{\H^n} \Phi_{n,m}(r'pC_\epsilon  \beta(n,m) |u_j-v_j|^{n/(n-m)})dV_g\Big)^{1/r'}.
\end{align*}
Choosing $J_0$ such that $J_0 \geqslant (r'pC_\epsilon )^{(n-m)/n}$ and using Adams' inequality\eqref{eq:FontanaMorpurgo2015}, we arrive at
\[
\int_{\H^n} \Phi_{n,m}(r'pC_\epsilon  \beta(n,m) |u_j-v_j|^{n/(n-m)})dV_g \leqslant C
\]
for any $j \geqslant J_0$. Using Lemma \ref{integrability} and the inequality $p< p_1$ we obtain
\begin{equation}\label{eq:Omegaj22}
\begin{split}
\sup_{j\in \N} \int_{\Omega_{j,2}^2} \Phi_{n,m} & (p \beta(n,m) |u_j|^{n/(n-m)}) dV_g \\
&\leqslant C\sup_{j\in \N} \int_{\H^n} \Phi_{n,m}(p_1 \beta(n,m)|v_j|^{n/(n-m)}) dV_g.
\end{split}
\end{equation} 
Combining \eqref{eq:mienkobichan}, \eqref{eq:Omegaj1}, \eqref{eq:Omegaj21}, and \eqref{eq:Omegaj22} we obtain the desired result.
\end{proof}

\subsection{Proof of Theorem \ref{Lionslemma} for compactly supported smooth functions} 
\label{subsec-ProofForCompactlySupport}

We follow the argument given in \cite{CCH2013} by $\check{\rm C}$erny, Cianchi and Hencl. This method was used in \cite{doO2014b} to establish the concentration--compactness principle for the Moser--Trudinger inequality in whole space $\R^n$. Recently, it was used and developed in \cite{VHN2016} to prove the concentration--compactness principle for the sharp Adams--Moser--Trudinger inequality in $\R^n$ for any domain (bounded and unbounded). In the case of bounded domains, the result in \cite{VHN2016} covers the results in \cite{doO2014a} for the even order of gradient, and improves the results in \cite{doO2014a} for the odd order of gradient.

Let us go back to the proof of Theorem \ref{Lionslemma}. As usual, we argue by contradiction. Suppose that there exists a sequence $\{u_j\}_j \subset C_0^\infty(\H^n)$ such that:
\begin{itemize}
 \item $\|\nabla^m u_j\|_{n/m} \leqslant 1$, 
 \item $u_j$ converges weakly to a nonzero function $u$ in $W^{m,n/m}(\H^n)$, and 
 \item there exists a number $p \in (1,P_{n,m}(u))$ such that 
\begin{equation}\label{eq:contradictionassumption}
\lim_{j\to +\infty} \int_{\H^n} \Phi_{n,m}(p \beta(n,m) |u_j|^{n/(n-m)}) dV_g =+\infty.
\end{equation} 
\end{itemize}
Our aim is to look for a contradiction to \eqref{eq:contradictionassumption}. Using Rellich--Kondrachov's theorem, by passing to a subsequence if necessary, we can assume that:
\begin{itemize}
 \item $u_j$ converges almost everywhere to $u$ in $\H^n$, 
 \item $u_j$ converges to $u$ in $L^p_{\rm loc}(\H^n)$ for any $p < +\infty$ and additionally 
 \item $\Delta^{(m-1)/2} u_j $ converges almost everywhere to $\Delta^{(m-1)/2} u$ in $\H^n$ if $m$ is odd.
\end{itemize} 
We will need the following simple result.

\begin{lemma}\label{tocutintegral}
Let $u \in L^{n/m}(\H^n)$ be such that $\|u\|_{n/m} \leqslant \widehat C$. Then for any $R > 0$ and any $p>0$ there exists a constant $C$ depending only on $n$, $m$, $p$, $R$, and $\widehat C$ such that
\[
\int_{\overline{\H^n \setminus B(0,R)}} \Phi_{n,m}(p \beta(n,m) |u^\sharp|^{n/(n-m)}) dV_g \leqslant C(n,m,p,R, \widehat C).
\]
\end{lemma}

\begin{proof}
Clearly,
\[
\widehat C^{n/m} \geqslant  \int_{\H^n}|u|^{n/m} dV_g =\int_{\H^n}|u^\sharp|^{n/m} dV_g.
\]
Let $x\in \H^n$ be aribitrary but fixed. The monotonicity of $u^*$ implies that
\[
\int_{\H^n}|u^\sharp|^{n/m} dV_g 
\geqslant \int_{B(0, d(0,x) )\}} |u^\sharp|^{n/m} dV_g
\geqslant (u^\sharp(x))^{n/m} |B(0, d(0,x))|.
\]
Hence, we can estimate $u^\sharp$ from above as follows
\[
u^\sharp(x) \leqslant \frac {\widehat C}{|B(0, d(0,x))|^{m/n}}.
\]
Hence for any $y \in \H^n$ such that $d(0,y) \geqslant R$ we have $u^\sharp(y) \leqslant C(n,m,R, \widehat C)$ for some constant $C$ depending only on $n$, $m$, $R$, and $\widehat C$. By the definition of the function $\Phi_{n,m}$, it is easy to check that there exists a constant $C$ depending only on $n$, $m$, $p$, $R$, and $\widehat C$ such that 
\[
\Phi_{n,m} \big(p \beta(n,m) (u^\sharp(y))^{n/(n-m)} \big) \leqslant C(n,m,p,R, \widehat C) (u^\sharp(y))^{n/m}
\]
for any $d(0,y) \geqslant R$. This proves Lemma \ref{tocutintegral} because $u \in L^{n/m}(\H^n)$.
\end{proof}

We now continue to prove Theorem \ref{Lionslemma}. Thanks to $\|\nabla^m u_j\|_{n/m} \leqslant 1$, we can apply the Poincar\'e--Sobolev inequality to obtain $\|u_j\|_{n/m} \leqslant C$ for any $j$ for some constant $C>0$ independent of $j$; see \cite{FM2015}. Now we write
\begin{align*}
\int_{\H^n} \Phi_{n,m} \big(p \beta & (n,m) |u_j|^{n/(n-m)} \big) dV_g \\
= &\int_{\H^n} \Phi_{n,m} \big(p \beta(n,m) |u_j^\sharp|^{n/(n-m)} \big) dV_g\\
=&\Big( \int_{B(0,R)} + \int_{\H^n \setminus B(0,R)} \Big) \Phi_{n,m} \big(p \beta(n,m) |u_j^\sharp|^{n/(n-m)} \big) dV_g.
\end{align*}
Now Lemma \ref{tocutintegral} and our assumption \eqref{eq:contradictionassumption} imply that
\begin{equation}\label{eq:cutoffintegrala}
\lim_{j\to +\infty} \int_{B(0,R)} \Phi_{n,m} \big(p \beta(n,m) |u_j^\sharp|^{n/(n-m)} \big) dV_g = +\infty.
\end{equation}
Note that for $l < n/m -1$ by H\"older's inequality we have 
\[
\int_{B(0,R)} (u^\sharp)^{ln/(n-m)} dV_g \leqslant |B(0,R)|^{1 -ml/(n-m)} \Big(\int_{B(0,R)} (u^\sharp)^{n/m} dV_g\Big)^{ml/(n-m)} .
\]
This inequality and \eqref{eq:cutoffintegrala} imply
\begin{equation}\label{eq:cutoffintegral}
\begin{split}
\lim_{j\to +\infty} \int_0^{|B(0,R)|} & \exp \big(p \beta(n,m) |u_j^*(s)|^{n/(n-m)} \big) ds \\
=& \lim_{j\to +\infty} \int_{B(0,R)} \exp \big(p \beta(n,m) |u_j^\sharp|^{n/(n-m)} \big) dV_g =+\infty.
\end{split}
\end{equation}
There are two possible cases:

\smallskip\noindent\textbf{Case 1: Suppose that $m$ is even.} In this case, we can express $m=2k$ for some $k\geqslant 1$. Denote
\[
f_j = \Delta^k_g u_j, \quad f =\Delta^k_g u.
\] 
By passing to a subsequence if necessary, $f_j^*$ converges almost everywhere in $(0,+\infty)$ and converges weakly in $L^{n/m}(0,+\infty)$ to a function $g$ such that $ g\geqslant f^*$. It is evident that 
\begin{equation}\label{eq:claim}
\int_0^{+\infty} g(s)^{n/m} ds \leqslant 1.
\end{equation}
Then Proposition \ref{keyforLions} and \eqref{eq:boundforMfunction1} give
\[
u_j^*(t) \leqslant \frac{c_{n,k}}{(n\Omega_n^{1/n})^m} \frac n{n-m}\int_t^{|B(0,R)|} \frac{f_j^*(s)}{s^{1-m/n}} ds + C(n,m,R) 
\]
for all $0 < t\leqslant |B(0,R)|$. Here $c_{n,k}$ is the constant given in Proposition \ref{estimateforsolution1}. Define
\[
v_j(t) = \frac{c_{n,k}}{(n\Omega_n^{1/n})^m} \frac n{n-m}\int_t^{|B(0,R)|} \frac{f_j^*(s)}{s^{1- m/n}} ds
\]
for all $0 < t\leqslant |B(0,R)|$. Clearly, we have $v_j(|B(0,R)|) =0$. For any $\delta > 0$, we also have
\[
(u_j^*(t))^{n/(n-m)} \leqslant (1+\delta) v_j(t)^{n/(n-m)} + C_\delta C(n,m,R)^{n/(n-m)},
\]
Choose $\delta > 0$ small enough such that $q :=(1+\delta) p < P_{n,m}(u)$, then we conclude from \eqref{eq:cutoffintegral} that
\begin{equation}\label{eq:reducetovj}
\lim_{j\to +\infty} \int_0^{|B(0,R)|} \exp(q \beta(n,m) |v_j(t)|^{n/(n-m)}) ds = +\infty.
\end{equation}
From the definition of $v_j$, we have
\begin{equation}\label{eq:boundforvj}
v_j(t) \leqslant \Big(\frac 1{\beta(n,m)} \log \Big(\frac{|B(0,R)|}t\Big)\Big)^{(n-m)/n}.
\end{equation}
We claim that for any $r\in (q, P_{n,m}(u))$, any $j_0\in \N$, and any $s_0\in (0, |B(0,R)|)$ there exist $j\geqslant j_0$ and $s\in (0,s_0)$ such that 
\begin{equation}\label{eq:claimeven}
v_j(s) \geqslant \Big(\frac 1{r \beta(n,m)} \log \Big(\frac{|B(0,R)|}s\Big)\Big)^{(n-m)/n}.
\end{equation}
Indeed, if this were not true, then there would exist $r \in (q, P_{n,m}(u))$, $j_0\in \N$ and $s_0\in (0,|B(0,R)|)$ such that 
\[
v_j(s) \leqslant \Big(\frac 1{r \beta(n,m)} \log \Big(\frac{|B(0,R)|}s\Big)\Big)^{(n-m)/n}
\]
for all $j\geqslant j_0$ and all $s\in (0,s_0)$. This and \eqref{eq:boundforvj} imply that 
\begin{align*}
\int_0^{|B(0,R)|} & \exp(q \beta(n,m) |v_j(t)|^{n/(n-m)}) ds\\
=& \int_0^{s_0} \exp(q \beta(n,m) |v_j(t)|^{n/(n-m)}) ds  +\int_{s_0}^{|B(0,R)|} \exp(q \beta(n,m) |v_j(t)|^{n/(n-m)}) ds\\
\leqslant &\int_0^{s_0} \Big(\frac{|B(0,R)|} s\Big)^{q/r} ds + \int_{s_0}^{|B(0,R)|} \Big(\frac{|B(0,R)|} s\Big)^{q} ds\\
\leqslant &C(n,m,q,r,s_0,R)
\end{align*}
for any $j\geqslant j_0$. This contradicts \eqref{eq:reducetovj}; hence proves our claim \eqref{eq:claimeven}. Thus, up to a subsequence, we can assume that there exists a sequence $\{s_j\} \subset (0, |B(0,R)|)$ such that $s_j \leqslant 1/j$ and that
\begin{equation}\label{eq:consequenceofclaim}
v_j(s_j) \geqslant \Big(\frac 1{r \beta(n,m)} \log \Big(\frac{|B(0,R)|}{s_j}\Big)\Big)^{(n-m)/n}.
\end{equation}
Given $L > 0$, let us consider the truncation operators $T^L$ and $T_L$ acting on functions $v$ through
\[
\left\{
\begin{split}
T^L(v) &= \min\{|v|,L\} \sign(v),\\
T_L(v) &= v-T^L(v).
\end{split}
\right.
\]
It is easy to see that $T^L(f_j^*)$ and $T_L(f_j^*)$ converge almost everywhere to $T^L(g)$ and $T_L(g)$ in $(0,+\infty)$, respectively. Since $\lim_{j\to +\infty} v_j(s_j) = +\infty$, given any $L > 0$, after passing to a subsequence if necessary, we can assume that $v_j(s_j) > L$ for any $j$. Then there exists $r_j\in (s_j,|B(0,R)|)$ such that $v_j(r_j) =L$. On the other hand, from the definition of $v_j$ and the monotonicity of $f_j^*$, we have
\[
v_j(s_j) \leqslant \frac{c_{n,k+1}}{(n\Omega_n^{1/n})^m} f_j^*(s_j) |B(0,R)|^{m/n},
\]
hence $\lim_{j\to +\infty} f_j^*(s_j) =\infty.$ Therefore, by passing again to a subsequence if necessary, we assume that $f_j^*(s_j) > L$ for any $j$, hence there exists $t_j\in (0, +\infty)$ such that $f_j^*(t_j) = L$ and $f^*_j(s) < L$ for any $ s> t_j$. Denote $a_j = \min\{t_j,r_j\}$. We then have 
\begin{align*}
v_j(s_j) -L  =& \frac{c_{n,k}}{(n\Omega_n^{1/n})^m}\frac{n}{n-m}\int_{s_j}^{r_j} \frac{f_j^*(s)}{s^{1-m/n}} ds\\
 \leqslant& \frac{c_{n,k}}{(n\Omega_n^{1/n})^m}\frac{n}{n-m}\int_{s_j}^{a_j} \frac{f_j^*(s)-L}{s^{1- m/n}} ds + \frac{c_{n,k}}{(n\Omega_n^{1/n})^m}\frac{n}{n-m}\int_{s_j}^{r_j} \frac{L}{s^{1-m/n}} ds\\
\leqslant& \Big(\int_{s_j}^{a_j} (f_j^*(s)-L)^{n/m}ds\Big)^{m/n} \Big(\frac 1{\beta(n,m)} \log \frac{a_j}{s_j}\Big)^{\frac{n-m}n} + \frac{c_{n,k+1}}{(n\Omega_n^{1/n})^m}L r_j^{m/n}\\
\leqslant& \Big(\int_0^{+\infty} (T_L(f_j^*))^{n/m}ds\Big)^{m/n} \Big(\frac 1{\beta(n,m)} \log \frac{|B(0,R)|}{s_j}\Big)^{\frac{n-m}n} \\
&+ \frac{c_{n,k+1}}{(n\Omega_n^{1/n})^m}L |B(0,R)|^{m/n}.
\end{align*}
The latter estimate and \eqref{eq:consequenceofclaim} imply that
\[
r^{-(n-m)/n} \leqslant \Big(\int_{0}^{+\infty} (T_L(f_j^*))^{n/m} ds\Big)^{m/n}
\]
for $j$ large enough, equivalently, this is
\[
r^{-(n-m)/m} \leqslant \int_{0}^{+\infty} (T_L(f_j^*))^{n/m} ds.
\]
Hence, for $j$ large enough
\[
1-r^{-(n-m)/m} \geqslant \int_{0}^{+\infty} \big[(f_j^*)^{n/m} -(T_L(f_j^*)]^{n/m} \big)ds.
\]
Thanks to \eqref{eq:claim}, by letting $j \nearrow + \infty$ and using Fatou's lemma we get
\begin{equation}\label{eq:welldoneeven}
1-r^{-(n-m)/m} \geqslant \int_{0}^{+\infty} \big[ g^{n/m} -(T_L(g))^{n/m} \big] ds.
\end{equation}
Now we try to obtain a contradiction by using \eqref{eq:welldoneeven}.

\smallskip\noindent\underline{Case 1.1}: Suppose $\|f\|_{n/m} < 1$. Since 
\[
\int_0^{+\infty} g^{n/m} ds \geqslant \int_0^{+\infty} (f^*)^{n/m}ds =\|f\|_{n/m}^{n/m}
\]
and 
\[
\lim_{L\to +\infty} \int_0^{+\infty} (T_L(g))^{n/m} ds =0,
\]
we can choose some $L > 0$ such that 
\begin{equation}\label{eq:Llargereven}
\frac{1 -\|f\|_{n/m}^{n/m}}{1 -\int_0^{+\infty}(g^{n/m} -(T_L(g))^{n/m}) ds} > \Big(\frac{r}{P_{n,2k}(u)}\Big)^{(n-m)/m}.
\end{equation}
Fix such a number $L$, it follows from \eqref{eq:welldoneeven} and \eqref{eq:Llargereven} that
\begin{align*}
r \geqslant & \Big(1 -\int_0^{+\infty}(g^{n/m} -(T_L(g))^{n/m}) ds\Big)^{-m/(n-m)} \\
 > & \frac r{P_{n,2k}} (1 -\|f\|_{n/m}^{n/m})^{-m/(n-m)} =r,
\end{align*}
which is a contradiction.

\smallskip\noindent\underline{Case 1.2}: Suppose $\|f\|_{n/m} =1$. Then from \eqref{eq:claim} we must have $\int_0^{+\infty} g^{n/m} ds =1$. Then we can choose some large $L > 0$ such that 
\[
\int_0^{+\infty}(g^{n/m} -(T_L(g))^{n/m}) ds > 1 -\frac12 \Big(\frac 1r\Big)^{(n-m)/m}.
\]
Fix such $L$, then we obtain a contradiction since by \eqref{eq:welldoneeven} we have
\[
1-r^{-(n-m)/m} \geqslant \int_0^{+\infty} \big( g^{n/m} -(T_L(g))^{n/m} \big)ds > 1 -\frac12 \Big(\frac 1r\Big)^{(n-m)/m}.
\]
This finishes our proof in case that $m$ is even.

\smallskip\noindent\textbf{Case 2: Suppose that $m$ is odd.} Since the case $m=1$ was proved in \cite{Kar2015}. Using the same argument in \cite{doO2014b} gives another proof of this case. Hence it suffices to consider $m\geqslant 3$. In this scenario, we write $m =2k+1$ for some $k\geqslant 1$. Denote
\[
f_j =\Delta^k_g u_j, \quad f= \Delta_g^k u.
\] 
Using Sobolev's inequality we have $\|f_j\|_{n/(2k)} \leqslant C$. Proposition \ref{keyforLions} and \eqref{eq:boundforMfunction1} give 
\[
u_j^*(t_1) -u_j^*(t_2) \leqslant \frac{c_{n,k}}{(n\Omega_n^{1/n})^{2k}} \frac{n}{n-2k} \int_{t_1}^{t_2} \frac{f_j^*(s)}{s^{1-2k/n}} ds + C(n,k),
\]
where $c_{n,k}$ is again the constant given in Proposition \ref{estimateforsolution1}. Since $\|\nabla f_j \|_{n/m} \geqslant \|\nabla f_j^\sharp\|_{n/m}$ and as in Case 2 in the proof of Theorem \ref{Tarsihyperbolic} in Section \ref{sec-AdamsTypeWithBoudary}, we know that
\begin{equation}\label{eq:chandaoham}
\begin{split}
1\geqslant & \int_{\H^n} |\nabla_g f^\sharp|^{n/m} dV_g
 = (n\Omega_n)^{n/m} \int_0^{+\infty} |(f^*)'(s)|^{n/m} (\sinh F(s) )^{n(n-1)/m} ds.
\end{split}
\end{equation}
Using integration by parts we obtain
\begin{align*}
\int_{t_1}^{t_2} \frac{f_j^*(s)}{s^{1-2k/n}} ds & = \frac n{2k} \int_{t_1}^{t_2} (-f_j^*)'(s) s^{2k/n} ds + \frac n{2k} t_2^{2k/n} f_j^*(t_2) - \frac n{2k} t_1^{2k/n} f_j^*(t_1).
\end{align*}
From \eqref{eq:chandaoham} and the fact that $\sinh F(s) \geqslant (s/\Omega_n)^{1/n}$ we easily deduce that $t^{2k/n} f_j^*(t) \leqslant C$ for some $C$ independent of $j$. Consequently, we obtain
\begin{equation}\label{eq:keyforLionsodd}
u_j^*(t_1) -u_j^*(t_2) \leqslant \frac{c_{n,k+1}}{(n\Omega_n^{1/n})^{2k}} \int_{t_1}^{t_2} (-f_j^*)'(s)s^{2k/n} ds + C(n,k),
\end{equation}
for some $C(n,k)$ depending only on $n$ and $k$. Note that \eqref{eq:keyforLionsodd} plays the same role as \eqref{eq:keyforLions} in our proof below when $m$ is odd. Our proof proceeds along the same line as in the case when $m$ is even; hence we limit ourselves to sketching the proof. Define
\[
v_j(t) = \frac{c_{n,k+1}}{(n\Omega_n^{1/n})^{2k}} \int_{t}^{|B(0,R)|} (-f_j^*)'(s)s^{2k/n} ds
\]
for $t\in (0,|B(0,R)|)$. Then for $q \in (p, P_{n,m}(u))$ we have
\begin{equation}\label{eq:reducetovjodd}
\lim_{j\to +\infty} \int_0^{|B(0,R)|} \exp(q \beta(n,m) |v_j(t)|^{n/(n-m)}) ds = +\infty.
\end{equation}
For any $r\in (q, P_{n,m}(u))$, for any $j_0\in \N$, and any $s_0\in (0, |B(0,R)|)$ we claim that there exist $j\geqslant j_0$ and $s\in (0,s_0)$ such that 
\begin{equation}\label{eq:claimodd}
v_j(s) \geqslant \Big(\frac 1{r \beta(n,2k)} \log \big(\frac{|B(0,R)|}s\big)\Big)^{(n-2k)/n}.
\end{equation}
Indeed, if this does not true from \eqref{eq:reducetovjodd} we will obtain a contradiction as in the case $m$ even since
\[
v_j(t) \leqslant \Big(\frac 1{\beta(n,2k+1)} \log \big(\frac{|B(0,R)|}t\big)\Big)^{(n-2k-1)/n}
\]
for all $t\in (0,|B(0,R)|)$. Hence, \eqref{eq:claimodd} holds. In particular, up to a subsequence, there exists a sequence $\{s_j\} \subset (0, |B(0,R)|)$ such that $ s_j \leqslant 1/j$ and that
\begin{equation}\label{eq:consequenceofclaimodd}
v_j(s_j) \geqslant \Big(\frac 1{r \beta(n,2k+1)} \log \big(\frac{|B(0,R)|}{s_j}\big)\Big)^{(n-2k-1)/n}.
\end{equation}
Since
\[
v_j(s_j) \leqslant \frac{c_{n,k+1}}{(n\Omega_n^{1/n})^{2k}} |B(0,R)|^{2k/n} \big[ f_j^*(s_j) -f_j^*(|B(0,R)|) \big],
\]
we conclude that
\[
\lim_{j\to +\infty}v_j(s_j) = \lim_{j\to +\infty} f_j^*(s_j) = +\infty.
\] 
Therefore, given $L > 0$, by passing to a subsequence, we can assume that $v_j(s_j) > L$ and $f_j^*(s_j) > L$. Hence there exists $r_j \in (s_j, |B(0,R)|)$ and $t_j\in (s_j, +\infty)$ such that $v_j(r_j) = L$ and $f_j^*(t_j) =L$. Denote $a_j =\min\{t_j, r_j\}$, we have
\begin{align*}
v_j(s_j) -L = &\frac{c_{n,k+1}}{(n\Omega_n^{1/n})^{m-1}}\int_{s_j}^{r_j} (-f_j^*)'(s) s^{(m-1)/n} ds\\
=&\frac{c_{n,k+1}}{(n\Omega_n^{1/n})^{m-1}} \Big( \int_{s_j}^{a_j}  + \int_{a_j}^{r_j} \Big) (-f_j^*)'(s) s^{(m-1)/n} ds\\
\leqslant &\frac{c_{n,k+1}}{(n\Omega_n^{1/n})^m} \Big((n\Omega_n)^{n/m} \int_{s_j}^{a_j} |(f_j^*)'(s)|^{n/m}  (\sinh F(s) )^{n(n-1)/m} ds \Big)^{m/n} \\
& \times \Big(\int_{r_j}^{a_j} (\sinh F(s))^{-n(n-1)/(n-m)} s^{(m-1)/(n-m)}ds\Big)^{(n-m)/n} \\
&+ \frac{c_{n,k+1}}{(n\Omega_n^{1/n})^{m-1}} a_j^{(m-1)/n} \big[ f_j^*(a_j) -f_j^*(r_j) \big] \\
\leqslant &\|T_L(f_j^\sharp)\|_{n/m} \Big(\frac1{\beta(n,m)} \log \big(\frac{|B(0,R)|}{s_j}\big)\Big)^{(n-m)/n} \\
&+ \frac{c_{n,k+1}}{(n\Omega_n^{1/n})^{m-1}} |B(0,R)|^{(m-1)/n} L.
\end{align*}
Here we have used the estimate $\sinh F(s) \geqslant (s/\Omega_n)^{1/n}$ and the facts that if $t_j < r_j$ then $f_j^*(a_j) -f_j^*(r_j) \leqslant L$ while if $t_j\geqslant r_j$ then $f_j^*(a_j) -f_j^*(r_j) =0$. Hence, for $j$ large enough we obtain from \eqref{eq:consequenceofclaimodd} the following
\[
r^{-(n-m)/m} \leqslant \int_{\H^n} |\nabla_g T_L(f_j^\sharp)|^{n/m} dV_g.
\]
Note that $T_L(f_j^\sharp) = (T_L(f_j))^{\sharp}$; thus
\[
r^{-(n-2k-1)/(2k+1)} \leqslant \int_{\H^n} |\nabla_g T_L(f_j^\sharp)|^{n/(2k+1)} dV_g\leqslant \int_{\H^n} |\nabla_g T_L(f_j)|^{n/(2k+1)} dV_g.
\] 
Notice that
\[
\int_{\H^n} |\nabla_g T_L(f_j)|^{n/m} dV_g + \int_{\H^n} |\nabla_g T^L(f_j)|^{n/m}dV_g = \int_{\H^n} |\nabla_g f_j|^{n/m}dV_g.
\]
Then, we have for $j$ large enough
\[
1-r^{-(n-m)/m} \geqslant \int_{\H^n} |\nabla_g T^L(f_j)|^{n/m} dV_g.
\]
We have that $T^L(f_j)$ converges almost everywhere to $T^L(f)$ on $\H^n$. Moreover, $\{T^L(f_j)\}_j$ is bounded sequence in $W^{1,n/m}(\H^n)$, by passing to a subsequence if necessary, we assume that 
\begin{itemize}
 \item $T^L(f_j)$ converges weakly to a function $g$ in $W^{1,n/m}(\H^n)$ and 
 \item $T^L(f_j)$ converges to $g$ in $L^p_{\rm loc}(\H^n)$ for any $p <n/(2k)$ by the Rellich--Kondrachov theorem. 
\end{itemize}
This shows that $g = T^L(f)$, hence by the weak lower semi-continuity of the $L^{n/m}$-norm of gradient, we have
\begin{equation}\label{eq:welldoneodd}
1-r^{-(n-m)/m} \geqslant \int_{\H^n} |\nabla_g T^L(f)|^{n/m} dV_g,
\end{equation}
which is similar to \eqref{eq:welldoneeven}.

\smallskip\noindent\underline{Case 2.1}: Suppose $\|\nabla_g f\|_{n/m} < 1$. Since 
\[
\lim_{L\to +\infty} \int_{\H^n} |\nabla_g T^L(f)|^{n/m} dV_g = \int_{\H^n} |\nabla_g f|^{n/m} dV_g,
\]
we can choose some large $L > 0$ such that 
\begin{equation}\label{eq:Llargeodd}
\frac{1 -\|\nabla_g f\|_{n/m}^{n/m}} {1 -\|\nabla_g T^L(f)\|_{n/m}^{n/m}} > \Big(\frac{r}{P_{n,2k+1}(u)}\Big)^{(n-m)/m}.
\end{equation}
Fix such $L >0$. Combining \eqref{eq:welldoneodd} and \eqref{eq:Llargeodd} gives
\begin{align*}
r \geqslant & (1 -\|\nabla_g T^L(f)\|_{n/m}^{n/m})^{- m/(n-m)}
 > \frac r{P_{n,m}(u)} (1 -\|\nabla_g f\|_{n/m}^{n/m})^{-m/(n-m)}=r,
\end{align*}
a contradiction.

\smallskip\noindent\underline{Case 2.2}: Suppose $\|\nabla_g f\|_{n/m} = 1$. Then we can choose $L > 0$ such that 
\[
\|\nabla_g T^L(f)\|_{n/m}^{n/m} > 1 -\frac 12 \Big(\frac1 r\Big)^{(n-m)/m}.
\]
Fix such $L >0$ and by using \eqref{eq:welldoneodd} we obtain a contradiction because
\[
1-r^{-(n-m)/m} \geqslant \|\nabla_g T^L(f)\|_{n/m}^{n/m} > 1 -\frac 12 \Big(\frac1 r\Big)^{(n-m)/m}.
\]
This finishes our proof when $m$ is odd.


\subsection{The sharpness of (\ref{eq:Lions})}

It remains to check the sharpness of the exponent $P_{n,m}(u)$ in Theorem \ref{Lionslemma}. To this purpose, we will show that for any $\alpha \in (0,1)$, there exists a sequence $\{u_j\}_j\subset W^{m, n/m}(\H^n)$ and $u\in W^{m,n/m}(\H^n)$ such that 
\begin{itemize}
 \item $\|\nabla_g^m u_j\|_{n/m} =1$, $\|\nabla_g^m u\|_{n/m} =\alpha $, 
 \item $u_j \rightharpoonup u$ in $W^{m,n/m}_0(\H^n)$, and 
 \item $u_j\to u$ almost everywhere in $\H^n$ 
\end{itemize}
such that
\[
\lim_{j\to+\infty}\int_{\H^n} \Phi_{n,m} \big(\beta(n,m) (1-\alpha^{n/m})^{-m/(n-m)} |u_j|^{n/(n-m)} \big) dx = +\infty.
\]
For $j\geqslant 2$, we define
\[
v_j(x) = 
\begin{cases}
\Big(\dfrac {\log j}{\beta(n,m)}\Big)^{1-m/n} + \dfrac{n\beta(n,m)^{m/n -1}}{2 (\log j)^{m/n}} \displaystyle\sum_{l=1}^{m-1} \frac {(1 -j^{2/n} |x|^2)^l}l &\mbox{if $0\leqslant |x|\leqslant j^{-1/n}$},\\
-n \beta(n,m)^{m/n -1} (\log j)^{-m/n} \log |x| &\mbox{if $j^{-1/n} \leqslant |x| < 1$},\\
\xi_j(x) &\mbox{if $1\leqslant |x|\leqslant 2$},
\end{cases}
\]
where $\xi_j\in C_0^\infty(B_2)$ are radial functions which are chosen such that $\xi_j =0$ on $\partial B_1$ and $\partial B_2$, and for $l =1,2, ... ,k-1$
\[
\frac{\partial^l \xi_j}{\partial r^l} \Big{|}_{\partial B_1} = (-1)^l (l-1)!\beta(n,m)^{m/n -1} (\log j)^{-m/n}, \quad \quad \frac{\partial^l \xi_j}{\partial r^l} \Big{|}_{\partial B_2} =0,
\]
and $\xi_j$, $|\nabla^l \xi_j|$ and $|\nabla^k \xi_j|$ are all $O((\log j)^{-m/n})$. The choice of these functions is inspired from \cite[Section 3]{Zhao2013}.

Consider the function $w_j(x) =v_j(3x)$ where $x\in \H^n$. Clearly, $w_j \in W^{m,n/m}(\H^n)$ with support in $B_{2/3}$. An easy computation shows that 
\[
1 \leqslant \int_{\H^n} |\nabla^m v_j(x)|^{n/m} dx \leqslant 1 + O((\log j)^{-1});
\]
hence
\begin{equation*}\label{eq:wdoublej}
1 -\frac{c}{\log j} \leqslant \|\nabla_g^m w_j\|_{n/m}^{n/m} \leqslant 1 + \frac{C}{\log j},
\end{equation*}
for some positive constants $c$ and $C$ independent of $j$. Setting
\[
\widetilde{w}_j = w_j /\|\nabla_g^m w_j\|_{ n/m},
\] 
we then have the following claims
\begin{itemize}
 \item $\widetilde{w}_j \rightharpoonup 0$ weakly in $W^{m,n/m}(\H^n)$ and
 \item $\widetilde{w}_j \rightharpoonup 0$ almost everywhere in $\H^n$. 
\end{itemize}
Taking a function $v\in C_0^\infty(B_1)$ in such a way that $v$ is constant in $B_{2/3}$ and $\|\nabla_g^m v\|_{n/m} = \alpha $. Then we define 
\[
u_j = v + (1-\alpha^{n/m})^{m/n} \widetilde{w}_j.
\] 
Clearly, $u_j\in W_0^{m,n/m}(\Omega)$ and $\|\nabla_g^m u_j\|_{n/m} = 1$ for all $j\geqslant 2$ since the supports of $\nabla_g^m v$ and $\nabla_g^m \widetilde{w}_j$ are disjoint and $u_j \rightharpoonup v$ in $W^{m,n/m}(\H^n)$. Replacing $v$ by $-v$ if necessary, we can assume that $v\geqslant A$ on $B_{2/3}$ for some $A >0$. Then we can estimate
\begin{align*}
\int_{\H^n} &\Phi_{n,m} (\beta(n,m) (1-\alpha^{n/m})^{-m/(n-m)} |u_j|^{n/(n-m)}) dV_g\\
&\geqslant \int_{|x| \leqslant j^{-1/n}} \Phi_{n,m}
\left( \begin{gathered}
\frac{\beta(n,m)}{ (1-\alpha^{n/m})^{m/(n-m)}} \times \hfill \\
\Big[A +\frac{(1-\alpha^{n/m})^{m/n}}{(1 + C/\log j)^{m/n}}\Big(\frac {\log j}{\beta(n,m)}\Big)^{\frac{n-m}{n}}\Big]^{\frac{n}{n-m}} \\ 
\end{gathered}  \right) dV_g\\
&\geqslant C'\omega_n \exp\Big(\Big[C + \frac{(\log j)^{(n-m)/n}}{(1 + C/\log j)^{m/n}}\Big]^{\frac{n}{n-m}} -\log j\Big), 
\end{align*} 
for some positive constants $C$ and $C'$ independent of $j$. It is easy to see that there exist a constant $C_1 \in (0, C)$ and some $j_0$ such that
\begin{align*}
\exp\Big(\Big[C +& \frac{(\log j)^{(n-m)/n}}{(1 + C/\log j)^{m/n}}\Big]^{n/(n-m)} -\log j\Big) \\
&\geqslant \exp\Big(\Big[C_1 + (\log j)^{(n-m)/n} \Big]^{n/(n-m)} -\log j\Big)
\end{align*} 
 for any $j\geqslant j_0$. Putting these estimates together, we deduce that
\[\begin{split}
\liminf_{j\to +\infty} \int_{\H^n} \exp \big( & \beta(n,m) (1-\alpha^{n/m})^{-m/(n-m)} |u_j|^{n/(n-m)} \big) dV_g \\
 \geqslant & \lim_{j\to +\infty} \exp\Big(\big(C_1 + (\log j)^{(n-m)/n}\big)^{n/(n-m)} -\log j\Big) = +\infty.
\end{split}\]
This proves the sharpness of \eqref{eq:Lions} as claimed.


\section*{Acknowledgments}

We would like to thank an anonymous referee for pointing out some inaccuracies in the original version of this manuscript. V.H.N would like to acknowledge the support of the CIMI postdoctoral research fellowship. Q.A.N would like to thank the Vietnam Institute for Advanced Study in Mathematics (VIASM) for excellent working environment where this work was initiated and carried. The research of Q.A.N is funded by the Vietnam National Foundation for Science and Technology Development (NAFOSTED) under grant number 101.02-2016.02.


\addtocontents{toc}{\protect\setcounter{tocdepth}{0}}

\section*{ORCID iDs}

Qu\cfac oc Anh Ng\^o: 0000-0002-3550-9689

\addtocontents{toc}{\protect\setcounter{tocdepth}{2}}


\end{document}